\documentclass[12pt]{article}
\usepackage[letterpaper, margin=1in]{geometry}

\usepackage{amsmath, amssymb, amsthm, color, mathtools}
\usepackage{caption, graphicx}
\graphicspath{ {figures/} }
\usepackage{array}
\usepackage{subcaption}
\usepackage{enumitem}
\newtheorem{theorem}{Theorem}
\newtheorem{corollary}{Corollary}[theorem]
\newtheorem{proposition}{Proposition}
\usepackage[longnamesfirst, authoryear]{natbib}
\usepackage{authblk}
\newtheorem{definition}{Definition}
\newtheorem{lemma}[theorem]{Lemma}
\newtheorem{example}{Example}
\newtheorem{remark}{Remark}

\DeclareMathOperator\supp{supp}
\usepackage{changepage}
\usepackage{hyperref}
\usepackage{setspace}
\hypersetup{hidelinks,colorlinks=true,linkcolor=blue,citecolor=blue,linktocpage=true}

\usepackage[title, titletoc]{appendix}
\providecommand{\keywords}[1]
{
  \small{\textit{Keywords---} #1}
}
\title{\Large Generalized van Trees inequality:\\
Local minimax bounds for non-differentiable functionals and \\irregular statistical models}
\author{Kenta Takatsu}
\author{Arun Kumar Kuchibhotla}
\affil{Department of Statistics and Data Science, Carnegie Mellon University}
\date{}

\DeclareMathOperator{\E}{\mathbb{E}}

\usepackage[longnamesfirst, authoryear]{natbib}
\begin{document}
\maketitle

\begin{abstract}
In a decision-theoretic framework, the minimax lower bound provides the worst-case performance of estimators relative to a given class of statistical models. For parametric and semiparametric models, the H\'{a}jek--Le Cam local asymptotic minimax (LAM) theorem provides the sharp local asymptotic lower bound. Despite its relative generality, this result comes with limitations as it only applies to the estimation of differentiable functionals under regular statistical models. On the other hand, minimax lower bound techniques such as Fano's or Assoud's are applicable in more general settings but are not sharp enough to imply the LAM theorem. To address this gap, we provide new non-asymptotic minimax lower bounds under minimal regularity assumptions, which imply sharp asymptotic constants. The proposed lower bounds do not require the differentiability of functionals or regularity of statistical models, extending the efficiency theory to broader situations where standard results fail. 
The use of the new lower bounds is illustrated through the local minimax lower bound constants for estimating the density at a point and directionally differentiable parameters.
\end{abstract}

\keywords{van Trees inequality, Hammersley–Chapman–Robbins bound, Local asymptotic minimax theorem, Local minimaxity, Efficiency theory, Non-differentiable functional estimation, Irregular statistical models}

\tableofcontents

\section{Introduction} 
In a decision-theoretic framework, the optimality of an estimation procedure is often motivated by the attainability to minimax lower bound. Specifically, a proposed estimator is considered rate-optimal when its convergence rate matches that of the corresponding minimax lower bound. A more refined assessment examines whether an estimator is \textit{asymptotically efficient}, meaning the risk of the estimator coincides with the minimax lower bound including the constant. If this is the case, the estimation procedure becomes no longer improvable at least asymptotically. Hence, the construction of precise minimax lower bounds has been a crucial step for evaluating many statistical procedures.

The optimality of functional estimation has been extensively studied in parametric and nonparametric models~\citep{ibragimov1981statistical, bickel1993efficient,lehmann2006theory}. Semiparametric efficiency theory, in particular, is considered a cornerstone for understanding the asymptotic optimality of functional estimation. The convolution theorem and the local asymptotic minimax (LAM) theorem are two fundamental results in the development of semiparametric efficiency~\citep{hajek1970characterization, hajek1972local, le1972limits}, which provide precise asymptotic constants; however, these theorems only apply to differentiable functionals and require certain regularity conditions on the underlying statistical models. As optimal estimation of non-differentiable functionals has gained growing interest, extending the classical efficiency theory to broader settings is an imminent task. 

This manuscript aims to strengthen the LAM theorem by developing new minimax lower bounds under minimal regularity assumptions imposed on functionals, estimators, and underlying statistical models. We consider a (nonparametric) functional $\psi: \mathcal{P} \mapsto \mathbb{R}^k$, where $\mathcal{P}$ is a collection of probability measures containing the distribution from which the observation $X$ is drawn. The unknown data-generating distribution is denoted by $P_{\theta_0}$ and a local statistical model containing $P_{\theta_0}$ is given by $\{P_\theta : \theta \in \Theta_0 \subseteq \Theta\}$, satisfying $\theta_0 \in  \Theta_0$. Here, the indexing set $\Theta_0$ is a subset of a metric space, which itself can be the space of probability measures. The primary focus of this manuscript is on the following general minimax risk for an arbitrary parameter set $\Theta_0 \subseteq \Theta$:
\begin{equation}\label{eq:local-minimax-risk}
    \inf_{T}\,  \sup_{\theta \in \Theta_0}\,  \mathbb{E}_{P_\theta}\|T(X)-\psi(P_\theta)\|^2
\end{equation}
where $\|\cdot\| : \mathbb{R}^k \mapsto \mathbb{R}_+$ for $\mathbb{R}_+ := [0, \infty)$ is any vector norm, and the infimum is taken over measurable functions. We also denote by $\mathbb{E}_{P_\theta}$ the expectation under $P_{\theta}$. It should be noted that this framework is not confined to a parametric model, and it can be extended to a nonparametric model using the standard machinery from semiparametric statistics \citep{bickel1993efficient, van2002semiparametric}.

\subsection{Summary of main results}
The main results are summarized as follows. On a conceptual level, we demonstrate that minimax risk for estimating non-differentiable functionals is characterized by the interplay between two non-negative terms in the proceeding display:
\begin{align*}
    \inf_{T}\,  \sup_{\theta \in \Theta_0}\,  \mathbb{E}_{P_\theta}\|T(X)-\psi(P_\theta)\|^2 \ge \sup_{\phi}\, \left[\left\{\textrm{S.Eff}(\phi; \Theta_0)\right\}^{1/2}-\left\{\textrm{Approx.Bias}(\psi, \phi; \Theta_0)\right\}^{1/2}\right]^2_+
\end{align*}
where $[t]_+ := \max(t, 0)$ and $\phi$ is any functional that approximates the original functional $\psi$. The first term, we call it \textit{surrogate efficiency}, is an efficiency bound for estimating the functional $\phi$. When $\phi$ provides sufficient structures such as smoothness, this term is well-studied in the literature for both parametric and nonparametric models. The second term, we call it \textit{approximation bias}, quantifies the deviation of the introduced functional $\phi$ from the original functional $\psi$. The overall minimax risk is then obtained by optimizing the choice of $\phi$. Such an interplay between surrogate efficiency and approximation bias resembles the classical bias-variance decomposition of an estimation error. Hence, the result of this manuscript can be seen as the minimax lower bound analog to the bias-variance trade-off.

We provide two concrete instances of the above idea. First, Theorem~\ref{cor:vt-approximation} shows that the minimax risk admits the lower bound based on the projection onto the set $\Phi_{\textrm{ac}}$ of all absolutely continuous functions of $\theta\in\Theta_0$. Let $\psi:\Theta_0 \mapsto \mathbb{R}$ be a possibly non-differentiable functional of interest\footnote{The formal result in Section~\ref{sec:vt-hellinger} holds for any vector-valued functional $\psi$.} and $\phi\in\Phi_{\textrm{ac}}$ with almost-everywhere derivative $\nabla \phi$. Assume that the statistical model $\{P_{\theta}:\,\theta\in\Theta_0\}$ is regular such that the Fisher information $\mathcal{I}(\theta)$ at $P_{\theta}$ is well-defined (see definition in \eqref{eq:hellinger-fisher-info}). Then, for all ``nice'' priors $Q$ on $\Theta_0$ (see Definition~\ref{as:regularprior-vt}) with the Fisher information $\mathcal{I}(Q)$ (see \eqref{eq:Fisher-information-prior}), it follows that:
\begin{align*}
& \inf_T\, \sup_{\theta\in\Theta_0}\,\mathbb{E}_{P_\theta}|T(X)-\psi(\theta)|^2 \ge \left[\Gamma^{1/2}_{Q, \phi} - \left(\int_{\Theta_0}|\psi(\theta)-\phi(\theta)|^2dQ(\theta)\right)^{1/2}\right]_+^2
\end{align*}
where 
\begin{align}
	\Gamma_{Q, \phi}:=\left(\int_{\Theta_0} \nabla \phi(\theta)\, dQ(\theta)\right)^{\top}\left(\mathcal{I}(Q) + \int_{\Theta_0} \mathcal{I}(\theta)\, dQ(\theta)\right)^{-1}\left(\int_{\Theta_0}  \nabla \phi(\theta)\, dQ(\theta)\right)\nonumber.
\end{align}
The lower bound holds for any choice of $\phi \in \Phi_{\textrm{ac}}$ and any ``nice'' prior $Q$, allowing for further optimization over them. Here, the surrogate efficiency is obtained by the \textit{van Trees inequality} \citep{gill1995applications, gassiat2013revisiting} for smooth functional $\phi$. 

Second, we consider approximating $\psi(t)$ with $\psi(t-h)$ for $h \in \Theta$. Such a local perturbation is well-defined, as long as $t, t-h \in \Theta_0$, even when the functional $\psi$ is non-differentiable. Then for any probability measure $Q$ over $\mathbb{R}^d$, Theorem~\ref{thm:crvtBound_hellinger} states for any $\psi:\mathbb{R}^d\mapsto\mathbb{R}^k$, that  
\begin{align*}
    &\inf_T\, \sup_{\theta \in\mathbb{R}^d}\, \mathbb{E}_{P_\theta}\|T(X)-\psi(\theta)\|^2\\
    &\qquad\ge \sup_{h\in\mathbb{R}^d}\,\left[\left(\frac{\|\int_{\mathbb{R}^d} \psi(t)-\psi(t-h)\,dQ\|^2}{4H^2(\mathbb{P}_0, \mathbb{P}_h)} \right)^{1/2}-\left(\int_{\mathbb{R}^d} \|\psi(t)-\psi(t-h)\|^2 \, dQ\right)^{1/2}\right]^2_+
\end{align*}
where $H^2(\mathbb{P}_0, \mathbb{P}_h)$ is the Hellinger distance between two mixture distributions $\mathbb{P}_0$ and $\mathbb{P}_h$ indexed by $h \in \mathbb{R}^d$ (see \eqref{eq:mixture-def}). Informally, $\mathbb{P}_h$ is a probability measure ``slightly perturbed'' from $\mathbb{P}_0$ by $h$, which is a mixture between $P_0$ and $Q$. 

Not only these lower bounds are non-asymptotic, but also they are asymptotically sharp for regular models. In Section~\ref{section:asympt_const}, we demonstrate that these bounds imply the LAM theorem.
Theorem~\ref{thm:crvtBound_hellinger} in particular holds under the irregularity exhibited through the underlying statistical model, which heuristically results in the undefined Fisher information and the failure of the LAM theorem. We address this irregularity and derive the asymptotic minimax rate for irregular problems. The conclusions of this manuscript are consistent with the literature on irregular estimation \citep{ibragimov1981statistical, shemyakin2014hellinger, lin2019optimal} and the asymptotic theory developed by H\'{a}jek and Le Cam \citep{hajek1970characterization,hajek1972local,le1972limits,pollardAsymptotia}. In summary, we provide non-asymptotic sequences of minimax lower bound, which remain valid beyond the context of the LAM theorem, including non-differentiable functionals or irregular models. Non-asymptotic minimax lower bounds in this manuscript converge to the asymptotic optimal constant for regular models and also yield the local asymptotic minimax results for irregular models, unifying these known results into a single expression.

Before proceeding, we would like to acknowledge that general-purpose minimax lower bounds can only go so far in the sense that their application to specific instances requires additional work. We present some examples to show how our lower bounds can be applied.

\subsection{Motivation}\label{section:lit}
\subsubsection{Limitations of the LAM theorem}
Classical local asymptotic results investigate the lower bound to \eqref{eq:local-minimax-risk} when $\Theta_0$ forms a shrinking ball around $\theta_0$. For example, the risk displayed by \eqref{eq:local-minimax-risk} becomes equivalent to the local asymptotic minimax risk under the root-$n$ rate, which is given by
\begin{equation}
    \liminf_{c\longrightarrow \infty}\,\liminf_{n\longrightarrow \infty}\, \inf_{T}\,  \sup_{\|\theta-\theta_0\| < cn^{-1/2}}\, n \mathbb{E}_{P_\theta}\|T(X)-\psi(P_\theta)\|^2 \quad \textrm{where} \quad X \sim P^n_{\theta}.
\end{equation}
where $P_{\theta_0}^n$ denotes the joint distribution of $n$ independent and identically distributed (IID) observations from $P_{\theta_0}$. Particularly, the LAM theorem states that the lower bound to the above display is given by the variance of the \textit{efficient influence function} under fairly demanding regularity conditions. 

One immediate limitation of the LAM theorem is its asymptotic nature. To address this limitation, numerous non-asymptotic minimax lower bounds have been proposed in the literature. Many share the following principle: maximizing the separation of functionals evaluated at two, or more, ``similar" distributions. The similarity of distributions is quantified by different metrics and divergences such as the Hellinger distance, total variation, or the $f$-divergences, which includes the KL and the chi-squared divergences \citep{guntuboyina2011lower}. While these non-asymptotic lower bounds may provide correct rates, they often fall short of recovering the correct constant implied by the LAM theorem. To the best of our knowledge, existing non-asymptotic lower bounds based on the Hellinger distance or total variation~\citep{lecam1973convergence,ibragimov1981statistical,assouad1983deux,donoho1987geometrizing, birge1987estimating} cannot imply the LAM theorem even for regular parametric models. On the other hand, Some chi-squared-based lower bounds {\citep{gill1995applications, gassiat2013revisiting}} can imply the LAM theorem. \citet{cai2011testing} proposed a chi-squared-based lower bound for non-differentiable functionals, resulting in a sharp constant for the problem they studied.

The second limitation of the LAM theorem is that it requires the existence of an efficient influence function. This excludes two classes of estimation problems: irregular distributions and non-differentiable functionals. Under irregular statistical models (e.g., $\mbox{Uniform}[0, \theta], \theta \ge 0$), the chi-squared divergences are often infinite, implying trivial results from the minimax lower bounds that use chi-squared divergences. On the other hand, other metrics such as the Hellinger or total variation distances, do not suffer from this problem under irregularity. Thus, one of the main motivations is to understand whether we can obtain lower bounds based on the Hellinger distance that yields the LAM theorem for regular parametric models. Regarding non-differentiable functionals, there is barely any analogous statement to the LAM theorem. A few works have extended local asymptotic minimaxity in the context of plug-in estimators for directionally differentiable functionals \citep{fang2014optimal} or specific non-differentiable functionals that are non-differentiable transforms of a regular parameter \citep{song2014local, song2014minimax}. This manuscript also aims to develop new local asymptotic results that apply to any estimator or possibly non-differentiable functionals 
\subsubsection{Existing minimax lower bounds for non-differentiable functionals and irregular models}

There has been a growing interest in investigating the minimax optimality of non-differentiable functional estimation. 
One popular method for deriving minimax lower bounds involves reducing the original estimation problem to a set of carefully designed composite versus composite hypothesis testing, known as the \textit{fuzzy hypothesis} approach ~\citep{ibragimov1987some, lepski1999estimation,nemirovski2000topics,tsybakov2008}.  This reduction often leads to the study of the divergence metrics between two mixture distributions. In view of the Neyman-Pearson Lemma, a uniformly most powerful test is based on the likelihood ratio, which has a natural connection to the total variation distance between two hypotheses. Minimax lower bounds can then be obtained in terms of various metrics, including the Hellinger distance, the chi-squared divergence, the KL divergence, and others, by invoking a series of inequalities between metrics. See, for instance, Section 2.7.4 of \cite{tsybakov2008}. However, the precision of the constant may be lost in the process of applying different inequalities between metrics. As a result, the minimax optimality of non-differentiable functional estimation is often considered in terms of rate, with no specific focus on the constant. 

In a slightly different approach, \cite{cai2011testing} extends the notion of fuzzy hypotheses to the constrained risk inequality~\citep{brown1996constrained}, deriving a precise minimax constant for the $L_1$-norm estimation under a Gaussian sequence model. In this manuscript, we aim to explore similar extensions of classical two-point risk lower bounds to mixture distributions in order to obtain a sharp constant for the minimax lower bounds. Towards this goal, we refer to the well-established literature on parametric estimation, which has extensively investigated risk lower bounds with accurate non-asymptotic constants. Among the commonly used bounds are the LAM bound \citep{fisher1922mathematical, radhakrishna1945information, cramer1999mathematical}, which was improved by \citet{barankin1949locally}, the higher-order analogue given by Bhattacharyya bound \citep{bhattacharyya1946some}, the Hammersley-Chapman-Robbins bounds \citep{hammersley1950estimating,chapman1951minimum}, and the Weiss-Weinstein bound \citep{weiss1985lower}. Each of these bounds requires varying degrees of regularity conditions and is a variant of the others, as unified by \citet{Weinstein1988AGC}. They all asymptotically imply the Cram\'{e}r-Rao bound under different levels of regularity conditions, and they offer sharper non-asymptotic constants for different settings. Notably, all of these bounds, including \cite{cai2011testing}, define the lower bounds in terms of the likelihood ratio between local models, or some analogous objects. This can be understood as a variant of the chi-squared divergence between local models, which typically yields trivial lower bounds for irregular models.

The lower bounds presented in this manuscript are also closely related to the technique of obtaining a sharp minimax lower bound by exploring the \textit{least-favorable prior} of the Bayes risk such as the \textit{van Trees inequality} \citep{van2004detection}. Traditionally, the literature on the van Trees inequality \citep{van2004detection, gill1995applications, jupp2010van} requires strong regularity conditions, such as pointwise differentiablity of the statistical model, while \citet{gassiat2013revisiting} recently proved that the van Trees inequality still holds without them. Taking inspiration from \citet{gassiat2013revisiting}, we provide further extensions of the van Trees inequality for non-differentiable functionals based on the Hellinger distance, which is well-defined for an irregular model. This property of the Hellinger distance has gathered a growing interest for studying minimax lower bounds in the presence of irregularity \citep{ibragimov1981statistical, donoho1987geometrizing, chen1997general,shemyakin2014hellinger, lin2019optimal, pollardAsymptotia}. The present construction is consistent with the message of these proponents. 

Finally, the presented lower bounds share connections to the classical result in nonparametric functional estimation, namely the \textit{modulus of continuity}, which was investigated by \citet{donoho1987geometrizing}. As a consequence of the new minimax lower bounds in this manuscript, we resolve one of the open problems since \cite{donoho1987geometrizing} to characterize the non-asymptotic minimax lower bound that implies a precise limiting constant. For linear functional estimation, the modulus of continuity has been analyzed with an attempt to establish non-asymptotic efficiency theory \citep{mou2022off}. The result in this manuscript can be seen as one of the first steps towards understanding non-asymptotics for non-differentiable functional estimation. The modulus of continuity is also considered in the context of impossibility results, often known as \textit{ill-posedness} in econometrics literature \citep{potscher2002lower, forchini2005ill}. The presented results generalize their asymptotic statements to the non-asymptotic context. 

\subsection*{Organization.}The remaining manuscript is organized as follows: Section~\ref{section:background} provides necessary background and states existing results that largely inspired this work. Section~\ref{section:van_trees} presents new minimax lower bounds. Section~\ref{section:asympt_const} investigates asymptotic properties of the proposed lower bounds with particular attention to the preservation of sharp constants. Section~\ref{section:nonasymptotic_lower_bound} presents the application to several estimation problems in the presence of non-differentiableness or irregularity and Section~\ref{section:estimators} provides a visual comparison of the new non-asymptotic lower bounds to upper bounds exhibited by different estimators. Finally, concluding remarks and a discussion of open problems are provided in Section~\ref{section:conclusion}.
\subsection*{Notation.}
Throughout the manuscript, we adopt the following convention for notation. We denote by $\|\cdot\|$ a general vector norm. Given $x\in\mathbb{R}$, we write $[x]_+ = \max(x,0)$. For $x\in\mathbb{R}^d$, $\|x\|_2$ denotes the Euclidean norm in $\mathbb{R}^d$. For a real-valued function $f$ on $S$, the supremum norm is defined as $\|f\|_\infty = \sup_{x\in S}|f(x)|$. The open $\mathbb{R}^d$-ball centered at $y_0 \in \mathbb{R}^d$ with radius $\delta > 0$ is denoted as $B(y_0, \delta) := \{y \,: \|y_0-y\|_2 < \delta\}$. The unit sphere in $\mathbb{R}^d$ is denoted by $\mathbb{S}^{d-1} := \{u \in \mathbb{R}^d : \|u\|_2=1\}$. We let $\mathbb{E}_P$ represent the expectation under $P$, whereas $\mathbb{E}_\theta$ denotes the expectation under $P_\theta$. Furthermore, $L_2(P)$ is the set of $P$-measurable functions that satisfy the condition
\[L_2(P) := \left\{f : \mathcal{X} \mapsto \mathbb{R}^k \, \bigg|\, \int \|f(x)\|_2^2 \, dP(x) < \infty\right\}.\]

\section{Setup and background}\label{section:background}
\subsection{Preliminaries}
Throughout this section, we consider probability measures defined on a shared measurable space $(\mathcal{X}, \mathcal{A})$ with a $\sigma$-finite measure $\nu$. We assume that any probability measure $P$ we consider is absolutely continuous with respect to $\nu$ and has a well-defined density function, denoted by $dP = dP/d\nu$. We omit the specification of the base measure $\nu$ when it is clear from the context. Suppose we observe $X$ from an unknown distribution $P_0$, belonging to a possibly nonparametric model $\mathcal{P}$. The \textit{parametric submodel} $\mathcal{P}_{\Theta} := \{P_\theta : \theta \in \Theta\} \subset \mathcal{P}$ is defined as a set of probability measures indexed by a parameter space $\Theta \subseteq \mathbb{R}^d$, which contains the data-generating distribution $P_0$. Without loss of generality, we assume that the 
true parameter corresponds to $\theta=0$ in $\Theta$.

The local behavior of a path $h \mapsto P_{t+h} \in \mathcal{P}_{\Theta}$ as it approaches $P_t$ is often of interest. A local path is said to be  \textit{Hellinger differentible} at $h = 0$ \citep{pollardAsymptotia} if there exists a measurable vector-valued function $\dot \xi_t : \mathcal{X}\mapsto \mathbb{R}^d$ that is square integrable with respect to the $\sigma$-finite measure $\nu$, and as $\|h\|_2 \longrightarrow 0$, 
\begin{align}
    \int \left[dP_{t+h}^{1/2} - dP_t^{1/2} - h^{\top}\dot \xi_t\right]^2 \, d\nu = o(\|h\|_2^2).\nonumber
\end{align}
The Fisher information matrix of $P_{t+h}$ at $h=0$ is defined as 
\begin{align}
    \mathcal{I}(t) := 4\int \dot \xi_t \,\dot \xi_t^{\top}\, d\nu \label{eq:hellinger-fisher-info}
\end{align} 
under the Hellinger differentiability. We also define a slightly stronger but more commonly used notion of differentiability for parametric paths, called \textit{quadratic mean differentiability} (QMD). A local path is differentiable in quadratic mean, or QMD, at $h = 0$ if there exists a $P_t$-square integrable vector-valued function $ g_t : \mathcal{X}\mapsto \mathbb{R}^d$ that satisfies, as $\|h\|_2 \longrightarrow 0$, 
\begin{align}\label{eq:QMD}
    \int \left[dP_{t+h}^{1/2} - dP_t^{1/2} - \frac{1}{2}h^{\top}g_t \, dP_t^{1/2}\right]^2\, d\nu = o(\|h\|_2^2).
\end{align}
The function $g_t$ is commonly referred to as the score function of the path at $h=0$. When $P_{t+h}$ is QMD at $h=0$, it implies the Hellinger differentiability with $\dot \xi_t = \frac{1}{2} g_t \, dP_t^{1/2}$ (See Theorem 20 of \citet{pollardAsymptotia}); however, the converse is not true. The corresponding Fisher information matrix under the QMD assumption is given by \[\mathcal{I}(t) := \int  g_t \, g_t^{\top}\, dP_t.\]
Neither definitions require the distributions to possess pointwise differentiable densities. For instance, consider the double-exponential density $t \mapsto \frac{1}{2}\exp(-|t - \theta|)$ with a parameter $\theta \in \Theta$. This density function is not differentiable at $t = \theta$, yet its Fisher information matrix exists.

When $\Theta \subseteq \mathbb{R}^d$ with some base measure $\mu$, we define a location family induced by a probability measure $Q$ on $\Theta$ as $\{Q(\cdot-h): h\in \Theta\}$. If we assume that $Q$ has an absolutely continuous density function $q$ with an almost-everywhere derivative $\nabla q$, the corresponding location family is differentiable in quadratic mean with the Fisher information given by
\begin{equation}\label{eq:Fisher-information-prior}
\mathcal{I}(Q) := \int \frac{\nabla q(t) \nabla q(t)^{\top}}{q(t)}I(q(t)>0)\, d\mu.
\end{equation}
See Example 7.8 of \cite{van2000asymptotic}. 
Finally, we impose the following regularity conditions on probability measure $Q$, under which $Q$ is called ``nice".
\begin{definition}[Nice priors]\label{as:regularprior-vt} Given an absolutely continuous mapping $\phi:\Theta_0\to\mathbb{R}^k$ and a vector norm $\|\cdot\| : \mathbb{R}^k \mapsto \mathbb{R}_+$, a probability measure $Q$ on $\Theta_0$  is ``nice" if it satisfies the following conditions: 
\begin{enumerate}
    \item[(1)] it has a Lebesgue density $q$ and $q$ is an absolutely continuous function with a positive definite Fisher information $\mathcal{I}(Q)$,
    \item[(2)] both $\|\phi\|$ and $\|\nabla \phi\|$ is $Q$-integrable and $\int_{\Theta_0}\, \mathrm{tr}\,(\mathcal{I}(t))\, dQ < \infty$, and 
    \item[(3)] it holds that $q(\theta) \to 0$ as $\theta$ approaches any boundary point of $\Theta_0$ with finite norm along some canonical direction.
\end{enumerate}
\end{definition}

This manuscript primarily focuses on two divergence metrics, namely the chi-squared divergence and the Hellinger distance. Consider two probability measures, $P_0$ and $P_1$, defined on a common measurable space. Further, assume that there exist densities $dP_0 = dP_0/d\nu$ and $dP_1= dP_1/d\nu$ with respect to a common $\sigma$-finite measure $\nu$. The chi-squared divergence is defined as 
\begin{align*}
    \chi^2(P_1 \| P_0) := \begin{cases}
        \int \left(dP_1/dP_0-1\right)^2 \, dP_0 & \mbox{if }P_1\textrm{ is dominated by } P_0 \\
        \infty & \mbox{otherwise}
    \end{cases}
\end{align*}
and the Hellinger distance is defined as 
\begin{align*}
     H^2(P_0, P_1) := \int \left(dP_0^{1/2}-dP_1^{1/2}\right)^2.
\end{align*}
\subsection{Minimax lower bounds via parametric paths}
This section states the \textit{local asymptotic minimax (LAM) theorem}, a fundamental result in the efficiency theory. We then provide several non-asymptotic minimax lower bounds that asymptotically imply the best constant in view of the LAM theorem. For clarity and ease of illustration, we focus on a real-valued function $\psi : \Theta \mapsto \mathbb{R}$ where $\Theta \subseteq \mathbb{R}^d$ under parametric models. Although the result applies to nonparametric models, we defer a detailed discussion to Section~\ref{sec:semiparametric_minimax}. We assume that $X_1, \dots, X_n \in \mathcal{X}^n$ are IID observations from $P_{\theta_0} \in \{P_\theta : \theta \in \Theta\}$ and $t\mapsto P_t$ is Hellinger differentiable at all $\theta \in \Theta$. We now state the LAM theorem.
\begin{theorem}[Local asymptotic minimax theorem \citep{hajek1970characterization, hajek1972local, le1972limits}]\label{theorem:LAM_parametric}
Assuming that the mappings $t \mapsto \psi(t)$ and $t \mapsto \mathcal{I}(t)$ are continuously differentiable at $t=\theta_0$, then for any measurable function $T : \mathcal{X}^n \mapsto \mathbb{R}$, 
\begin{align}\label{eq:lam_original_lower_bound}
    \liminf_{c \longrightarrow \infty} \, \liminf_{n \longrightarrow \infty}\, \inf_{T}\, \sup_{\|\theta-\theta_0\| < cn^{-1/2}} \, n \mathbb{E}_{P_\theta^n}|T(X)-\psi(\theta)|^2\ge\nabla \psi(\theta_0)^{\top} \mathcal{I}(\theta_0)^{-1}\nabla \psi(\theta_0).
\end{align}
\end{theorem}
When the asymptotic risk of an estimator matches the constant provided by the LAM theorem, the estimator can no longer be improved asymptotically. This, for instance, suggests the asymptotic efficiency of maximum likelihood estimators (MLEs) for this specific estimation task. 
It is essential to note that the first supremum over a neighborhood around $\theta_0$ is a crucial feature that cannot be removed. This supremum over a small neighborhood eliminates any estimator that performs exceptionally well on a Lebesgue measure-zero set, known as superefficient estimators. Without the supremum, the LAM theorem only applies to a family of \textit{regular} estimators, which is restrictive in the context of non-differentiable functional estimation \citep{hirano2012impossibility}. See Section 2.3 of \citet{fang2014optimal} for further discussion on the motivation behind the local asymptotic minimax risk.


We state two existing risk lower bounds that play significant roles in the present development. The first result is the Hammersley-Chapman-Robbins (HCR) bound, initially introduced by \citet{hammersley1950estimating} and \citet{chapman1951minimum}:
\begin{theorem}[HCR bound \citep{hammersley1950estimating, chapman1951minimum}]\label{theorem:HCR-bound}
Let $\psi : \Theta \mapsto \mathbb{R}$ be a real-valued mapping for $\Theta \subseteq \mathbb{R}^d$. 
\begin{align*}  
    &\inf_{T:\mathrm{unbiased}}\, \sup_{\theta \in \Theta}\,\mathbb{E}_\theta|T(X)-\psi(\theta)|^2 \ge \sup_{\theta_0, \theta_1 \in \Theta}\, \frac{|\psi(\theta_1) -\psi(\theta_0)|^2}{\chi^2\left(P_{\theta_1} \| P_{\theta_0}\right)}.
\end{align*}
\end{theorem}
Assuming that $P_{\theta_0}$ is a fixed data generating distribution and taking the limit as $\theta_1 \rightarrow \theta_0$, the HCR lower bound implies the Cram\'{e}r–Rao bound under the strong regularity conditions. 
The regularity conditions can be weakened to the Hellinger differentiability by deriving an analogous lower bound in terms of the Hellinger distance (see \cite{simons1983cramer} and Exercise VI.5 of \cite{polyanskiy2022information}). However, one major limitation of the HCR bound is that it only holds for unbiased estimators. In particular, if the function $\psi$ is non-differentiable, no sequences of unbiased estimators exist \citep{hirano2012impossibility}, indicating that the HCR lower bound cannot be directly applied to the context of this manuscript.

The second result is the van Trees inequality, which considers the minimax lower bound in terms of the worst Bayes risk. It achieves a sharp constant by seeking the supremum over all possible priors, known as the \textit{least-favorable prior}. Let $\Theta \subseteq \mathbb{R}^d$ and $Q$ be a probability measure defined on $\Theta$ with a density function $dQ$. Further assuming that $Q$ is ``nice", we state the following:
\begin{theorem}[The van Trees inequality \citep{gassiat2013revisiting}]\label{theorem:original_van_trees}
Suppose $\{P_\theta : \theta \in \Theta\}$ for $\Theta \subseteq \mathbb{R}^d$ is Hellinger differentiable for all $\Theta$ and the mapping $\psi$ is absolutely continuous on $\Theta$. Then for any measurable function $T$ and a ``nice" prior distribution $Q$ (see Definition~\ref{as:regularprior-vt}) ,
\begin{align*}
    &\inf_{T}\,\sup_{\theta \in \Theta}\,\mathbb{E}_\theta|T(X)-\psi(\theta)|^2\ge \left(\int_\Theta \nabla \psi(\theta)\, dQ(\theta)\right)^{\top}\left(\mathcal{I}(Q) + \int_{\Theta} \mathcal{I}(\theta)\, dQ(\theta)\right)^{-1}\left(\int_\Theta \nabla \psi(\theta)\, dQ(\theta)\right)
\end{align*}
where $\mathcal{I}(Q)$ is the Fisher information of the location family induced by the prior $Q$ defined in~\eqref{eq:Fisher-information-prior}. 
\end{theorem}
The van Trees inequality is particularly desirable since it does not confine the choice of estimators to unbiased estimators. Furthermore, it has been shown that the exact constant for the LAM theorem can be obtained using the van Trees inequality by assuming $\psi$ is continuously differentiable at $\theta_0$ and selecting a prior distribution $Q$ that concentrates at the true parameter $\theta_0$ with the rate $n^{-1/2}$ \citep{gassiat2013revisiting}. However, the van Trees inequality gives a trivial lower bound for an irregular statistical model where the Fisher information is undefined. The main results, presented below, aim to remove these limitations from the existing results.

\section{General minimax lower bounds under weaker assumptions}\label{section:van_trees}

This section provides the main results of this manuscript. Section~\ref{sec:smooth-approximation} provides a general minimax lower bound via smooth approximation. Sections~\ref{sec:vt-chi-sq} and \ref{sec:vt-hellinger} provide two extensions of the HCR bound without requiring the differentiability of functionals. These results hold for nonparametric functionals and are achieved through the standard least-favorable parametric paths argument from semiparametric statistics, which is postponed to Section~\ref{sec:semiparametric_minimax}. Proofs of all theorems in this section are provided in the Supplementary Material.

\subsection{Lower bound based on the approximating functionals}\label{sec:smooth-approximation}
When the target functional $\psi$ does not possess certain properties, such as differentiability, it is often useful to approximate $\psi$ with an alternative functional $\phi$ at the expense of the bias introduced by the approximation. This is a common approach for non-differentiable functional estimation where the target functional is first smoothed by, for instance, convolution or basis expansion. We formalize this idea in the first main result. 

Suppose $X$ is drawn from an unknown distribution $P_0$, which belongs to a statistical model $\mathcal{P} := \{P_t : t \in \Theta\}$ defined on a measurable space $(\mathcal{X}, \mathcal{A})$ with each possessing a density with respect to $\sigma$-finite measure $\nu$. Let $\psi : \mathcal{P} \mapsto \mathbb{R}^k$ denote a vector-valued functional where the estimand of interest is the evaluation of the functional at the population parameter $\psi(P_{0})$. The follow result is merely a consequence of the reverse triangle inequality:
\begin{lemma}\label{thm:L2-approximation}
Given a measure space $(\Theta, \mathcal{T}, \mu)$, let $\Theta_0 \subseteq \Theta$ be any subset of $\Theta$ and let $\mathcal{Q} = \mathcal{Q}(\Theta_0)$ be a collection of probability measures on $\Theta_0$ equipped with a density function with respect to the base measure $\mu$, denoted as $dQ = dQ/d\mu$. Let $\Phi$ be any collection of functionals $\Phi := \{\phi : \mathcal{P} \mapsto  \mathbb{R}^k\}$. For any measurable function $T : \mathcal{X} \mapsto \mathbb{R}^k$ and vector norm $\|\cdot\| : \mathbb{R}^k \mapsto \mathbb{R}_+$, it holds that
\begin{align*}
&\sup_{\theta\in\Theta_0}\,\mathbb{E}_{\theta}\|T(X)-\psi(P_\theta)\|^2  \\
&\qquad \ge \sup_{\phi \in \Phi,\, Q\in \mathcal{Q}}\, \left[\left(\int \,\mathbb{E}_{\theta}\left\|T(X)-\phi(P_\theta)\right\|^2\, dQ\right)^{1/2} - \left(\int \|\psi(\theta)-\phi(\theta)\|^2dQ(\theta)\right)^{1/2}\right]_+^2.
\end{align*}
\end{lemma}
This result states that the minimax lower bound for any functional estimation can be expressed as the trade-off between the Bayes risk of estimating $\phi$ and the approximation error of $\psi$ by $\phi$. For a real-valued functional, the approximation error is expressed as $L_2(Q)$-norm, hence the functional $\phi$ needs not to approximate $\psi$ uniformly well. 

We present the application of this result by considering the projection of $\psi$ onto the set of absolutely continuous functions. The Bayes risk for absolutely continuous functionals can be characterized by the van Trees inequality. We define a collection of probability measures $\mathcal{Q}^\dag$ supported on $\Theta_0$, satisfying Definition \ref{as:regularprior-vt}. Note that $\mathcal{Q}^\dag$ depends on $\phi$. We then state the following:
\begin{theorem}\label{cor:vt-approximation}
Suppose $\{P_\theta : \theta \in \Theta_0\}$ for any open subset $\Theta_0 \subseteq \mathbb{R}^d$ is Hellinger differentiable for all $\Theta_0$ and $\phi : \mathbb{R}^d \mapsto \mathbb{R}^k$ is an absolutely continuous vector-valued functions on $\Theta_0$ with almost-everywhere derivative $\nabla \phi$. Then for any measurable function $T : \mathcal{X} \mapsto \mathbb{R}^k$, vector norm $\|\cdot\| : \mathbb{R}^k \mapsto \mathbb{R}_+$, its dual norm  $\|\cdot\|_*$, 
\begin{align*}
&\inf_T\, \sup_{\theta\in\Theta_0}\,\mathbb{E}_{\theta}\|T(X)-\psi(\theta)\|^2 \ge\sup_{Q \in \mathcal{Q}^\dag}\, \left[\sup_{\|u\|_* \le 1}\, \|\Gamma^{1/2} _{Q, \phi}u\|- \left(\int \|\psi(\theta)-\phi(\theta)\|^2dQ(\theta)\right)^{1/2}\right]_+^2
\end{align*}
where
\begin{align*}
	\Gamma_{Q, \phi} := \left(\int_{\Theta_0} \nabla \phi(t)\,dQ\right)^\top \left(\mathcal{I}(Q)+\int_{\Theta_0}\mathcal{I}(t)\, dQ\right)^{-1} \left(\int_{\Theta_0} \nabla \phi(t)\,dQ\right).
\end{align*}
\end{theorem}
Theorem~\ref{cor:vt-approximation} is a straightforward application of the multivariate van Trees inequality (Theorem 12 of \cite{gassiat2013revisiting}) in conjunction with Lemma~\ref{thm:L2-approximation}. Here, $\mathcal{I}(t)$ is the Fisher information associated with the data-generating distribution $P_t$. It is well-known that the Fisher information for an $n$-fold product measure is given by $n\mathcal{I}(t)$ where $\mathcal{I}(t)$ is the Fisher information associated with a single observation under $P_t$. Hence, under $n$ IID observations from $P_t$, the theorem above yields the identical lower bound but replaces $\Gamma_{Q, \phi}$ with 
\begin{align*}
    \left(\int_{\Theta_0} \nabla \phi(t)\,dQ\right)^\top \left(\mathcal{I}(Q)+n\int_{\Theta_0} \mathcal{I}(t)\, dQ\right)^{-1} \left(\int_{\Theta_0} \nabla \phi(t)\,dQ\right).
\end{align*}

\subsection{Lower bound based on the chi-squared divergence}\label{sec:vt-chi-sq}
We now move on to direct minimax lower bounds without introducing an alternative functional. 
We first define two probability measures on a product space $(\mathcal{X} \times \mathbb{R}^d)$. Specifically, we let 
\begin{align}\label{eq:mixture-def}
    d\mathbb{P}_0(x, \theta) := dP_\theta(x)\, dQ(\theta)\quad\mbox{and}\quad d\mathbb{P}_h(x,\theta) := dP_{\theta+h}(x)\, dQ(\theta+h).
\end{align}
The measure under the translation $Q(\theta+h)$ is well-defined for all $h \in \mathbb{R}^d$. We first present the results for parametric models with $\Theta = \mathbb{R}^d$ for ease of exposition, with an extension to more general sets $\Theta \subseteq \mathbb{R}^d$ discussed in Section~\ref{section:diffeomorphism}. We now present the following lower bound:

\begin{theorem}
\label{thm:crvtBound}
For any probability measure $Q$ on $\mathbb{R}^d$ and $\lambda \in [0,1]$, define 
\begin{align*}
	A_{\lambda, \psi, Q, h} := \frac{\|\int_{\mathbb{R}^d} \psi(t)-\psi(t-h)\,dQ(t)\|^2}{\chi^2(\mathbb{P}_h \| \lambda\mathbb{P}_h + (1-\lambda)\mathbb{P}_0)}\quad\textrm{and}\quad
	B_{\psi, Q, h}:= \int_{\mathbb{R}^d}\|\psi(t)-\psi(t-h)\|^2\,dQ(t).
\end{align*}
Then for any measurable function $T : \mathcal{X} \mapsto \mathbb{R}^k$, vector norm $\|\cdot\| : \mathbb{R}^k \mapsto \mathbb{R}_+$, 
\begin{align*}
    \inf_{T}\,\sup_{\theta \in \mathbb{R}^d}\,\mathbb{E}_\theta\|T(X)-\psi(\theta)\|^2  &\ge \sup_{Q,\, \lambda\in[0,1],\, h\in\mathbb{R}^d}\,\left[\sqrt{(1-\lambda)^2A_{\lambda, \psi, Q, h}} - \sqrt{\lambda B_{\psi, Q, h}}\right]_+^2.
\end{align*}
\end{theorem}
The parameter $\lambda$ is introduced to prevent a trivial lower bound. This might happen, for instance, when the statistical model $\mathcal{P}$ exhibits some irregularity, resulting in the divergent denominator of $A_{\lambda, \psi, Q, h}$ when $\lambda = 0$. 
\begin{corollary}\label{cor:lambda-zero}
	When $\lambda =0$, Theorem~\ref{thm:crvtBound} implies that 
\begin{align}~\label{eq:mixture_HCR}
    &\inf_{T}\, \sup_{\theta \in \mathbb{R}^d}\, \mathbb{E}_\theta\|T(X)-\psi(\theta)\|^2~\ge~ \sup_{Q,\, h\in\mathbb{R}^d}\,\frac{\|\int_{\mathbb{R}^d} \psi(t)-\psi(t-h)\,dQ(t)\|^2}{\chi^2(\mathbb{P}_h \| \mathbb{P}_0)}.
\end{align}
\end{corollary}
This corollary is only useful when the involved chi-squared divergence is finite. This result extends the classical HCR bound from a two-point risk (Theorem~\ref{theorem:HCR-bound}) to mixture distributions $\mathbb{P}_h$ and $\mathbb{P}_0$. Additionally, the lower bound no longer requires the infimum to be taken over unbiased estimators.

\subsection{Lower bound based on the Hellinger distance}\label{sec:vt-hellinger}
We present another mixture extension of the HCR bound. This extension is stated under the Hellinger distance, which is desirable since it does not require additional regularity conditions. The corresponding result has connections to the classical minimax lower bound in terms of the Hellinger distance by \cite{donoho1987geometrizing}, which has been frequently considered for non-differentiable functionals, irregular estimation, or non-asymptotic minimax lower bound. See Section~\ref{supp:modulus} of the Supplementary Material for discussion. 
\begin{theorem}
\label{thm:crvtBound_hellinger}
For any probability measure $Q$ on $\mathbb{R}^d$, define 
\begin{align*}
	A_{\psi, Q, h} := \frac{\|\int_{\mathbb{R}^d} \psi(t)-\psi(t-h)\,dQ(t)\|^2}{4H^2(\mathbb{P}_0, \mathbb{P}_h)} \quad\textrm{and}\quad
	B_{\psi, Q, h}:= \int_{\mathbb{R}^d}\|\psi(t)-\psi(t-h)\|^2\,dQ(t).
\end{align*}
Then for any measurable function $T : \mathcal{X} \mapsto \mathbb{R}^k$ and vector norm $\|\cdot\| : \mathbb{R}^k \mapsto \mathbb{R}_+$,
\begin{align*}
    &\inf_{T} \, \sup_{\theta \in\mathbb{R}^d}\, \mathbb{E}_{\theta}\|T(X)-\psi(\theta)\|^2\ge \sup_{h\in\mathbb{R}^d}\,\left[\sqrt{A_{\psi, Q, h}}-\sqrt{B_{\psi, Q, h}}\right]^2_+.
\end{align*}
\end{theorem}
Theorem~\ref{thm:crvtBound_hellinger} can be compared to Lemma 1 of~\cite{simons1983cramer} which restricts to unbiased estimators and does not consider mixtures. The local behavior of the Hellinger distance has been a frequent tool in analyzing irregular statistical models \citep{ibragimov1981statistical,donoho1987geometrizing, shemyakin2014hellinger, duchi2018right, lin2019optimal}. One of the fundamental two-point risk bounds in terms of the Hellinger distance is provided by Theorem 6.1 of \cite{ibragimov1981statistical}; however, the original proof does not provide the optimal constant, leaving room for a minor improvement. Here, we present a refined two-point risk lower bound with an optimal constant. This implies the improved constants for many applications of Theorem 6.1 of \cite{ibragimov1981statistical}, such as Theorem 1 of \citet{lin2019optimal}. Although this improvement still fails to recover the asymptotic constant of the LAM theorem as we later demonstrate, it may still be of independent interest for its simplicity.
\begin{lemma}[Refinement of Theorem 6.1 of \cite{ibragimov1981statistical} for real-valued functionals]\label{lemma:sharper_ihbound}
For any real-valued functional $\psi : \Theta \mapsto \mathbb{R}$, a measurable function $T$ and $\theta_1, \theta_2 \in \Theta$, we have
\begin{align}
     &\frac{1}{2}\left\{\mathbb{E}_{\theta_1}|T(X) - \psi(\theta_1)|^2 + \mathbb{E}_{\theta_2}|T(X) - \psi(\theta_2)|^2\right\} \ge \left[\frac{1 - H^2(P_{\theta_1}, P_{\theta_2})}{4}\right]_+|\psi(\theta_1) - \psi(\theta_2)|^2.\nonumber
\end{align}
This implies that 
\begin{align}
     &\inf_{T}\, \sup_{\theta \in\Theta}\, \mathbb{E}_{\theta}|T(X)-\psi(\theta)|^2 \ge \sup_{\theta_1, \theta_2 \in \Theta}\left[\frac{1 - H^2(P_{\theta_1}, P_{\theta_2})}{4}\right]_+|\psi(\theta_1) - \psi(\theta_2)|^2.\nonumber
\end{align}
\end{lemma}
Lemma~\ref{lemma:sharper_ihbound} can be compared to Lemma~1 of~\cite{simons1983cramer}.
It should be noted that analogous results as Lemma~\ref{lemma:sharper_ihbound} have been reported in the literature as early as \citet{donoho1987geometrizing}. This result is presented to underscore the insufficiency of the two-point risk inequality in recovering the asymptotic constant---a conclusion also reached by \citet{donoho1987geometrizing}. Although Lemma~\ref{lemma:sharper_ihbound} is stated for a real-valued functional, these results can be extended to any vector-valued functional, with the leading constant naturally depending on the vector norm $\|\cdot
\|$. 
\begin{remark}[Generalized van Trees inequality]
Theorems~\ref{cor:vt-approximation}--\ref{thm:crvtBound_hellinger} satisfy the following:
\begin{align*}
    \mbox{Theorems~\ref{cor:vt-approximation}--\ref{thm:crvtBound_hellinger}}\,  \Longrightarrow \, \mbox{the van Trees inequality}  \, \Longrightarrow \, \mbox{LAM theorem}
\end{align*}
with progressively stringent assumptions. Hence, all new minimax lower bounds in this manuscript can be viewed as the generalization of the van Trees inequality. While Theorems~\ref{cor:vt-approximation} and \ref{thm:crvtBound_hellinger} imply the LAM theorem under the weakest known assumptions, Theorem~\ref{thm:crvtBound} requires additional regularity conditions. The proof of the LAM theorem based on the van Trees inequality is due to \cite{gassiat2013revisiting}. Here, we provide heuristic arguments that Theorems~\ref{cor:vt-approximation}--\ref{thm:crvtBound_hellinger} imply the van Trees inequality, deferring details on regularity conditions to the Supplementary Material. Theorem~\ref{cor:vt-approximation} trivially implies the van Trees inequality as the second approximation error term becomes zero for absolutely continuous functionals. Both Theorems~\ref{thm:crvtBound} and \ref{thm:crvtBound_hellinger} imply the van Trees inequality by taking the limit as $\|h\|_2 \longrightarrow 0$. 
Under additional regularity conditions on the statistical model, it follows
\begin{align}
    \chi^2(\mathbb{P}_h \| \mathbb{P}_0)= \chi^2(Q_h\|Q)+\int_{\mathbb{R}^d} \chi^2(P_{t+h}\|P_t)\,\frac{dQ_h^2}{dQ} = h^{\top}\left(\mathcal{I}(Q) + \int_{\mathbb{R}^d} \mathcal{I}(t) \,dQ\right) h+ o(\|h\|_2^2)\label{eq:local_chi2}\nonumber
\end{align}
provided that we can exchange the limit and the integral under the dominated convergence theorem. If $\psi$ is differentiable almost everywhere under the integral, the inequality~\eqref{eq:mixture_HCR} implies the van Trees inequality as $\|h\|_2 \longrightarrow 0$. 
On the other hand, the following expansion holds under weaker conditions:
\begin{align*}
H^2(\mathbb{P}_0, \mathbb{P}_h)&=H^2(Q_h, Q) +\int_{\mathbb{R}^d} H^2(P_{t+h}, P_{t})\,dQ_h^{1/2} \, dQ^{1/2} \\
&= \frac{1}{4}h^{\top}\left(\mathcal{I}(Q) +\int_{\mathbb{R}^d} \mathcal{I}(t) \, dQ \right)h + o(\|h\|^2)
\end{align*}
as $\|h\|_2 \longrightarrow 0$. Therefore, Theorem~\ref{thm:crvtBound_hellinger} implies the van Trees inequality under weaker conditions than Theorem~\ref{thm:crvtBound}. 
\end{remark}

\subsection{The extension from $\mathbb{R}^d$ to a general parameter space}\label{section:diffeomorphism}
Thus far, Theorems~\ref{thm:crvtBound} and \ref{thm:crvtBound_hellinger} are restricted to the case of $\Theta_0 = \mathbb{R}^d$. These results may still be useful for deriving a \textit{global} minimax risk. However, it cannot be applied directly to a \textit{local} minimax risk, for instance, when $\Theta_0 \subset \Theta$ (a strict subset) is a shrinking neighborhood of the parameter space around a particular value $\theta_0$. In this section, Theorems~\ref{thm:crvtBound} and \ref{thm:crvtBound_hellinger} are extended to the general space $\Theta_0 \subseteq \mathbb{R}^d$ by constructing a diffeomorphism\footnote{A diffeomorphism is an isomorphism of differentiable manifolds. It is an invertible function that maps one differentiable manifold to another such that both the function and its inverse are differentiable.} $\varphi : \mathbb{R}^d \mapsto \Theta_0$. Although we only illustrate the extension using Theorem~\ref{thm:crvtBound_hellinger}, an identical argument applies to Theorem~\ref{thm:crvtBound}.

Given a functional $\psi: \Theta_0 \mapsto \mathbb{R}^k$ and a diffeomorphism $\varphi : \mathbb{R}^d \mapsto \Theta_0$, we define $\widetilde{\psi} : \mathbb{R}^d\mapsto\mathbb{R}^k$, a composite function $(\psi \circ \varphi)(t)$ for $t\in\mathbb{R}^d$ and also define
\begin{align*}
    d\widetilde{\mathbb{P}}_0(x,t) := dP_{\varphi(t)}(x)\,dQ(t) \quad \mbox{and} \quad  d\widetilde{\mathbb{P}}_h(x,t) := dP_{\varphi(t+h)}(x)\, dQ(t+h),
\end{align*}
where $Q$ is a probability measure on $\mathbb{R}^d$. The direct application of Theorem~\ref{thm:crvtBound_hellinger} leads to the following corollary:
\begin{corollary}
    For any probability measure $Q$ on $\mathbb{R}^d$ and a diffeomorphism $\varphi : \mathbb{R}^d \mapsto \Theta_0$, define a composite function $\widetilde{\psi} := (\psi \circ \varphi)$,  
\begin{align*}
	A_{\widetilde{\psi}, Q, h} := \frac{\|\int_{\mathbb{R}^d} \widetilde \psi(t)-\widetilde\psi(t-h)\,dQ(t)\|^2}{4H^2(\widetilde{\mathbb{P}}_0, \widetilde{\mathbb{P}}_h)} \quad\textrm{and}\quad
	B_{\widetilde{\psi}, Q, h}:= \int_{\mathbb{R}^d}\|\widetilde\psi(t)-\widetilde\psi(t-h)\|^2\,dQ(t).
\end{align*}
Then for any measurable function $T : \mathcal{X} \mapsto \mathbb{R}^k$ and vector norm $\|\cdot\| : \mathbb{R}^k \mapsto \mathbb{R}_+$,
\begin{align*}
    &\inf_{T}\, \sup_{\theta \in\Theta_0}\, \mathbb{E}_{\theta}\|T(X)-\psi(\theta)\|^2\ge \sup_{\varphi}\, \sup_{h\in\mathbb{R}^d}\,\left[\sqrt{A_{\widetilde{\psi}, Q, h}}-\sqrt{B_{\widetilde{\psi}, Q, h}}\right]^2_+
\end{align*}
where the first supremum is over the diffeomorphism $\varphi$ between $\mathbb{R}^d$ and $\Theta_0$.
\end{corollary}

We provide the intuition behind the role of the diffeomorphism with an example. Suppose, for instance, the local parameter space $\Theta_0$ is an open $\|\cdot\|$-ball around $\theta_0$ whose radius shrinks at the rate $n^{-r}$ for $r > 0$. For such $\Theta_0$, we may consider diffeomorphism $\varphi$ in the form of 
\begin{equation}
    \varphi(t) := \theta_0 + n^{-r} \varphi_0(t)\nonumber
\end{equation}
where $\varphi_0 : \mathbb{R}^d \mapsto B([0],1)$ and $B([0],1)$ is an open unit ball in $\mathbb{R}^d$. Without loss of generality, we assume that $\varphi_0(0)=0$ so $\varphi(0)=\theta_0$. Then $\varphi$ admits the following Taylor expansion:
\begin{align}
	\varphi(\delta)-\varphi(0) = n^{-r} \varphi_0(\delta) \approx n^{-r} \nabla\varphi_0(\delta)\delta.
\end{align}
Hence as $\|\delta\|\longrightarrow 0$, the functional $\psi(\varphi(\delta))$ approaches $\psi(\theta_0)$ through a nonlinear path uniquely defined by the gradient of $\varphi_0$. 

\section{Asymptotic properties}\label{section:asympt_const}
This section examines the asymptotic properties of new minimax lower bounds from Section~\ref{section:van_trees}. Specifically, we investigate whether these bounds can recover established asymptotic constants, such as the semiparametric efficiency bound as well as the local minimax rates for irregular estimation. These findings further reinforce our general understanding that (1) the local behavior of the Hellinger distance is easier to assert than that of the chi-squared divergence and (2) refined constants may be obtained by the mixture-based method instead of the two-point risk method.

\subsection{Semiparametric efficiency bound}\label{sec:semiparametric_minimax}
We first demonstrate that Theorems~\ref{cor:vt-approximation}--\ref{thm:crvtBound_hellinger} imply the LAM theorem for nonparametric models, commonly referred to as the \emph{semiparametric efficiency bound}. Interestingly, they can imply a whole spectrum of such results; in classical LAM results, the neighborhood around a parameter shrinks at an $n^{-1/2}$ rate while we can let that neighborhood shrink at an arbitrary rate. This includes the superefficiency phenomenon as well. For example, if the neighborhood is a singleton, then the lower bound should be exactly zero. See Section~\ref{section:estimators} for more details. 

To begin, we introduce additional notation. We consider the estimation of nonparametric functional $\psi : \mathcal{P} \mapsto \mathbb{R}$, where $\mathcal{P}$ is a collection of probability measures belonging to an infinite-dimensional set. We reduce the study of nonparametric functionals to their behavior along parametric paths $P_t$. Although a univariate path suffices for our purpose, all statements in this section can be extended to a multivariate path. In what follows, we define the notion of the smoothness of infinite-dimensional functionals along the QMD path, which is defined as \eqref{eq:QMD}. We observe that the theory remains intact under the Hellinger differentiability, and we only use the QMD for consistent use of classical terminology. A functional is called \textit{pathwise differentiable} given a QMD path with the score function $g_0$ if there exists a real-valued measurable function $\dot \psi_{0} : \mathcal{X}\mapsto \mathbb{R}$ such that
\begin{equation}\label{eq:pathwise_differentiable}
    \left|\psi(P_t) - \psi(P_0) + t  \int \dot \psi_{0}\, g_0 \, dP_0\right| = o(t),\quad\mbox{as}\quad t\longrightarrow0.
\end{equation}
The definition of pathwise differentiability implies that the local behavior of $\psi$ only depends on each path through the linear functional of $\dot \psi_{0}$ and $g_0$. Below, we denote each path by $P_{t, g} \in \mathcal{P}$ for a given generic score function $g$ to make the dependence explicit.  Then, we observe that the minimax risk of functional estimation in a nonparametric model must be at least larger than the supremum of the minimax risk along any parametric paths indexed by $g$. Here, the supremum is taken over the entire space of $g$, called the \textit{tangent set} $\mathcal{T}_{P_0}$ of $\mathcal{P}$ at $P_0$. A functional is pathwise differentiable \textit{relative to} $\mathcal{T}_{P_0}$ if equation \eqref{eq:pathwise_differentiable} holds for all $g \in \mathcal{T}_{P_0}$. Although the function $\dot \psi_{0}$ is not generally unique, there is a unique projection of $\dot \psi_{0}$ onto the \textit{tangent space}, or the closed linear span of the tangent set. The projected function is called the \textit{efficient influence function} and plays an important role in the semiparametric efficiency theory. We consider the nonparametric model, containing all distributions on the shared measurable space. The tangent space associated with this model at $P_0$ corresponds to the collection of mean-zero functions in $L_2(P_0)$, i.e., the entire Hilbert space of mean-zero, finite-variance functions. This is the maximal tangent space and we denote by $L_2^0(P_0)$. With this terminology in place, we present the following nonparametric analog of the LAM theorem.
\begin{theorem}[Local asymptotic minimax theorem II (Theorem 5.2 of \citet{van2002semiparametric})]\label{theorem:LAM-nonparametric}
Let the functional $\psi : \mathcal{P}\mapsto \mathbb{R}$ be pathwise differentiable at $P_0$ relative to the maximal tangent space $L_2^0(P_0)$ with an efficient influence function $\dot \psi_{0}$. Then for any measurable function $T$ of the $n$ IID observations from $P_{t,g}$,
\begin{align}
    \sup_{g\in L_2^0(P_0)}\,\liminf_{c \longrightarrow \infty}\,\liminf_{n \longrightarrow \infty}\, \inf_{T}\, \sup_{|t| < cn^{-1/2}} \, n \mathbb{E}_{P^n_{t, g}}|T(X)-\psi(P_{t, g})|^2\ge\int \dot \psi^2_{0}\, dP_0.
\end{align}
\end{theorem}
The resulting lower bound is often called the \textit{semiparametric efficiency bound} and it extends the LAM theorem to nonparametric contexts. We now demonstrate the application of Theorems~\ref{cor:vt-approximation}--\ref{thm:crvtBound_hellinger} and Lemma~\ref{lemma:sharper_ihbound} to derive the semiparametric efficiency bound. Before applying Theorem~\ref{thm:crvtBound} in particular, we introduce an additional condition. Following Example 25.16 of \cite{van2000asymptotic}, we define smooth and bounded parametric paths as follows:
\begin{equation}\label{eq:differentiable_path}
    dP_{t,g}(x) = \frac{1}{C_t}\kappa(t g(x))\, dP_0(x)
\end{equation}
where $C_t := \int \kappa(t g(x))\, dP_0$. We assume that $\kappa(0)=\kappa'(0)=1$ and $\|\kappa'\|_\infty \le K$ and $\|\kappa''\|_\infty \le K$ for some constant $K$. For instance, the function $\kappa(t) := 2/(1+\exp(-2t))$ satisfies this condition. This choice of parametric paths allows the score function to be unbounded but the paths are bounded themselves. Crucially, this path asserts the necessary regularity conditions for the convergence of the probability measure in the chi-squared divergence, uniformly over $g \in \mathcal{T}_{P_0}\subseteq L^0_2(P_0)$ (See Lemma 1 of \citet{duchi2021constrained}). We then state the following result related to the semiparametric efficiency bound (SEB).
\begin{proposition}\label{prop:LAMsemiparametric}
Assuming the identical setting as Theorem~\ref{theorem:LAM-nonparametric}, the following statements hold:
\begin{enumerate}
	\item[(i)] Theorem~\ref{cor:vt-approximation} implies the SEB for any QMD parametric path,
    \item[(ii)] if the parametric path is defined as \eqref{eq:differentiable_path}, then Theorem~\ref{thm:crvtBound} implies the SEB, 
    \item[(iii)] Theorem~\ref{thm:crvtBound_hellinger} implies the SEB for any QMD parametric path, and 
    \item[(iv)] Lemma~\ref{lemma:sharper_ihbound} implies the SEB for any QMD parametric path with the lower bound multiplied by the constant $C \approx 0.28953$. 
\end{enumerate}
\end{proposition}
Proofs for (i)--(iv) are available in Supplementary Material. We also provide an analogous statement for the parametric version of the LAM theorem (Theorem~\ref{theorem:LAM_parametric}) in Supplementary Material.
The statement (iv) of Proposition~\ref{prop:LAMsemiparametric} indicates that two-point risk lower bounds may be insufficient for providing the optimal constant. The statement (ii) of Proposition~\ref{prop:LAMsemiparametric} shows that the lower bound based on the chi-squared divergence restricts the choice of parametric paths beyond QMD. Such restriction can be undesirable for certain cases, as we discuss in the following remark.

\begin{remark}[The implication from restricting parametric paths]
When the linear closure of the tangent set spans the entire $L_2^0(P_0)$, the choice of parametric path does not impact the lower bound. Therefore, there is no loss in selecting a specific path such as the one given by \eqref{eq:differentiable_path}. However, if the tangent set is a strict subset of $L_2^0(P_0)$, such as a tangent cone, the constraint on the path becomes undesirable. For instance, \cite{kuchibhotla2021semiparametric} considers the projection of arbitrary QMD parametric paths onto a working statistical model in order to satisfy certain shape constraints. The statement (ii) of Proposition~\ref{prop:LAMsemiparametric} may not be generally applicable in such cases.
\end{remark}



\subsection{Local minimax rate for irregular estimation}
As motivated earlier, Theorem~\ref{thm:crvtBound_hellinger} is also attractive in the context of irregular problems. Recently, \citet{lin2019optimal} has proposed the use of the H\"{o}lder smoothness of the local Hellinger distance, relative to $P_0$, as the degrees of irregularity for a given estimation problem. This proposal is motivated by the classical result in the literature, namely, Theorem 6.1 of \cite{ibragimov1981statistical}. This result provides a two-point risk inequality in terms of the H\"{o}lder smoothness of the functional and that of the local Hellinger distance between two distributions in a model. In what follows, we demonstrate that Theorem~\ref{thm:crvtBound_hellinger} also recovers the asymptotic minimax rate given by Theorem 6.1 of \cite{ibragimov1981statistical}. 
\begin{lemma}[Theorem 6.1 of \citet{ibragimov1981statistical}]\label{lemma:alpha-beta} Consider the estimation of a real-valued functional $\psi(\theta_0)$ based on the $n$ IID observations from $P_{\theta_0}$. We define $\mathrm{sign}(x)  := x/\|x\|$ for $x \in \mathbb{R}^d$. We assume that there exists a constant $\delta > 0$ such that for all $t, t+h \in \{\theta: \|\theta_0-\theta\| < \delta \}$ and as $\|h\| \longrightarrow 0$, 
\begin{align*}
    H^2(P_t, P_{t+h}) &= C_{1}(t, \mathrm{sign}(h))\|h\|^\alpha + o(\|h\|^\alpha),\quad\alpha\in(0,2],\\
    \psi(t+h)-\psi(t) &= C_{2}(t, \mathrm{sign}(h))\|h\|^\beta  + o(\|h\|^\beta),\quad\beta > 0,
\end{align*}
where $C_1$ and  $C_2$ are real-valued functions of $t$ of $\mathrm{sign}(h)$. Furthermore, these constants are assumed continuous in $t$, uniformly bounded such that $\underline{c}_1 < \|C_{1}\|_\infty < \overline{C}_1$ and $\underline{c}_2 < \|C_{2}\|_\infty < \overline{C}_2$ for some $\underline{c}_1, \overline{C}_1, \underline{c}_2, \overline{C}_2\in(0,\infty)$. Theorem~\ref{thm:crvtBound_hellinger} then implies that there exists a constant $C$, depending on $\underline{c}_1, \overline{C}_1, \underline{c}_2, \overline{C}_2, \alpha$, and $\beta$, satisfying 
\begin{align*}
    \liminf_{c \longrightarrow \infty}\,\liminf_{n \longrightarrow \infty}\,\inf_{T}\,\sup_{\|\theta-\theta_0\| < cn^{-1/\alpha}}\, n^{2\beta/\alpha}\mathbb{E}_{P_\theta^n}|T(X)-\psi(\theta)|^2 \ge C > 0.
\end{align*}
\end{lemma}
The proof of this lemma can be found in Supplementary Material. We note that $\alpha$ is defined on the interval $(0, 2]$ with $\alpha=2$ corresponding to the standard Hellinger differentiable models where the Fisher information is well-defined. Hence the most regular setting corresponds to $\alpha=2$ and $\beta=1$. Although the Hellinger differentiability fails for $\alpha < 2$, the Hellinger distance is always well-defined. For instance, a set of uniform distributions $\text{Unif}(0, \theta)$ indexed by $\theta > 0$ corresponds to $\alpha = 1$. We refer to \cite{lin2019optimal} for more examples of irregular models. While Theorem~\ref{thm:crvtBound_hellinger} can account for this type of irregularity, Theorems~\ref{cor:vt-approximation} and~\ref{thm:crvtBound} yield trivial lower bounds of zero; Theorem~\ref{cor:vt-approximation} fails due to the lack of the Fisher information and Theorem~\ref{thm:crvtBound} fails as the chi-squared divergence for an irregular model is often infinite.

\section{Applications: Local minimax lower bounds}\label{section:nonasymptotic_lower_bound}
In this section, new minimax lower bounds are applied to the functional estimation problems involving points of non-differentiability. 
\subsection{Nonparametric density estimation}
\begin{example}[Nonparametric density estimation]
    Let $X_1, \ldots, X_n \in \mathcal{X} \subseteq \mathbb{R}$ be $n$ IID observations from an unknown density function $f_0$. The functional of interest is the density value at a pre-specified point $x_0 \in \mathcal{X}$, that is, $\psi(f):= f(x_0)$.
\end{example}

We assume that the true density $f_0$ is $s$-times continuously differentiable at $x_0$. We then analyze the following class of density functions:
\begin{align}
	\mathcal{F}(s, f_0, x_0, M) := \left\{f \textrm{ is $s$-times differentiable at $x_0$, satisfying } 
	|f^{(s)}(x_0)| \le (1+M)|f^{(s)}_0(x_0)|\nonumber
\right\}
\end{align}
for fixed $M>0$. Furthermore, we define the following localized set of density functions:
\begin{align}
	U(\delta; \varepsilon) := \left\{f\in\mathcal{F}(s, f_0, x_0, M)\,:\, \int_{|x-x_0|\le \varepsilon}\, |f^{(k)}(x)-f^{(k)}_0(x)|\, dx\le \delta \textrm{ for all } k \in \{0, 1, \ldots, s\}\right\}\nonumber
\end{align}
for fixed $\varepsilon > 0$. This model $U(\delta; \varepsilon)$ generalizes the localized set considered by \citet{anevski2011monotone}. The parameter $M$ is introduced to prevent the true density $f_0$ from lying on the boundary of the local model $U(\delta; \varepsilon)$. As $\psi(f)$ is non-differentiable functional, we consider the following approximation via convolution:
\begin{align}
	\phi(f) = \phi(f; K, h) := \int h^{-1}K\left(\frac{x-x_0}{h}\right) \, f(x)\, dx\nonumber
\end{align}
where $K$ is a kernel function and $h>0$ is a bandwidth parameter. The collection of approximation functionals $\Phi := \{f \mapsto \phi(f; K, h)\, : \, \forall K, h>0\}$ is indexed by the choice of $K$ and the values of $h > 0$. We consider a kernel function that satisfies the following conditions: 
\begin{enumerate}[label=\textbf{(A\arabic*)}]
\item A function $K$ is assumed to satisfy the following conditions:\label{as:kernel}
\begin{enumerate}
    \item[(i)] it is uniformly bounded,
    \item[(ii)] it is $s$-times differentiable with the uniformly bounded $s$th derivative, 
    \item[(iii)] it integrates to one over its support, and 
    \item[(iv)] for all integer $k$ where $0 < k < s$, $\int_{-1}^1 u^k K(u)\, du = 0$ and $\int_{-1}^1 u^s K(u)\, du < \infty$.
\end{enumerate}
\end{enumerate}

We define the following class of approximation functionals:
\begin{align*}
    \Phi := \left\{\phi(f; K, h) :  \textrm{ for all $h > 0$ and $K$ satisfying \ref{as:kernel}}\right\}.
\end{align*}
We now present the application of Theorem~\ref{cor:vt-approximation} to the estimation of $\psi(f)$ approximated with $\phi(f)$:
\begin{lemma}\label{lemma:density-lam}
Let $\delta_n := c_0 n^{-r}$ for $r \in [0,(2s+1)^{-1})$. Then for any $\varepsilon>0$, Theorem~\ref{cor:vt-approximation} implies,
\begin{align}
	&\liminf_{n\longrightarrow \infty}\, \inf_T \sup_{f\in U(\delta_n; \varepsilon)} n^{2s/(2s+1)}\mathbb{E}_{f}|T(X)-f(x_0)|^2 \nonumber\\
	&\qquad \ge \sup_K C(s, M, K)\, f_0(x_0)^{2s/(2s+1)}|f_0^{(s)}(x_0)|^{2/(2s+1)}\nonumber
\end{align}
where $C(s,M,K)$ is a constant only depending on $s, M$ and $K$. 
\end{lemma}
The proof of this lemma is provided in Supplementary Material. To the best of our knowledge, this is a new local asymptotic minimax constant in the context of nonparametric density estimation. A similar object to the term in the lower bound has appeared in the classical literature on kernel density estimation \citep{Devroye1985nonparametric}; Letting $f_0$ be any density on $[0, 1]$ with continuous $s$th derivative, Theorem 11 (Chapter 4, page 49) of \citet{Devroye1985nonparametric} states that 
\begin{align*}
    \liminf_{n \longrightarrow \infty}\, \sup_{f \in H(f_0)}\,n^{2s/(2s+1)} \frac{\mathbb{E}_{f}\int |T-f|^2}{(\int f)^{2s/(2s+1)}(\int |f^{(s)}|^2)^{1/(2s+1)}} > 0
\end{align*}
for any estimator $T$ where $H(f_0)$ is a set of all densities of the form $\sum_{i=1}^\infty \pi_i f_0(x+x_i)$, $\pi_i$ is any probability vector and $\{x_i\}$ is an increasing sequence of real numbers such that $x_{i+1}-x_i > 1$\footnote{The term in the denominator $(\int f)^{2s/(2s+1)}$ may seem redundant since $\int f=1$. This term is generally not trivial for the lower bound of $\mathbb{E}_{f}\int |T-f|^p$ when $p \neq 2$. We include this term for the better comparison with Lemma~\ref{lemma:density-lam}.}. Lemma~\ref{lemma:density-lam} provides a similar result for a pointwise squared risk. Simple algebra also suggests that the kernel density estimator, defined as 
\begin{align*}
    \widehat f := \frac{1}{nh^*}\sum_{i=1}^n K\left(\frac{X_i-x_0}{h^*}\right) \quad \mbox{where}\quad  h^* \propto \left(\frac{f_0(x_0)}{n|f_0^{(s)}(x_0)|^2}\right)^{1/(2s+1)}
\end{align*}
attains the matching upper bound including the dependency on $f_0$ \citep{abramson1982bandwidth, woodroofe1970choosing, hall1993plug, brown1997superefficiency}. Such an estimator is often called \textit{adaptive} in the nonparametric density estimation literature\footnote{This notion of adaptivity is different from the adaptation to unknown smoothness $s$.}.  A reasonable estimate of $h^*$ is obtained by the ``plug-in" rule where $f_0$ and $f_0^{(s)}$ in $h^*$ are replaced with pilot estimators $\widetilde f(x_0)$ and $\widetilde f^{(s)}(x_0)$, constructed from a separate data. The final estimator is defined as
\begin{align*}
    \check{f} := \frac{1}{nh_{\textrm{plug-in}}}\sum_{i=1}^n K\left(\frac{X_i-x_0}{h_{\textrm{plug-in}}}\right) \quad \text{where}\quad h_{\textrm{plug-in}} \propto \left(\frac{\widetilde f(x_0)}{n|\widetilde f^{(s)}(x_0)|^2}\right)^{1/(2s+1)}.
\end{align*}
Assuming these pilot estimators converge in quadratic mean, \citet{brown1997superefficiency} states that the estimator $\check{f}$ asymptotically achieves the minimum pointwise squared risk, i.e., the lowest squared risk under fixed $f_0$, among all choice of bandwidth parameters (See equation (2.5) of \citet{brown1997superefficiency}). \citet{brown1997superefficiency} however argues that such a pointwise assessment of the nonparametric estimator under fixed $f_0$ is often misleading in view of superefficiency, necessitating the evaluation under ``uniform" risk. Lemma~\ref{lemma:density-lam} does not contradict the message by \citet{brown1997superefficiency}. The pointwise risk of the estimator is equivalent to the local minimax risk over a singleton $\{f_0\}$. Lemma~\ref{lemma:density-lam} does not allow the neighborhood $U(n^{-r}; \varepsilon)$ to shrink too fast, which prevents the issue of superefficiency. Hence, the adaptivity to the unknown $f_0$ is not purely an artifact of superefficiency. 

The van Trees inequality has been previously applied to analyze minimax lower bounds in nonparametric problems. See, for instance, Theorem 4 of \citet{anevski2011monotone} and Example 2.3 of \citet{tsybakov2008}. In order to apply the classical van Trees inequality, these approaches typically reduce the original nonparametric problem to parameter estimation along a parametric submodel. Consequently, existing results do not account for risk over a shrinking neighborhood around $f_0$, and do not provide the local dependency on $f_0$. This new result is obtained by a refined application of the van Trees inequality. 

\subsection{Simple directionally differentiable functionals}
We now apply Theorem~\ref{cor:vt-approximation} to the estimation of $\psi(P_\theta):=\max(\theta,0)$ for $\theta \in \mathbb{R}$. This is one of the canonical examples of directionally differentiable functionals. More complex problems such as interval regression \citep{fang2014optimal} and testing the shape of regression \citep{juditsky2002nonparametric}, can be reduced to this form. 

\begin{example}[Estimating $\max(\theta, 0)$ when $\theta=0$]\label{main-example2} Suppose $X_1, \ldots, X_n$ are $n$ IID observations drawn from $P_0$ that belongs to a local model $\{P_\theta : |\theta| < \delta, \theta\in\mathbb{R}\}$ for fixed $\delta > 0$. We assume that this model is Hellinger differentiable with the Fisher information $\mathcal{I}(t)$ for $t \in (-\delta, \delta)$. The functional of interest is $\psi(P_\theta) = \max(0, \theta)$.
\end{example}
\begin{lemma}\label{lemma:max-vt}
Under Example~\ref{main-example2}, Theorem~\ref{cor:vt-approximation} implies that 
    \begin{align}
	\inf_T\, \sup_{|\theta|< \delta}\, \mathbb{E}_\theta |T(X)-\psi(P_\theta)|^2
 &\ge \sup_{a\in[0,1]}\, \frac{a^2}{4\pi^2\delta^{-2}w_a^{-2} + n\sup_{|t|<\delta}\mathcal{I}(t)}\label{eq:kepler-main}
\end{align}
where 
$w_a := 2/(y_a+1)$ and $y_a$ is the inverse of the following equation:
\begin{align}
    y_a - \sin(-\pi y_a)/\pi = |2a-1|.\nonumber
\end{align}
\end{lemma}
The lower bound, given by \eqref{eq:kepler-main}, investigates the least-favorable prior using variational calculus and optimization under absolute moment constraints \citep{ernst2017minimizing}. The corresponding derivation can be found in Supplementary Material. The lower bound involves the inverse of the \textit{Kepler equation} from Celestial Mechanics, which does not have a closed-form solution \citep{kepler}. Hence, this lower bound relies on a computational method to approximate $y_a$. This technique seems to be less known in the literature but has been mentioned in the context of non-asymptotic minimax lower bound for Gaussian mean estimation under bounded constraints \citep{levit2010minimax}. 

Next, we provide a non-asymptotic lower bound for estimating $\max(\theta^\alpha, 0)$ for $0 < \alpha \le 1$ and $\theta\in\mathbb{R}$ based on Theorem~\ref{thm:crvtBound_hellinger}. This function is also non-differentiable at $\theta=0$ and the corresponding lower bound behaves differently depending on the location of the true parameter.

\begin{example}[Estimating $\max(\theta^\alpha, 0)$ for $\alpha > 0$]\label{example2}
Consider the identical settings as Example~\ref{main-example2}, except the functional of interest is now $\psi(P_\theta) = \max(0, \theta^\alpha)$ for $\alpha > 0$. 
\end{example}
\begin{lemma}\label{lemma:max-power}
    Under Example~\ref{example2}, Theorem~\ref{thm:crvtBound_hellinger} implies the following:
\begin{enumerate}
	\item[(i)]When $\theta_0 < 0$, it holds that 
	\begin{align}
	&\inf_{T}\, \sup_{|\theta-\theta_0| < \delta} \, \mathbb{E}_{\theta}|T(X)-\psi(P_\theta)  |^2 \ge  \sup_{Q, \varphi_0}\,\frac{\delta^2\alpha^2 |\mathbb{E}_Q \{\theta_0 +\delta\varphi_0(t)\}^{\alpha-1}\varphi'_0(t)\, I\{t: |\theta_0|/  \delta <\varphi_0(t)\}|^2}{\mathcal{I}(Q) +n\delta^2\mathbb{E}_Q [ \{\varphi_0'(t)\}^2\, \mathcal{I}(\theta_0 + \delta\varphi_0(t))]}\nonumber
\end{align}
\item[(ii)]When $\theta_0 > 0$, it holds that 
\begin{align}
	\inf_{T}\, \sup_{|\theta-\theta_0| < \delta} \, \mathbb{E}_{\theta}|T(X)-\psi(P_\theta)  |^2 \ge 
    \sup_{Q, \varphi_0}\,\frac{\delta^2\alpha^2 |\mathbb{E}_Q \{\theta_0 +\delta\varphi_0(t)\}^{\alpha-1}\varphi'_0(t)\, I\{t: \varphi_0(t) < \theta_0/  \delta \}|^2}{\mathcal{I}(Q) +n\delta^2\mathbb{E}_{Q}[  \{\varphi_0'(t)\}^2\, \mathcal{I}(\theta_0 + \delta\varphi_0(t))]}.\nonumber
\end{align}
\item[(iii)]When $\theta_0 = 0$, it holds that 
\begin{align}
	\inf_{T}\, \sup_{|\theta| < \delta} \, \mathbb{E}_{\theta}|T(X)-\psi(P_\theta)  |^2 \ge\sup_{Q, \varphi_0}\, \frac{\delta^{2\alpha}\alpha^2 |\mathbb{E}_Q \varphi_0(t)^{\alpha-1}\varphi'_0(t)\, I\{t:\varphi_0(t) >0 \}|^2}{\mathcal{I}(Q) +n\delta^2\mathbb{E}_Q  [\{\varphi_0'(t)\}^2\, \mathcal{I}(\delta\varphi_0(t))]}.\nonumber
\end{align}
\end{enumerate}
The supremums are taken over any probability measure $Q$ on $\mathbb{R}$ and any increasing deffeomorphism $\varphi_0 : \mathbb{R} \mapsto (-1, 1)$ such that $\varphi_0(0)=0$ and $\|\varphi'_0\|_\infty < C$ for some constant.
\end{lemma}
Here, the choice of priors no longer needs to satisfy Definition~\ref{as:regularprior-vt}. Instead, Lemma~\ref{lemma:max-power} posits certain requirements over the choice of diffeomorphism.

\begin{remark}[Minimax rates of convergence under IID observations]\label{remark:minimax-rate}
	From the expressions above, we can deduce the local minimax rates of convergence. When $\theta_0 > 0$, the lower bound (ii) of Lemma~\ref{lemma:max-power} above implies that 
	\begin{align}
		\frac{\delta^2O(1)}{\mathcal{I}(Q) + n\delta^2O(1)} = \frac{O(1)}{\delta^{-2}\mathcal{I}(Q) + nO(1)}.\nonumber
	\end{align} 
	Balancing two terms in the denominator, we choose $\delta = O(n^{-1/2})$; the overall rate of convergence is $n^{-1}$, or so-called parametric rate, provided $\theta_0 > 0$. This is reasonable given that $\psi(\theta_0)$ is a differentiable parameter when $\theta_0$ is bounded away from zero.
 
 Similarly when $\theta_0 = 0$, the lower bound (iii)  of Lemma~\ref{lemma:max-power} implies that 
	\begin{align}
		\frac{\delta^{2\alpha}O(1)}{\mathcal{I}(Q) + n\delta^2O(1)} = \frac{O(1)}{\delta^{-2\alpha}\mathcal{I}(Q) + n\delta^{2-2\alpha}O(1)}.\nonumber
	\end{align} 
	Balancing two terms in the denominator, we choose $\delta = O(n^{-1/2})$ again; the overall rate of convergence is now $n^{-\alpha}$ when $\theta_0 =0$, which is strictly slower than the parametric rate when $\alpha < 1$. We thus conclude that the minimax rates of convergence for estimating $\max(\theta^\alpha, 0)$ remain the same as the rates for $\theta^\alpha$ discussed in \citet{gill1995applications}.
\end{remark}
\begin{remark}[Local asymptotic minimax constants] Lemma~\ref{lemma:max-power} also recovers the precise constants for the local asymptotic minimax lower bound. In addition to the setting of Lemma~\ref{lemma:max-power}, we further assume that $\mathcal{I}(t)$ is continuous at $\theta_0$. Given the observation from Remark~\ref{remark:minimax-rate}, we replace $\delta$ with $cn^{-1/2}$ and analyze their limits. When $\theta_0 < 0$, there exists $n$ large enough that $|\theta_0|\ge cn^{-1/2}$, and hence we have
\begin{align}
	\liminf_{c\longrightarrow \infty}\, \liminf_{n \longrightarrow \infty}\, \inf_T \, \sup_{|\theta-\theta_0|<cn^{-1/2}}n\mathbb{E}_{P^n_\theta} |T(X)-\psi(P_\theta)|^2 = 0\nonumber.
\end{align}
When $\theta_0 > 0$, there exists $n$ large enough such that $\theta_0\ge cn^{-1/2}$, and thus $I\{t: \varphi_0(t) < \theta_0/(cn^{-1/2})\} = 1$ for all $t \in \mathbb{R}$. We then have, 
\begin{align}
	&\liminf_{c\longrightarrow \infty}\, \liminf_{n \longrightarrow \infty}\, \inf_T \, \sup_{|\theta-\theta_0|<cn^{-1/2}}n\mathbb{E}_{P^n_\theta} |T(X)-\psi(P_\theta)|^2 \ge \frac{\alpha^2\theta_0^{2\alpha-2}}{\mathcal{I}(\theta_0)}\frac{|\mathbb{E}_Q  \varphi_0'(t)|^2}{\mathbb{E}_Q  \{\varphi_0'(t)\}^2}\nonumber
\end{align}
where we invoke the dominated convergence theorem in view of the continuity of $\mathcal{I}(t)$ at $t=\theta_0$ and the uniform boundedness of $\varphi'_0$. The proof in Supplementary Material claims that the supremum of the last quantity involving $\varphi_0$ is $1$. Hence the local asymptotic minimax constant is $\alpha^2\theta_0^{2\alpha-2}\mathcal{I}(\theta_0)^{-1}$. The case with $\theta_0=0$ is more involved. By the analogous argument, 
\begin{align}
	&\liminf_{c\longrightarrow \infty}\, \liminf_{n \longrightarrow \infty}\, \inf_T \, \sup_{|\theta-\theta_0|<cn^{-1/2}}n^\alpha \mathbb{E}_{P^n_\theta} |T(X)-\psi(P_\theta)|^2 \nonumber\\
	&\qquad\ge \liminf_{c\longrightarrow \infty}\, \frac{c^{2\alpha}\alpha^2 |\mathbb{E}_Q \varphi_0(t)^{\alpha-1}\varphi'_0(t)\, I\{t >0 \}|^2}{\mathcal{I}(Q) +c^{2}\mathcal{I}(\theta_0)\mathbb{E}_Q  \{\varphi_0'(t)\}^2}.\nonumber
\end{align}
The $\liminf$ may not be achieved by the extreme values of $c$, and it depends on specific choices of $Q$ and $\varphi_0$. Taking the supremum of this object over the choice of $Q$ and $\varphi_0$ requires more involved analysis, which we do not pursue in this manuscript.    
\end{remark}

\subsection{Parameter estimation under an irregular model}
The final example provides a local asymptotic minimax result when the Hellinger differentiability fails. 
We demonstrate that Theorem~\ref{thm:crvtBound_hellinger} recovers a local minimax rate as well as correct parameter dependence. 
\begin{example}[Estimating $\theta_0$ from $\mathrm{Unif}(0, \theta_0)$ for $\theta_0 > 0$]\label{example3}
Suppose $X_1, \dots, X_n$ are $n$ IID observations from $\mathrm{Unif}(0, \theta_0)$, which belong to a statistical model $\{\mathrm{Unif}(0, \theta) : 0 < \theta\}$. The functional of interest is $\psi(P_\theta) = \theta$. 
\end{example}

\begin{lemma}\label{lemma:uniform}
Under Example~\ref{example3}, Theorem~\ref{thm:crvtBound_hellinger} implies the following local asymptotic minimax lower bound:
\begin{align*}
    &\liminf_{c\longrightarrow \infty}\, \liminf_{n\longrightarrow \infty}\, \inf_{T}\,\sup_{|\theta -\theta_0|< cn^{-1}}\mathbb{E}_{P_\theta^n}n^2|T(X)-\theta|^2\\
    &\qquad \ge \theta_0^2 \sup_{\eta,\,Q,\,\varphi_0}\, \left[\frac{\left|\int \eta\varphi_0'(t)\, dQ(t)\right|}{2\left(2 -2\int  \exp \left(- \eta\varphi'_0( t)/2\right)\, dQ(t)\right)^{1/2}} -\left(\int \{\eta\varphi_0'(t)\}^2\, dQ(t)\right)^{1/2}\right]^2_+ = C\theta_0^2
\end{align*}
for $C>0$ where the supremum is over $\eta\in\mathbb{R}$, diffeomorphism $\varphi_0$ from $\mathbb{R}$ to $(-1,1)$ and any prior distributions $Q$ over $\mathbb{R}$.
\end{lemma}
The preceding display shows that the local minimax risk for parameter estimation under uniform distribution behaves as $O(\theta_0^2/n^2)$. Without taking into account the constant, this is the correct known dependency. Theorem 4.9 and Proposition 4.5 of \cite{korostelev2011mathematical} further prove that the sharpest constant is $1$ based on the Bayes risk with a uniform prior. This implies that Lemma~\ref{thm:L2-approximation}, which is also based on the worst Bayes risk, can recover the correct constant for irregular problems. However, Theorem~\ref{thm:crvtBound_hellinger} does not readily accommodate a prior with compact support, and it seems to require more involved analysis around the choice of diffeomorphism. We provide additional discussions in Supplementary Material. 


\section{Visual illustration: Upper bounds attained by plug-in estimators}\label{section:estimators}
In this section, we investigate the precise gaps between the risk of estimators and the non-asymptotic local minimax lower bounds for fixed $\delta > 0$ and $n \ge 1$ defined in Examples~\ref{main-example2} and~\ref{example2}. The results in this section are provided mostly for illustration, suggesting non-asymptotic assessment beyond the asymptotic efficiency bound. Consider the following estimation problem:
\begin{align}
	\sup_{|\theta| < \delta}\, n\mathbb{E}_\theta |T(X) - \max(\theta, 0)|^2 \label{eq:max-parameter}
\end{align}
where we observe $n$ IID observations $X:=(X_1, \ldots, X_n)$ from $N(\theta_0, 1)$. We define the local model as $\{N(\theta,1) : |\theta| < \delta\}$. We construct lower bounds to \eqref{eq:max-parameter} using two methods. The first lower bound~\eqref{eq:vT} is due to equation \eqref{eq:kepler-main} of Lemma~\ref{lemma:max-vt}, which is given by
\begin{equation}\label{eq:vT}
		\inf_T\, \sup_{|\theta|< \delta}\, n\mathbb{E}_\theta |T(X)-\psi(P_\theta)|^2\ge \sup_{a\in[0,1]}\, \frac{na^2}{4\pi^2\delta^{-2}w_a^{-2} + n},\tag{\texttt{vT}}
	\end{equation}
	where $w_a$ is defined in Lemma~\ref{lemma:max-vt}. The second lower bound~\eqref{eq:diffeo} is based on Lemma~\ref{lemma:max-power} (iii) with $\alpha=1$. As computing the supremum over $\varphi_0$ and $Q$ is challenging, we simplify the problem by focusing on the following choices:
	\begin{align}\label{eq:diffeo}
	Q \in \{N(\mu, \sigma^2) : (\mu, \sigma)\in \mathbb{R}\times (0,\infty)\}\quad \text{and} \quad  \varphi_0 \in \{t \mapsto \pi/2 \arctan(t/\eta) : \eta > 0\}\nonumber
\end{align}
and optimize the parameters $(\eta, \mu, \sigma)$. We show in the Supplementary Material that this reduction leads to the following simpler optimization:
\begin{align}
	\inf_T \, \sup_{|\theta| < \delta}\, n\mathbb{E}_\theta |T(X) - \max(\theta, 0)|^2 \ge \sup_{\xi_1,\xi_2}\, \frac{4 n\xi_2^2|\E[ (1+(\xi_1+Z\xi_2)^2)^{-1}I(Z >-\xi_1/\xi_2)]|^2}{\pi^2\delta^{-2} + 4n\E[(1+(\xi_1+Z\xi_2)^2)^{-2}] }\tag{\texttt{diffeo}}
\end{align} 
where $Z \overset{d}{=}N(0,1)$ and $\xi_1,\xi_2 \in \mathbb{R}\times(0,\infty)$.

The exact local minimax risk of three estimators is considered: the constant estimator, the plug-in maximum likelihood estimator (MLE), and the plug-in preliminary-test (pre-test) estimator. Specifically, the plug-in pre-test estimator is an example of an irregular estimator, which may exhibit superefficiency at a Lebesgue measure zero set. We now define the estimators. The constant estimator returns a predetermined value regardless of the observation, and we consider the case where $S_n^{\text{const}}:=\delta/2$. This is the best constant estimator that minimizes the local minimax risk, which is given by 
\begin{align}
	\sup_{|\theta| < \delta}\, n\mathbb{E}_\theta |S_n^{\text{const}} - \max(\theta, 0)|^2 = n\delta^2/4 \nonumber.
\end{align}

The remaining two estimators are defined as follows:
\begin{align}
    S_n^{\text{plug-in}} &:=  \max(\widehat \theta_{\text{MLE}}, 0) \quad \text{where} \quad 
    \widehat \theta_{\text{MLE}} := \overline{X}_n = \frac{1}{n}\sum_{i=1}^n X_i, \quad \text{and}\label{eq:plugin_mle}\\
    S_n^{\text{pre-test}} &:=  \max(\widehat \theta_{\text{pre-test}}, 0)\quad \text{where} \quad 
    \widehat \theta_{\text{pre-test}} := \begin{cases}
    \overline{X}_n & \text{If}\quad  |\overline{X}_n| \ge C_n\\
    0 &\text{otherwise}.
    \end{cases}\label{eq:plugin_pretest}
\end{align}
This family of preliminary-test estimators \citep{sclove1972non} is known to be superefficient when $\theta_0=0$. For example, it becomes identical to the Hodges' estimator when $C_n = n^{-1/4}$. The following proposition provides the exact local minimax risks of two estimators.
\begin{proposition}\label{prop:risk-mle}
Assuming $X_1,\dots,X_n$ are $n$ IID observation from $N(0, 1)$ and the plug-in MLE is defined as \eqref{eq:plugin_mle}, the local minimax risk of this estimator is given by 
\begin{align}
   & \sup_{|\theta|<\delta}\, n\mathbb{E}_{\theta} |S_n^{\mathrm{plug-in}} - \max(\theta, 0)|^2 = \sup_{0 \le \theta < \delta}\,\left\{\E [Z^2 I(Z \ge -n^{1/2}\theta)] + n\theta^2P (Z < -n^{1/2}\theta)\right\}\nonumber.
\end{align}
where $Z \overset{d}{=} N(0,1)$. Similarly for the pre-test estimator defined as \eqref{eq:plugin_pretest}, the local minimax risk of this estimator is given by 
\begin{align}
   & \sup_{|\theta|<\delta}\, n\mathbb{E}_{\theta} |S_n^{\mathrm{pre-test}} - \max(\theta, 0)|^2  = \sup_{0 \le \theta < \delta}\,\left\{\E [Z^2 I(Z \ge n^{1/4}-n^{1/2}\theta)] + n\theta^2P (Z <  n^{1/4}-n^{1/2}\theta)\right\}\nonumber.
\end{align}
\end{proposition}


Figure~\ref{fig:estimators} shows that for any fixed sample size $n$ as $\delta$ increases, the lower bounds tend to $\sigma^2 = 1$, and the best among the estimators considered also has risk tending to $\sigma^2=1$. For any fixed sample size $n$ and ``small'' $\delta$, the lower bound is close to zero and the best lowest risk among the estimators considered also tends to zero. This is expected since the constant estimator at zero will have zero risk at $\theta_0 = 0$. Finally, Figure~\ref{fig:estimators} also shows an interesting comparison between different lower bounds. Neither~\eqref{eq:vT} nor~\eqref{eq:diffeo} is a clear winner. 

\begin{figure*}
    \centering
        \centering
        \includegraphics[width=5in]{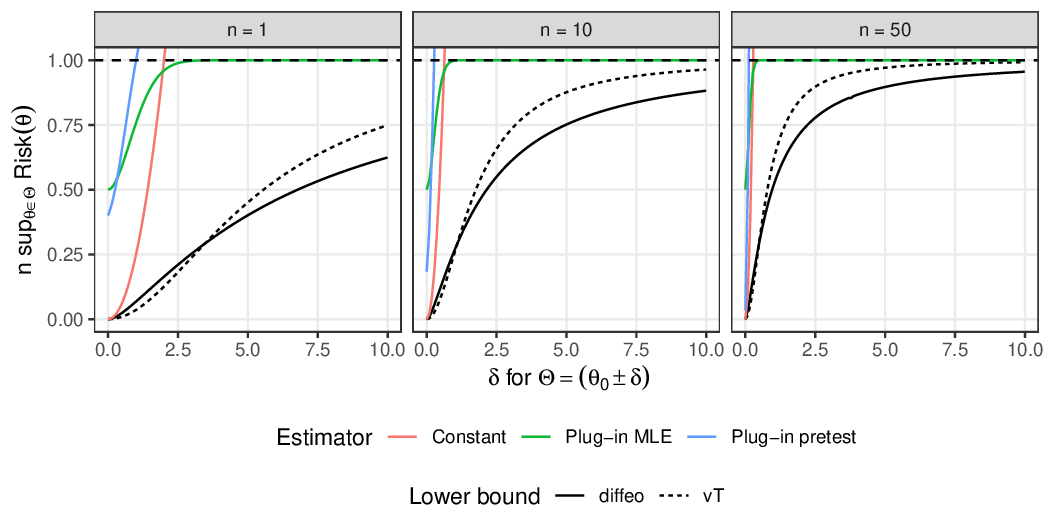}
    \caption{The non-asymptotic local minimax lower bounds and the risk given by different estimators: For fixed $n \ge 1$ with varying size of local neighborhood.}
    \label{fig:estimators}
\end{figure*}

\section{Concluding remarks}\label{section:conclusion}
This manuscript presents new minimax lower bounds for functional estimation without requiring the differentiability of functionals or the regularity of statistical models. We provide three new approaches: the approximation via absolutely continuous functionals (Theorem~\ref{cor:vt-approximation}), the mixture HCR bound based on the chi-squared divergence (Theorem~\ref{thm:crvtBound}), and the Hellinger distance (Theorem~\ref{thm:crvtBound_hellinger}). Only Theorem~\ref{thm:crvtBound_hellinger} provides a non-trivial result for irregular estimation problems. Unlike standard minimax bound for non-differentiable functionals based on testing reduction including the \textit{fuzzy hypothesis} approach, this manuscript focuses on preserving a precise asymptotic constant. All results generalize the van Trees inequality, and imply the well-known semiparametric efficiency bound (Proposition~\ref{prop:LAMsemiparametric}). Theorem~\ref{thm:crvtBound_hellinger} also implies local minimax results for irregular models (Lemma~\ref{lemma:alpha-beta}).

The flexibility of the proposed lower bounds offers many potential applications, especially for non-differentiable functionals or estimation under irregularity. For example, one may consider the minimax lower bound for non-pathwise differentiable functionals or the irregularity that arises on the boundary of projection operations. See a short remark in Section~\ref{supp:extensions} of the Supplementary Material. A similar local asymptotic minimax lower bound was recently derived in the context of plug-in estimators \citep{fang2014optimal}, but there remain many open problems for general estimators. Finally, as none of Theorem~\ref{cor:vt-approximation}--\ref{thm:crvtBound_hellinger} requires the IID observation, the potential application to the sharp minimax lower bound under non-IID observations would also be of interest.


As we discussed in Section~\ref{section:lit}, a precise minimax lower bound for non-differentiable functional estimation was also discovered by \cite{cai2011testing}, who studied the mixture extension of the constrained risk inequality \citep{brown1996constrained}. 
Their method is also attractive as two mixture distributions are constructed based on moment-matching priors, enabling various analytic tools from approximation theory \citep{wu2020polynomial}. The current result only considers a single prior over a parametric path, and it would be interesting to explore a technique based on moment-matching priors.

Finally, \citet{levit2010minimax} discusses the importance of non-asymptotic constants in the minimax paradigm. In particular, \citet{levit2010minimax} observes that different lower bounding methods provide sharp constants according to small, moderate, or large sample sizes. Although the results in this manuscript are valid non-asymptotically and converge to the sharp constant in the limit, we did not study the sharpness of the derived constants for different sample sizes. While \citet{levit2010minimax} focuses on bounded Gaussian mean estimations, the proposed lower bounds can be extended to non-asymptotic analysis for more complex functional estimation problems, which is an important direction for future work.

\bibliographystyle{apalike}
\bibliography{ref.bib}

\newpage
\begin{appendices}
\section{Proof of the lower bound based on the approximation}\label{supp:approx}
We briefly review the setting and notation to which we frequently refer. Suppose we observe a random variable $X$ from an unknown distribution $P_{\theta}$, which belongs to a statistical model $\mathcal{P} := \{P_t : t \in\Theta\}$ defined on a measurable space $(\mathcal{X}, \mathcal{A})$. Here, $\mathcal{A}$ denotes $\sigma$-algebra on $\mathcal{X}$. Let $\psi: \mathcal{P} \mapsto \mathbb{R}^k$ denote the vector-valued function of interest. In other words, the target estimand is described as the evaluation of the functional at the population parameter $\psi(P_{\theta})$.
\subsection{Proof of Lemma~\ref{thm:L2-approximation}}
Let $(\Theta, \mathcal{T})$ be a measurable space with respect to the base measure $\nu$ where $\mathcal{T}$ is a $\sigma$-algebra on $\Theta$. We denote by $\mathcal{Q}= \mathcal{Q}(\Theta)$ be a collection of probability measures on $\Theta$ equipped with a density function with respect to the base measure on $\Theta$. We define $\Phi := \{\phi : \mathcal{P} \mapsto \mathbb{R}^k\}$ to be a collection of arbitrary functionals that maps from $\mathcal{P}$ to $\mathbb{R}^k$. For each $Q \in \mathcal{Q}$ and $\phi \in \Phi$, it follows that  
\begin{align*}
	&\left(\int \mathbb{E}_{\theta}\|T(X)-\psi(P_\theta)\|^2 \, dQ\right)^{1/2}
    \\
    &\qquad= \left(\int \mathbb{E}_{\theta}\|T(X)-\phi(P_\theta)+\phi(P_\theta)-\psi(P_\theta)\|^2\, dQ\right)^{1/2}\\
    &\qquad\ge  \left(\int \mathbb{E}_{\theta}\|T(X)-\phi(P_\theta)\|^{2}\, dQ\right)^{1/2}-\left(\int\|\phi(P_\theta)-\psi(P_\theta)\|^2\, dQ\right)^{1/2}
\end{align*}
by the reverse triangle inequality. When the lower bound is negative, we replace it with the trivial lower bound of zero. 
As the choice of $\phi$ was arbitrary, we conclude by taking the supremum over $\phi$ such that
\begin{align*}
&\sup_{\theta\in\Theta}\,\mathbb{E}_{\theta}\|T(X)-\psi(P_\theta)\|^2 \\
&\qquad \ge \sup_{\phi \in \Phi,\, Q\in \mathcal{Q}}\, \left[\left(\int \,\mathbb{E}_{\theta}\left\|T(X)-\phi(P_\theta)\right\|^2\, dQ\right)^{1/2} - \left(\int \|\psi(\theta)-\phi(\theta)\|^2dQ(\theta)\right)^{1/2}\right]_+^2.
\end{align*}
\subsection{Proof of Theorem~\ref{cor:vt-approximation}}\label{supp:vt-approx}
Assume that a collection of probability measures $\mathcal{Q}^\dag$ supported on $\Theta_0$ satisfies Definition \ref{as:regularprior-vt}. Note that $\mathcal{Q}^\dag$ depends on $\phi$. 
For any absolutely continuous function $\phi$ and $Q \in \mathcal{Q}^\dag$, we evoke the multivariate van Trees inequality (Theorem 12 of \citet{gassiat2013revisiting}) under the Hellinger differentiability of the statistical model $P_t$ for all $t \in \Theta_0$. For any vector norm $\|\cdot\|$, let $\|\cdot\|_*$ be its dual norm. Then, for any vector $u$ such that $\|u\|_*\le  1$, we have
	\begin{align}
		\int \,\mathbb{E}_{\theta}\left\|T(X)-\phi(\theta)\right\|^2\, dQ &\ge \int_{\Theta_0}\,\mathbb{E}_{\theta}|u^{\top}\big(T(X)-\phi(P_\theta)\big)|^2\, dQ\\
		&\ge u^\top \left(\int_{\Theta_0}\nabla\phi(t)\, dQ\right)^\top \left(\mathcal{I}(Q)+\int_{\Theta_0}\mathcal{I}(t)\, dQ\right)^{-1}\left(\int_{\Theta_0}\nabla\phi(t)\, dQ\right) u
	\end{align}
	by the van Trees inequality. Combining this result with Lemma~\ref{thm:L2-approximation}, we have 
	\begin{align*}
	\sup_{\theta\in\Theta_0}\mathbb{E}_{\theta}\|T(X)-\psi(P_\theta)\|^2 &
    \ge  \left[\left(u^\top \Gamma_{Q, \phi}u\right)^{1/2}-\left(\int\|\phi(P_\theta)-\psi(P_\theta)\|^2\, dQ\right)^{1/2}\right]_+^2
\end{align*}
where 
\begin{align}
	\Gamma_{Q, \phi}:=\left(\int_{\Theta_0} \nabla \phi(t)\, dQ\right)^{\top}\left(\mathcal{I}(Q) + \int_{\Theta_0} \mathcal{I}(t)\, dQ\right)^{-1}\left(\int_{\Theta_0}  \nabla \phi(t)\, dQ\right).
\end{align}
This holds for any $\phi \in \Phi_{\textrm{ac}}$, $Q\in\mathcal{Q}^\dag$ and $u$ such that $\|u\|_*\le 1$. We thus conclude the claim by taking supremum over them.

\clearpage
\section{Proofs of the mixture extensions of the HCR bound}
Let $\psi: \mathbb{R}^d \mapsto \mathbb{R}^k$ be the vector-valued function of interest, and let $Q$ be a prior distribution on $\mathbb{R}^d$ equipped with a density function. We define two probability measures on a product space $(\mathcal{X} \times \mathbb{R}^d)$:
 \[d\mathbb{P}_0(x, \theta) := dP_\theta(x)\, dQ(\theta)\quad\textrm{and}\quad  d\mathbb{P}_h(x,\theta) := dP_{\theta+h}(x)\, dQ(\theta+h).\]
\subsection{Proof of Theorem~\ref{thm:crvtBound}}
For each $\lambda \in [0,1]$, we define a mixture distribution as $\mathbb{P}_h^\lambda := (1-\lambda)\mathbb{P}_0 + \lambda \mathbb{P}_h$ with the corresponding density function given by 
\begin{align*}
d\mathbb{P}_h^\lambda(x, \theta) := (1-\lambda)d\mathbb{P}_0 (x,\theta) + \lambda d\mathbb{P}_h (x,\theta).
\end{align*}

Throughout the proof, we denote by $\mathbb{E}_\theta$ the expectation under $P_\theta$ with a fixed parameter $\theta$, by $\mathbb{E}_{\mathbb{P}_h}$ the expectation under the joint probability measure $\mathbb{P}_h$, and by $\mathbb{E}_{\mathbb{P}_h^\lambda}$ the expectation under the joint mixture probability measure $\mathbb{P}_h^\lambda$.
For any measurable functions $\psi : \mathbb{R}^d \mapsto\mathbb{R}^k$, $T : \mathcal{X}\mapsto\mathbb{R}^k$ and $h \in \mathbb{R}^d$, we have
\begin{align}
    \mathbb{E}_{\mathbb{P}_h} \left(T(X)-\psi(\theta)\right)
    &=\iint_{\mathcal{X}\times \mathbb{R}^d}\left(T(x)-\psi(t)\right) \,d\mathbb{P}_h(x,t) \nonumber\\
    &= \iint_{\mathcal{X}\times \mathbb{R}^d} \left(T(x)-\psi(t)\right) \, dP_{t+h}(x)\, dQ(t+h) \nonumber\\
    &=\iint_{\mathcal{X}\times (\mathbb{R}^d-h)} \left(T(x)-\psi(u-h)\right) \, dP_{u}(x)\, dQ(u) \nonumber
\end{align}
where $\mathbb{R}^d-h$ is the set $\{u-h: u\in \mathbb{R}^d\}$. In particular, we have $\mathbb{R}^d-h = \mathbb{R}^d$. We then obtain 
\begin{align}
    \mathbb{E}_{\mathbb{P}_h} \left(T(X)-\psi(\theta)\right)
    &=\int_{\mathbb{R}^d}\mathbb{E}_{t}T(X)\,dQ(t) - \int_{\mathbb{R}^d} \psi(u-h)\,dQ(u). \nonumber
\end{align}
Similarly under the mixture distribution $\mathbb{P}_h^\lambda$, we have
\begin{align}
    &\mathbb{E}_{\mathbb{P}_h^\lambda} \left(T(X)-\psi(\theta)\right)\nonumber
    \\
    &\qquad=\iint_{\mathcal{X}\times \mathbb{R}^d}\left(T(x)-\psi(t)\right) \, d\mathbb{P}_h^\lambda(x,t)\nonumber\\
    &\qquad=(1-\lambda) \iint_{\mathcal{X}\times \mathbb{R}^d} \left(T(x)-\psi(t)\right) \, d\mathbb{P}_0(x,t) + \lambda\iint_{\mathcal{X}\times \mathbb{R}^d} \left(T(x)-\psi(t)\right) \, d\mathbb{P}_h(x,t) \nonumber\\
    &\qquad=(1-\lambda) \iint_{\mathcal{X}\times \mathbb{R}^d} \left(T(x)-\psi(t)\right) \,dP_t(x)\, dQ(t) \nonumber\\
    &\qquad\qquad+ \lambda\iint_{\mathcal{X}\times (\mathbb{R}^d-h)} \left(T(x)-\psi(u-h)\right)\, dP_u(x)\, dQ(u) \nonumber\\
    &\qquad=(1-\lambda)\int_{\mathbb{R}^d}\mathbb{E}_tT(X)\, dQ(t) + \lambda\int_{\mathbb{R}^d}\mathbb{E}_tT(X)\, dQ(t)- (1-\lambda) \int_{\mathbb{R}^d} \psi(u)\, dQ(u) \nonumber\\
    &\qquad\qquad- \lambda \int_{\mathbb{R}^d} \psi(u-h)\, dQ(u)\nonumber\\
    &\qquad=\int_{\mathbb{R}^d}\mathbb{E}_tT(X)\, dQ(t) - (1-\lambda) \int_{\mathbb{R}^d} \psi(u)\, dQ(u) - \lambda \int_{\mathbb{R}^d} \psi(u-h)\, dQ(u).\nonumber
\end{align}
Therefore, it follows
\begin{align}
    \mathbb{E}_{\mathbb{P}_h} \left(T(X)-\psi(\theta)\right)-\mathbb{E}_{\mathbb{P}_h^\lambda} \left(T(X)-\psi(\theta)\right) &= (1-\lambda) \left(\int_{\mathbb{R}^d} \psi(u)\,dQ(u)-\int_{\mathbb{R}^d} \psi(u-h)\,dQ(u)\right). \label{eq:remainder_mixture_shifted}
\end{align}

Next, we consider the following ratio of the joint probability over $\mathcal{X} \times \mathbb{R}^d$ given by
\begin{align}
    D_{h,\lambda} : (x, t) \mapsto \frac{d\mathbb{P}_h(x,t)-d\mathbb{P}_h^\lambda(x, t)}{d\mathbb{P}_h^\lambda(x, t)}\nonumber.
\end{align}
It follows from \eqref{eq:remainder_mixture_shifted} that 
\begin{align}
    \mathbb{E}_{\mathbb{P}_h^\lambda} D_{h,\lambda}(X, \theta) \left(T(X)-\psi(\theta)\right) &=  \mathbb{E}_{\mathbb{P}_h} \left(T(X)-\psi(\theta)\right)-\mathbb{E}_{\mathbb{P}_h^\lambda} \left(T(X)-\psi(\theta)\right)\nonumber\\
    &=(1-\lambda) \int_{\mathbb{R}^d} \left(\psi(u)-\psi(u-h)\right)\, dQ(u). \nonumber
\end{align}
Applying Cauchy-Schwarz inequality to the dual norm of the above display, we obtain 
\begin{align}
    (1-\lambda)^2\left\|\int_{\mathbb{R}^d} \left(\psi(t)-\psi(t-h)\right)\,dQ(t)\right\|^2 
    &= \left\|\mathbb{E}_{\mathbb{P}_h^\lambda} D_{h,\lambda}(X, \theta) \left(T(X)-\psi(\theta)\right)\right\|^2 \nonumber \\
    &= \sup_{a:\|a\| \le 1} \left|\mathbb{E}_{\mathbb{P}_h^\lambda} D_{h,\lambda}(X, \theta) a^{\top}\left(T(X)-\psi(\theta)\right)\right|^2\nonumber\\
    &\le \left\{\mathbb{E}_{\mathbb{P}_h^\lambda} D_{h,\lambda}^2(X, \theta)\right\}\sup_{a:\|a\| \le 1}\left\{\mathbb{E}_{\mathbb{P}_h^\lambda}|a^{\top}(T(X)-\psi(\theta))|^2\right\}\nonumber\\
    &\le \left\{\mathbb{E}_{\mathbb{P}_h^\lambda} D_{h,\lambda}^2(X, \theta)\right\} \left\{\mathbb{E}_{\mathbb{P}_h^\lambda}\|T(X)-\psi(\theta)\|^2\right\}\label{eq:vt_cauchy_schwarz}.
\end{align}
We now analyze each term in the last expression. First, by the definition of the chi-squared divergence, we have
\begin{align*}
    \mathbb{E}_{\mathbb{P}_h^\lambda} D^2(X, \theta) &= \iint_{\mathcal{X}\times \mathbb{R}^d} \left(\frac{d\mathbb{P}_h(x,t)-d\mathbb{P}_h^\lambda(x, t)}{d\mathbb{P}_h^\lambda(x, t)}\right)^2\, d\mathbb{P}_h^\lambda(x, t) = \chi^2(\mathbb{P}_h \| \lambda\mathbb{P}_h + (1-\lambda)\mathbb{P}_0).
\end{align*}
Next, for the second term of the upper bound in \eqref{eq:vt_cauchy_schwarz}, we have
\begin{align*}
    &\mathbb{E}_{\mathbb{P}_h^\lambda}\left\|T(X)-\psi(t)\right\|^2 \\
    &\qquad=\iint_{\mathcal{X}\times \mathbb{R}^d}\left\|T(x)-\psi(t)\right\|^2 \, d\mathbb{P}_h^\lambda(x,t)\nonumber\\
    &\qquad=(1-\lambda) \iint_{\mathcal{X}\times \mathbb{R}^d} \left\|T(x)-\psi(t)\right\|^2 \,dP_t(x)\, dQ(t) \nonumber\\
    &\qquad\qquad +\lambda \iint_{\mathcal{X}\times \mathbb{R}^d} \left\|T(x)-\psi(t)\right\|^2 \,dP_{t+h}(x)\, dQ(t+h) \nonumber\\
    &\qquad=(1-\lambda) \iint_{\mathcal{X}\times \mathbb{R}^d} \left\|T(x)-\psi(t)\right\|^2 \,dP_{t}(x)\, dQ(t) \nonumber\\
    &\qquad\qquad + \lambda\iint_{\mathcal{X}\times \mathbb{R}^d} \left\|T(x)-\psi(t)-\psi(t+h)+\psi(t+h)\right\|^2 \,dP_{t+h}(x)\, dQ(t+h)\nonumber\\
    &\qquad\le (1-\lambda) \int_{\mathbb{R}^d}\mathbb{E}_t\|T(X)-\psi(t)\|^2 \, dQ(t) \nonumber\\
    &\qquad\qquad + \lambda\bigg((1+L)\int_{\mathbb{R}^d-h}\mathbb{E}_u\|T(X)-\psi(u)\|^2 \, dQ(u) \\
    &\qquad\qquad\qquad+(1+1/L)\int_{\mathbb{R}^d-h}\|\psi(u)-\psi(u-h)\|^2\,dQ(u) \bigg)\nonumber\\
    &\qquad=(1+L\lambda) \int_{\mathbb{R}^d}\mathbb{E}_t\|T(X)-\psi(t)\|^2 \, dQ(t) + \lambda(1+1/L)\int_{\mathbb{R}^d}\|\psi(t)-\psi(t-h)\|^2\,dQ(t) \nonumber
\end{align*}
where we use $(a+b)^2 \le (1+L)a^2 + (1+1/L)b^2$ for any $L \ge 0$, which follows from $2ab \le a^2L + b^2/L$. Putting all intermediate results together, the inequality \eqref{eq:vt_cauchy_schwarz} implies 
\begin{align*}
    \eqref{eq:vt_cauchy_schwarz}\implies&\frac{(1-\lambda)^2\left\|\int_{\mathbb{R}^d} \left(\psi(t)-\psi(t-h)\right)\,dQ(t)\right\|^2}{\chi^2(\mathbb{P}_h \| \lambda\mathbb{P}_h + (1-\lambda)\mathbb{P}_0)} \\
    &\qquad \le (1+L\lambda) \int_{\mathbb{R}^d}\mathbb{E}_t\|T(X)-\psi(t)\|^2 \,dQ(t) + \lambda(1+1/L)\int_{\mathbb{R}^d}\|\psi(t)-\psi(t-h)\|^2\,dQ(t)\\
    \implies&\int_{\mathbb{R}^d}\mathbb{E}_t\|T(X)-\psi(t)\|^2 \,dQ(t) \\
    &\qquad \ge \frac{1}{1+L\lambda}\bigg[\frac{(1-\lambda)^2\left\|\int_{\mathbb{R}^d} \left(\psi(t)-\psi(t-h)\right)\,dQ(t)\right\|^2}{\chi^2(\mathbb{P}_h \| \lambda\mathbb{P}_h + (1-\lambda)\mathbb{P}_0)} \\
    &\qquad\qquad- \lambda(1+1/L)\int_{\mathbb{R}^d}\|\psi(t)-\psi(t-h)\|^2\,dQ(t)\bigg]_+.
\end{align*}

Since the Bayes risk is bounded by minimax risk for any prior distribution $Q$, we have 
\begin{align*}
    &\sup_{\theta \in \mathbb{R}^d}\mathbb{E}_\theta\|T(X)-\psi(\theta)\|^2\\
    &\qquad\ge \int_{\mathbb{R}^d}\mathbb{E}_t\|T(X)-\psi(t)\|^2 \,dQ(t)\\
    &\qquad \ge \frac{(1-\lambda)^2}{1+L\lambda}\left[\frac{\left\|\int_{\mathbb{R}^d} \left(\psi(t)-\psi(t-h)\right)\,dQ(t)\right\|^2}{\chi^2(\mathbb{P}_h \| \lambda\mathbb{P}_h + (1-\lambda)\mathbb{P}_0)}- \frac{\lambda(1+1/L)}{(1-\lambda)^2}\int_{\mathbb{R}^d}\|\psi(t)-\psi(t-h)\|^2\,dQ(t) \right]_+.
\end{align*}
As the introduced parameter $L\ge 0$ is arbitrary, we can optimize its choice. First, we observe that the lower bound is in the form 
\[\frac{(1-\lambda)^2}{1+L\lambda}\left[A- \frac{\lambda(1+1/L)}{(1-\lambda)^2}B \right]_+.\]
where
\begin{align}
	A := \frac{\left\|\int_{\mathbb{R}^d} \left(\psi(t)-\psi(t-h)\right)\,dQ(t)\right\|^2}{\chi^2(\mathbb{P}_h \| \lambda\mathbb{P}_h + (1-\lambda)\mathbb{P}_0)} \quad\textrm{and}\quad B:=\int_{\mathbb{R}^d}\|\psi(t)-\psi(t-h)\|^2\,dQ(t) \nonumber.
\end{align}
When $\lambda = 0$, the above display takes $A^2$ regardless of the value of $L$. The argument for $B=0$ is also similar. Hence we focus on the case with $\lambda, B > 0$. First we observe that 
\begin{align}
	\frac{(1-\lambda)^2}{1+L\lambda}\left[A- \frac{\lambda(1+1/L)}{(1-\lambda)^2}B \right]_+ = \frac{\lambda^2 B}{1+L\lambda}\left[\left(\frac{(1-\lambda)^2A}{\lambda ^2 B}- \frac{1}{\lambda}\right)-\frac{1}{L\lambda}\right]_+ = \frac{\lambda^2 B}{1+L\lambda}\left[\Gamma-\frac{1}{L\lambda}\right]_+\nonumber
\end{align}
where 
\[\Gamma := \frac{(1-\lambda)^2A}{\lambda ^2 B}- \frac{1}{\lambda}.\]
The optimization of $L\ge 0$ is thus equivalent to 
\begin{align}
	\max_{\ell\ge 0}\, \left\{\frac{1}{1+\ell}\left(\Gamma-\frac{1}{\ell}\right)\right\}.\nonumber
\end{align}
We later check if the attained maximum is positive otherwise replace the optima with zero. The optimal value of $\ell$ is given by 
\begin{align}
	\ell^* = \frac{1+\sqrt{1+\Gamma}}{\Gamma} \quad \textrm{when} \quad \Gamma > 0.\nonumber
\end{align}  
The corresponding optima is 
\begin{align}
	\frac{1}{1+\ell^*}\left(\Gamma-\frac{1}{\ell^*}\right) &= \frac{\Gamma}{\Gamma+1+\sqrt{1+\Gamma}}\left(\Gamma-\frac{\Gamma}{1+\sqrt{1+\Gamma}}\right)\nonumber\\
	&= \frac{\Gamma^2\sqrt{1+\Gamma}}{\sqrt{1+\Gamma}\left(\sqrt{\Gamma+1}+1\right)^2}\nonumber\\
	&= \frac{\left(\sqrt{\Gamma+1}-1\right)^2\left(\sqrt{\Gamma+1}+1\right)^2}{\left(\sqrt{\Gamma+1}+1\right)^2}\nonumber\\
	&= \left(\sqrt{\Gamma+1}-1\right)^2\nonumber.
\end{align}
This is non-negative when $\Gamma>0$ and thus it is a valid optimum for $(1+\ell)^{-1}[\Gamma-1/\ell]_+$. Plugging this result into the original expression, we obtain 
\begin{align}
	\lambda^2 B\left(\sqrt{ \frac{(1-\lambda)^2A}{\lambda ^2 B}- \frac{1}{\lambda}+1}-1\right)^2 =\left(\sqrt{ (1-\lambda)\left\{A-\lambda(A+B)\right\}}- \sqrt{\lambda^2 B}\right)^2.\nonumber
\end{align}
When $\Gamma \le 0$, or equivalently $(1-\lambda)^2 A \le \lambda B$, the lower bound becomes zero. As $T$ and $h$ only appear on one side of the inequality, we conclude the claim by taking the infimum over  $T$  and the supremum over $h\in\mathbb{R}^d$.

\subsection{Proof of Theorem~\ref{thm:crvtBound_hellinger}}
Following the analogous derivation leading up to equation \eqref{eq:remainder_mixture_shifted} given by the proof of Theorem~\ref{thm:crvtBound} with the special case of $\lambda=0$, we obtain
\begin{align}
    \mathbb{E}_{\mathbb{P}_h} \left(T(X)-\psi(\theta)\right)-\mathbb{E}_{\mathbb{P}_0} \left(T(X)-\psi(\theta)\right) &= \int_{\mathbb{R}^d} \left(\psi(u)-\psi(u-h)\right)\,dQ(u). \label{eq:remainder_shifted}
\end{align}
Next, we consider the density ratio of the joint probability over $(\mathcal{X} \times \mathbb{R}^d)$ given by
\begin{align}
    D_h := (x, t) \mapsto \frac{d\mathbb{P}_h(x,t)-d\mathbb{P}_0(x,t)}{d\mathbb{P}_0(x,t)}\nonumber.
\end{align}
By the application of Cauchy-Schwarz inequality to \eqref{eq:remainder_shifted}, we obtain 
\begin{align}
    \left\|\int_{\mathbb{R}^d} \left(\psi(u)-\psi(u-h)\right)\,dQ(u)\right\|^2 &=\left\|\mathbb{E}_{\mathbb{P}_0} D_h(X, \theta) \left(T(X)-\psi(\theta)\right)\right\|^2\nonumber\\
    &=\left\|\mathbb{E}_{\mathbb{P}_0} \left(\frac{d\mathbb{P}_h(X,\theta)-d\mathbb{P}_0(X,\theta)}{d\mathbb{P}_0(X,\theta)}\right) \left(T(X)-\psi(\theta)\right)\right\|^2 \nonumber \\
    &= \left\|\mathbb{E}_{\mathbb{P}_0} \left(\sqrt{\frac{d\mathbb{P}_h(X,\theta)}{d\mathbb{P}_0(X,\theta)}}-1\right)\left(\sqrt{\frac{d\mathbb{P}_h(X,\theta)}{d\mathbb{P}_0(X,\theta)}}+1\right) \left(T(X)-\psi(\theta)\right)\right\|^2 \nonumber \\
    &\le \left\{\mathbb{E}_{\mathbb{P}_0} \left(\sqrt{\frac{d\mathbb{P}_h(X,\theta)}{d\mathbb{P}_0(X,\theta)}}+1\right)^2 \left\|T(X)-\psi(\theta)\right\|^2 \right\}H^2(\mathbb{P}_0, \mathbb{P}_h)\nonumber \\
    &\le 2\left(\mathbb{E}_{\mathbb{P}_0} \|T(X)-\psi(\theta)\|^2 + \mathbb{E}_{\mathbb{P}_h} \|T(X)-\psi(\theta)\|^2\right)H^2(\mathbb{P}_0, \mathbb{P}_h) \label{eq:vt_cauchy_schwarz_hellinger}
\end{align}
where we apply the inequality $(a+b)^2 \le 2(a^2+b^2)$ in the last step. We now analyze two terms inside the parenthesis of the above display. Since minimax risk gives an upper bound of Bayes risk, we have 
\[\mathbb{E}_{\mathbb{P}_0} \|T(X)-\psi(\theta)\|^2  = \int_{\mathbb{R}^d}\mathbb{E}_{t} \|T(X)-\psi(t)\|^2 \, dQ(t)\le \sup_{\theta \in \mathbb{R}^d} \mathbb{E}_\theta \|T(X)-\psi(\theta)\|^2.\]
For the second term, we have
\begin{align*}
    &\mathbb{E}_{\mathbb{P}_h} \left\|T(X)-\psi(\theta)\right\|^2 \\
    &\qquad = \int_{\mathbb{R}^d}\mathbb{E}_{t+h}\left\|T(X)-\psi(t)\right\|^2\,dQ(t+h) \\
    &\qquad =\int_{\mathbb{R}^d}\mathbb{E}_{t+h}\left\|T(X)-\psi(t)-\psi(t+h)+\psi(t+h)\right\|^2\,dQ(t+h) \\
    &\qquad \le \int_{\mathbb{R}^d}\mathbb{E}_{t+h}\left\|T(X)-\psi(t+h)\right\|^2\,dQ(t+h) \\
    &\qquad\qquad + \int_{\mathbb{R}^d}\|\psi(t+h)-\psi(t)\|^2 \,dQ(t+h) \\
    &\qquad\qquad + 2\sqrt{\int_{\mathbb{R}^d}\mathbb{E}_{t+h}\left\|T(X)-\psi(t+h)\right\|^2\,dQ(t+h) }\, \sqrt{\int_{\mathbb{R}^d} \left\|\psi(t+h)-\psi(t)\right\|^2\,dQ(t+h)}\\
    &\qquad=\left(\sqrt{\int_{\mathbb{R}^d}\mathbb{E}_{t+h}\left\|T(X)-\psi(t+h)\right\|^2\,dQ(t+h) }+\sqrt{\int_{\mathbb{R}^d} \left\|\psi(t+h)-\psi(t)\right\|^2\,dQ(t+h)}\right)^2
\end{align*}
where the inequality follows from Cauchy-Schwarz inequality. Therefore \eqref{eq:vt_cauchy_schwarz_hellinger} implies the following:
\begin{align*}
     &\eqref{eq:vt_cauchy_schwarz_hellinger} \implies\frac{\left\|\int_{\mathbb{R}^d} \left(\psi(u)-\psi(u-h)\right)\,dQ(u)\right\|^2}{2H^2(\mathbb{P}_0, \mathbb{P}_h)} \\
    &\qquad\le\sup_{\theta \in\mathbb{R}^d}\mathbb{E}_{\theta} \|T(X)-\psi(\theta)\|^2\\
    &\qquad\qquad+ \left( \sqrt{\int_{\mathbb{R}^d}\mathbb{E}_{t+h}\|T(X)-\psi(t+h)\|^2\,dQ(t+h) }+\sqrt{\int_{\mathbb{R}^d} \left\|\psi(t+h)-\psi(t)\right\|^2\,dQ(t+h)}\right)^2\\
    &\qquad\le\sup_{\theta \in\mathbb{R}^d}\mathbb{E}_{\theta} \|T(X)-\psi(\theta)\|^2\\
    &\qquad\qquad+ \left( \sqrt{\sup_{\theta \in\mathbb{R}^d}\mathbb{E}_{\theta}\|T(X)-\psi(\theta)\|^2}+\sqrt{\int_{\mathbb{R}^d-h} \left\|\psi(t)-\psi(t-h)\right\|^2\,dQ(t)}\right)^2\\
        &\qquad\le2\left( \sqrt{\sup_{\theta \in\mathbb{R}^d}\mathbb{E}_{\theta}\|T(X)-\psi(\theta)\|^2}+\sqrt{\int_{\mathbb{R}^d} \left\|\psi(t)-\psi(t-h)\right\|^2\,dQ(t)}\right)^2
\end{align*}
where the last step used the fact that $\sqrt{\int_{\mathbb{R}^d} \left\|\psi(t)-\psi(t-h)\right\|^2\,dQ(t)} \ge 0$.
We thus obtain 
\begin{align*}
&\frac{\left\|\int_{\mathbb{R}^d} \left(\psi(u)-\psi(u-h)\right)\,dQ(u)\right\|^2}{2H^2(\mathbb{P}_0, \mathbb{P}_h)} \\
&\qquad \le2\left( \sqrt{\sup_{\theta \in\mathbb{R}^d}\mathbb{E}_{\theta}\|T(X)-\psi(\theta)\|^2}+\sqrt{\int_{\mathbb{R}^d} \|\psi(t)-\psi(t-h)\|^2\,dQ(t)}\right)^2
\end{align*}
\begin{align*}
    \implies&\sup_{\theta \in\mathbb{R}^d}\mathbb{E}_{\theta}\|T(X)-\psi(\theta)\|^2 \\
    &\qquad\ge \left[\frac{\left\|\int_{\mathbb{R}^d} \left(\psi(u)-\psi(u-h)\right)\,dQ(u)\right\|}{2H(\mathbb{P}_0, \mathbb{P}_h)} -\left(\int_{\mathbb{R}^d} \left\|\psi(t)-\psi(t-h)\right\|^2\,dQ(t)\right)^{1/2}\right]^2_+.
\end{align*}
As with the proof of Theorem~\ref{thm:crvtBound}, the function $T$ and $h$ only appear on one side of the inequality, and we conclude the claim by taking the infimum over $T$ and the supremum over $h \in \Theta$.


\subsection{Proof of Lemma~\ref{lemma:sharper_ihbound}}
The main idea is provided by Theorem 6.1 of \citet{ibragimov1981statistical}. Here, we derive the improved constant. We focus on the estimation of the real-valued functional $\psi : \Theta \mapsto \mathbb{R}$ for $\Theta \subseteq \mathbb{R}^d$. An analogous proof can be applied to each specific choice of the vector norm $\|\cdot\|$ (See remark 6.1 of \citet{ibragimov1981statistical}).
For any real-valued measurable function $T := \mathcal{X}\mapsto\mathbb{R}$, two points in the parameter space $\{\theta, \theta+h\}\in\Theta$, and an arbitrary scalar constant $C$, we have 
\begin{align}
    \left|\mathbb{E}_{\theta+h} T(X) - \mathbb{E}_{\theta} T(X)\right|^2 &= \left|\int_\mathcal{X} (T(x)-C) \, \left(dP_{\theta+h}(x)-dP_\theta(x)\right)\right|^2\nonumber\\
    &=\left|\int_\mathcal{X}  (T(x)-C) \left(dP^{1/2}_{\theta+h}(x)-dP^{1/2}_\theta(x)\right)\left(dP^{1/2}_{\theta+h}(x)+dP^{1/2}_\theta(x)\right)\right|^2\nonumber\\
    &\le2H^2(P_{\theta+h}, P_\theta)\left(\int_\mathcal{X}  |T(x)-C|^2  \, dP_{\theta+h}(x)+\int_\mathcal{X}  |T(x)-C|^2 \, dP_{\theta}(x)\right).\label{eq:proofB2_CS_ineq}
\end{align}
where the last inequalities follow by Cauchy–Schwarz inequality and $(a+b)^2 \le 2(a^2 + b^2)$. The first integral on the right-hand side can be written out as
\begin{align*}
    \int_\mathcal{X}  \left|T(x)-C\right|^2 \,dP_{\theta+h}(x)&=\int_\mathcal{X}  |T(x)-\mathbb{E}_{\theta+h}T(X)+\mathbb{E}_{\theta+h}T(X)-C|^2 \,dP_{\theta+h}(x) \\
    &=\int_\mathcal{X}|T(x)-\mathbb{E}_{\theta+h}T(X)|^2\,dP_{\theta+h}(x)+ |\mathbb{E}_{\theta+h}T(X)-C|^2.
\end{align*}
We also have the following standard bias-variance decomposition:
\begin{align*}
    \int_\mathcal{X}|T(x)-\mathbb{E}_{\theta+h}T(X) |^2 \,dP_{\theta+h}(x) &= \mathbb{E}_{\theta+h}|T(X)- \psi(\theta+h)|^2 - |\mathbb{E}_{\theta+h}T(X)-\psi(\theta+h)|^2.
\end{align*}
Putting together, we obtain
\begin{align*}
    &\int_\mathcal{X}  \left|T(x)-C\right|^2 \,dP_{\theta+h}(x)\\
    &\qquad=\mathbb{E}_{\theta+h}|T(X)- \psi(\theta+h)|^2 - |\mathbb{E}_{\theta+h}T(X)-\psi(\theta+h)|+ |\mathbb{E}_{\theta+h}T(X)-C|^2.
\end{align*}
By repeating the analogous argument for $\int_\mathcal{X}  |T(x)-C|^2 \,dP_{\theta}(x)$ and plugging them into \eqref{eq:proofB2_CS_ineq}, we obtain 
\begin{align*}
     &\left|\mathbb{E}_{\theta+h} T(X) - \mathbb{E}_{\theta} T(X)\right|^2 \\
     &\qquad \le2H^2(P_{\theta+h}, P_\theta)\left(\mathbb{E}_{\theta+h}|T(X)-\psi(\theta+h)|^2 -|\mathbb{E}_{\theta+h}T(X) - \psi(\theta+h)|^2\right.\\
     &\qquad \qquad+ |\mathbb{E}_{\theta+h}T(X)-C|^2 + \mathbb{E}_{\theta}|T(X)-\psi(\theta)|^2 -|\mathbb{E}_{\theta}T(X) - \psi(\theta)|^2+ |\mathbb{E}_{\theta}T(X)-C|^2\left.\right).
\end{align*}
Since the above inequality holds for an arbitrary scalar constant $C$, we choose $C$ to minimize the upper bound. As the optimal $C^*$ is attained by $C^* = 1/2\left(\mathbb{E}_{\theta}T(X)+\mathbb{E}_{\theta+h}T(X)\right)$, we can further simplify the expression as 
\begin{align}
     \left|\mathbb{E}_{\theta+h} T(X) - \mathbb{E}_{\theta} T(X)\right|^2 &\le2H^2(P_{\theta+h}, P_\theta)\left(\mathbb{E}_{\theta+h}|T(X)-\psi(\theta+h)|^2 +\mathbb{E}_{\theta}|T(X)-\psi(\theta)|^2\right.\nonumber \\
     &\qquad  \left.-|d(\theta+h)|^2-|d(\theta)|^2+ \frac{\left|\mathbb{E}_{\theta}T(X)-\mathbb{E}_{\theta+h}T(X)\right|^2}{2}  \right)\, .\label{eq:middle_step_IH}
\end{align}
where $d(t):=\mathbb{E}_t T(X)-\psi(t)$.
Hence, the above display immediately implies the following:
\begin{align*}
    \eqref{eq:middle_step_IH}&\implies \mathbb{E}_{\theta+h}|T(X)-\psi(\theta+h)|^2 +\mathbb{E}_{\theta}|T(X)-\psi(\theta)|^2 \\
     &\qquad\ge \frac{1-H^2(P_{\theta+h}, P_\theta)}{2H^2(P_{\theta+h}, P_\theta)}\left|\mathbb{E}_{\theta+h} T(X) - \mathbb{E}_{\theta} T(X)\right|^2+ |d(\theta+h)|^2+|d(\theta)|^2\\
     &\implies \mathbb{E}_{\theta+h}|T(X)-\psi(\theta+h)|^2 +\mathbb{E}_{\theta}|T(X)-\psi(\theta)|^2 \\
     &\qquad \ge \frac{1-H^2(P_{\theta+h}, P_\theta)}{2H^2(P_{\theta+h}, P_\theta)}\left| \psi(\theta+h)-\psi(\theta)+
     d(\theta+h)-d(\theta)\right|^2+ |d(\theta+h)|^2+|d(\theta)|^2.
\end{align*}
The existing proofs by \cite{ibragimov1981statistical} and \cite{lin2019optimal} proceed by splitting the analysis into two cases: (i) $\max\{|d(\theta+h)|, |d(\theta)|\} < 1/4|\psi(\theta+h)-\psi(\theta)|$ and (ii) $|d(\theta)| > 1/4|\psi(\theta+h)-\psi(\theta)|$ where the leading constant $1/4$ for the boundary was chosen for convenience by \citet{ibragimov1981statistical} as they did not focus on the optimal constant. We instead optimize this boundary to obtain a sharper constant.

We now define the following function:
\[
\eta(x, y) := A|B + x - y|^2 + |x|^2 + |y|^2
\]
with $A := (1 - H^2(P_{\theta + h}, P_{\theta}))/(2H^2(P_{\theta + h}, P_{\theta}))$ and $B := \psi(\theta + h) - \psi(\theta)$. The lower bound can be written as
\[
E_{\theta + h}|T(X) - \psi(\theta + h)|^2 + E_{\theta}|T(X) - \psi(\theta)|^2 \ge \eta(d(\theta + h), d(\theta)).
\]
This further implies that
\[
E_{\theta + h}|T(X) - \psi(\theta + h)|^2 + E_{\theta}|T(X) - \psi(\theta)|^2 \ge \min_{x, y\in\mathbb{R}}\,\eta(x, y).
\]
The minimizer of $\eta$ is given by $x^* = -y^* = -\frac{2AB}{4A + 2}$ and hence we have,
\[
\min_{x, y}\eta(x, y) = \frac{4AB^2}{(4A + 2)^2} + \frac{8A^2B^2}{(4A + 2)^2} = \frac{4AB^2(1 + 2a)}{(4A + 2)^2} = \frac{AB^2}{2A + 1}.
\]
Additionally, we have
\begin{align*}
    \partial_{xx}\eta = 2A+2,\quad  \partial_{xy}\eta = -2A,\quad \text{ and }\quad \partial_{yy}\eta = 2A+2
\end{align*}
and the discriminant is given by $(2A+2)^2-4A^2 = 4(2A+1)$. Therefore, the optima is at a saddle point when $A < -1/2$. The constant $A$ is defined as $(1 - H^2(P_{\theta + h}, P_{\theta}))/(2H^2(P_{\theta + h}, P_{\theta}))$, which is monotone decreasing with respect to $H^2(P_{\theta + h}, P_{\theta})$ and attains the minimum $-1/4$ at $H^2(P_{\theta + h}, P_{\theta})=2$. Thus the minimizer of $\eta$ is well-defined for all values of $A$ and thereby all values of $H^2(P_{\theta + h}, P_{\theta})$. Putting it together, the minimum of the function $\eta$ with respect to $d(\theta + h)$ and $d(\theta)$ is given by 
\begin{align*}
    \frac{AB^2}{2A+1} &= \frac{1-H^2(P_{\theta+h}, P_\theta)}{2H^2(P_{\theta+h}, P_\theta)}\left| \psi(\theta+h)-\psi(\theta)\right|^2  \left(\frac{1-H^2(P_{\theta+h}, P_\theta)}{H^2(P_{\theta+h}, P_\theta)}+1\right)^{-1} \\
    &= \frac{1-H^2(P_{\theta+h}, P_\theta)}{2}\left| \psi(\theta+h)-\psi(\theta)\right|^2.
\end{align*}

When $H^2(P_{\theta+h}, P_\theta) > 1$, the leading constant becomes negative and we can replace it with a trivial lower bound of zero. 
Therefore, we conclude
\[
 \mathbb{E}_{\theta + h}|T(X) - \psi(\theta + h)|^2 + \mathbb{E}_{\theta}|T(X) - \psi(\theta)|^2 \ge  \left[\frac{1 -H^2(P_{\theta + h}, P_{\theta})}{2}\right]_+|\psi(\theta + h) - \psi(\theta)|^2.
\]
Since minimax risk gives an upper bound of the average of risks at arbitrary two points, we have
\begin{align*}
    \sup_{t\in\Theta}\, \mathbb{E}_{t}|T(X) - \psi(t)|^2 & \ge \sup_{\{\theta,\theta+h\} \in \Theta}\,\frac{\mathbb{E}_{\theta + h}|T(X) - \psi(\theta + h)|^2 + \mathbb{E}_{\theta}|T(X) - \psi(\theta)|^2}{2}\\
    &\ge  \sup_{\{\theta,\theta+h\} \in \Theta}\,\left[\frac{1 -H^2(P_{\theta + h}, P_{\theta})}{4}\right]_+|\psi(\theta + h) - \psi(\theta)|^2.
\end{align*}
As $T$ and $h$ only appear on one side of the inequality, we conclude the claim by taking the infimum over $T$ and the supremum over any pair of parameters $\theta, \theta+h \in \Theta$. 
An analogous proof can be extended to a vector-valued functional with a general vector norm $\|\cdot\|$. We only provide a result with a real-valued functional since the optimal constant depends on the choice of the norm as it requires the derivative with respect to the norm. 

\clearpage
\section{Additional technical facts about divergence metrics}\label{sec:local-behavior}
This manuscript primarily focuses on two divergence metrics, namely the chi-squared divergence and the Hellinger distance. This section discusses sufficient conditions that imply certain limiting behaviors of these metrics over a parametric path as it passes through $P_0$. The base measure for the sample space $\mathcal{X}$ is denoted by $\nu$. 
\subsection{Regularity conditions}
 We first extend the standard notion of absolute continuity of a univariate function to a multivariate function as follows:
\begin{definition}[Multivariate absolute continuity]\label{def:multi-abs-cts}
A function $\omega : \mathbb{R}^d\mapsto \mathbb{R}$ is absolutely continuous over an open $\mathbb{R}^d$-ball $B(\theta, \delta)$ if for all  $u \in \mathbb{S}^{d-1}$, the induced univariate function $t \mapsto \omega(\theta + tu)$ is absolutely continuous over $0 < t < \delta$. 
\end{definition}

The concept of multivariate absolute continuity has been explored in various real analysis literature. For instance, \citet{maly1999absolutely} and \citet{hencl2004notes} extend the classical $\delta$-$\varepsilon$ definition of absolute continuity by considering the oscillation of the functions over $d$-dimensional balls. Additionally, \citet{vsremr2010absolutely} presents a similar extension using $d$-dimensional hyper-cubes. We do not claim that Definition~\ref{def:multi-abs-cts} is the most general definition of multivariate absolutely continuous functions. 

We introduce the following regularity conditions:
\begin{enumerate}[label=\textbf{(A\arabic*)},leftmargin=2cm]
\setcounter{enumi}{1}
\item There exists $\delta >0$ such that \label{as:regular_density_iso}
\begin{enumerate}
    \item for $\nu$-almost everywhere, the mapping $t \mapsto dP_{t}$ is absolutely continuous over an open $\mathbb{R}^d$-ball $B([0], \delta)$ with the gradient $t \mapsto u^{\top}\dot{\rho}_{t, u}$ for each $u \in \mathbb{S}^{d-1}$,
    \item for all $u \in \mathbb{S}^{d-1}$, the gradient is continuous in $t$ such that $\lim_{\|t\| \longrightarrow 0}u^{\top}\dot{\rho}_{t, u} = u^{\top}\dot{\rho}_0$ for $\nu$-almost everywhere, and
    \item for all $0 < t < \delta$ and $u \in \mathbb{S}^{d-1}$, $dP_0(x) = 0$ implies $u^{\top}\dot{\rho}_{t, u}(x)=0$, and
    \begin{align*}
        \int \sup_{0 < t_1,t_2 < \delta}u^{\top}\left(\frac{\dot{\rho}_{t_1, u}\dot{\rho}_{t_2, u}^{\top}}{dP_0}\right) u\, d\nu < \infty.
    \end{align*}
\end{enumerate}
\item There exits $\delta >0$ such that \label{as:regular_root_density_iso}
\begin{enumerate}
    \item for $\nu$-almost everywhere, the mapping $t \mapsto dP^{1/2}_{t}$ is absolutely continuous over an open $\mathbb{R}^d$-ball $B([0], \delta)$ with the gradient $t \mapsto u^{\top}\dot{\gamma}_{t, u}$ for each $u \in \mathbb{S}^{d-1}$,
    \item for all $u \in \mathbb{S}^{d-1}$, the gradient is continuous such that $\lim_{\|t\| \longrightarrow 0}u^{\top}\dot{\gamma}_{t, u} = u^{\top}\dot{\gamma}_0$ for $\nu$-almost everywhere, and 
    \begin{align*}
        \int \sup_{0 < t_1,t_2 < \delta}u^\top \dot{\gamma}_{t_1, u}\dot{\gamma}_{t_2, u}^{\top}u \, d\nu< \infty.
    \end{align*}
\end{enumerate}
\item The statistical model $\{P_t : t \in\Theta\}$ is Hellinger differentiable.\label{as:hellinder-diff}
\end{enumerate}

We claim that 
\[\ref{as:regular_density_iso} \implies  \ref{as:regular_root_density_iso} \implies \ref{as:hellinder-diff},\]
meaning that \ref{as:regular_density_iso} is most restrictive and \ref{as:hellinder-diff} is most general.
\begin{proof}[\bfseries{Proof of \ref{as:regular_density_iso} $\implies$ \ref{as:regular_root_density_iso}}]
Given a small positive scalar $\eta > 0$, we define $dP_{t,\eta} := dP_{t}+\eta$. The resulting object is no longer a probability density since it does not integrate to one. Since $dP_{t,\eta}$ is bounded away from zero, the gradient of $dP^{1/2}_{t,\eta}$ exists and is given by $\frac{1}{2}\left\{\dot{\rho}_{t, u}/(dP_{t}+\eta)^{1/2}\right\}$. 
It now follows for any $0 < b < \delta$ that 
\begin{align*}
    dP^{1/2}_{t,\eta}(ub) - dP^{1/2}_{t,\eta}(0) = \int_0^b\frac{u^{\top}\dot{\rho}_{t, u}}{2(dP_{t}+\eta)^{1/2}}\, dt= \int_0^b\frac{u^{\top}\dot{\rho}_{t, u}}{2(dP_{t}+\eta)^{1/2}}I(dP_{t} > 0)\, dt.
\end{align*}
The second equality follows by \ref{as:regular_density_iso}(c) as $dP_0(x) = 0$ implies $u^{\top}\dot{\rho}_{t, u}(x)=0$, which justifies to insert the indicator $I(dP_{t} > 0)$ inside the integral. Assuming for a moment that we can invoke the dominated convergence theorem, it follows that 
\begin{align*}
    \lim_{\eta \longrightarrow 0}\left\{dP^{1/2}_{t,\eta}(ub) - dP^{1/2}_{t,\eta}(0)\right\} = \int_0^b\lim_{\eta \longrightarrow 0}\frac{u^{\top}\dot{\rho}_{t, u}}{2(dP_{t}+\eta)^{1/2}}I(dP_{t} > 0)\, dt = \int_0^b\frac{u^{\top}\dot{\rho}_{t, u}}{2dP^{1/2}_{t}}I(dP_{t} > 0)\, dt.
\end{align*}
Hence, we conclude \ref{as:regular_root_density_iso}(a) with
\begin{align}\label{eq:gradient}
    \dot{\gamma}_{t, u} = \frac{\dot{\rho}_{t, u}}{2dP^{1/2}_{t}}I(dP_{t} > 0).
\end{align}
It thus remains to check the condition for the dominated convergence theorem. This follows since we have that
\begin{align*}
    \int_0^b\left|\frac{u^{\top}\dot{\rho}_{t, u}}{2(dP_{t}+\eta)^{1/2}}I(dP_{t} > 0)\right|\, dt &\le \int_0^b\left|\frac{u^{\top}\dot{\rho}_{t, u}}{2(dP_{t})^{1/2}}I(dP_{t} > 0)\right|\, dt \\
    &\le \left(\frac{1}{2}\int_0^b\frac{u^{\top}\dot{\rho}_{t, u}\dot{\rho}_{t, u}^{\top}u}{dP_{t}}I(dP_{t} > 0)\, dt\right)^{1/2}.
\end{align*}
The last term is integrable by \ref{as:regular_density_iso}(c) and the dominated convergence theorem holds. The second statement \ref{as:regular_root_density_iso}(b) follows directly from \eqref{eq:gradient}, \ref{as:regular_density_iso}(b) and \ref{as:regular_density_iso}(c).
\end{proof}

\begin{proof}[\bfseries{Proof of \ref{as:regular_root_density_iso} $\implies$ \ref{as:hellinder-diff}}]
By the absolute continuity of the square root density, it follows
\begin{align*}
    dP_t^{1/2}- dP_0^{1/2} = \int_0^t \dot\gamma_{s, \mathrm{sign}(t)} \, ds= \int_0^1 \dot\gamma_{t\widetilde{s}, \mathrm{sign}(t)}t \, d\widetilde{s}
\end{align*}
and hence we have
\begin{align*}
    \int\left(dP_t^{1/2}- dP_0^{1/2}\right)^2 \, d\nu &= \int\left(\int_0^1 \dot\gamma_{t\widetilde{s}, \mathrm{sign}(t)}t \, d\widetilde{s}\right)^2\, d\nu \\
    &= t^{\top}\left\{ \int\left(\int_0^1\int_0^1 \dot\gamma_{t\widetilde{s}_1, \mathrm{sign}(t)}\dot\gamma_{t\widetilde{s}_2, \mathrm{sign}(t)}^{\top}\, d\widetilde{s}_1\, d\widetilde{s}_2\right)\, d\nu \right\}\,t.
\end{align*}
By the assumption that $\lim_{s\longrightarrow 0}\dot\gamma_{s, u} = \dot\gamma_{0}$ for $\nu$-almost where and uniform integrability $s \mapsto u^\top \dot\gamma_{s, u}$, we conclude 
\begin{align*}
    \int\left(dP_t^{1/2}- dP_0^{1/2}\right)^2 \, d\nu =t^{\top} \left(\int\dot\gamma_{0}\dot\gamma_{0}^{\top}\, d\nu\right)t + o(\|t\|^2)
\end{align*}
as $\|t\| \longrightarrow 0$. This concludes the claim that the local path $t \mapsto P_t$ is Hellinger differentiable at $t=0$ with $\dot \xi_0 = \dot \gamma_0$, $\nu$-almost everywhere. As a result, we have 
\[\mathcal{I}(0) = 4 \int \dot \xi_0\dot \xi_0^{\top}\, d\nu = 4 \int \dot \gamma_0\dot \gamma_0^{\top}\, d\nu.\]
\end{proof}
Additionally, the converse is not true. For instance, location families with compact support fail to satisfy \ref{as:regular_density_iso}(c), while they may satisfy \ref{as:regular_root_density_iso} under certain conditions. See Example 7.1 of \cite{polyanskiy2022information} for more details. There is also an example where \ref{as:hellinder-diff} holds but \ref{as:regular_root_density_iso} fails. See Remark \ref{remark:fisher-defect} in the following subsection.

\subsection{Supporting lemmas on divergence metrics}\label{sec:supplement-on-divs}
In this section, we provide the supporting lemmas related to the local behaviors of the chi-squared divergence and the Hellinger distance. We provide the derivation for completeness and to unify terminology between \cite{pollardAsymptotia} and \cite{polyanskiy2022information}. Throughout this section, we use the following notation consistently. Let $\varphi : \mathbb{R}^d \mapsto \Theta$ be a continuously differentiable mapping and $\|\nabla \varphi\|_\infty < C$ for some universal constant. Unless specified otherwise, we denote the mixture probability measures by
\begin{align*}
    d\widetilde{\mathbb{P}}^n_0(x,t) := dP^n_{\varphi(t)}(x)\,dQ(t) \quad \mbox{and} \quad  d\widetilde{\mathbb{P}}^n_h(x,t) := dP^n_{\varphi(t+h)}(x)\, dQ(t+h),
\end{align*}
where $P_{\varphi(t)}^n$ is an $n$-fold product measure of $P_{\varphi(t)}$, $Q(\cdot)$ is a probability measure on $\mathbb{R}^d$ with a bounded and absolutely continuous density function $q$ with respect to the base measure on $\Theta$.

Under the regularity conditions from the previous subsection, the following local expansion of divergence metrics can be obtained:
\begin{lemma}\label{lemma:chi-convergence1}
Assume that the paths $h \mapsto P_{\varphi(t+h)}$ satisfies \ref{as:regular_density_iso} and $\nabla \varphi$ is continuous at $t$,
    \begin{align*}
    \chi^2\left(P_{\varphi(t+h)}, P_{\varphi(t)}\right)& = h^{\top} \nabla\varphi(t)^{\top} \mathcal{I}(\varphi(t)) \nabla\varphi(t) h^{\top} + o(\|h\|^2).
\end{align*}
\begin{lemma}\label{lemma:chi-convergence2}
Assume that the paths $h \mapsto P_{\varphi(t+h)}$ satisfies \ref{as:regular_root_density_iso} and $\nabla \varphi$ is continuous at $t$, 
    \begin{align*}
        &\chi^2(\widetilde{\mathbb{P}}^n_h \| \lambda\widetilde{\mathbb{P}}^n_h + (1-\lambda)\widetilde{\mathbb{P}}^n_0) \\
        &\qquad = (1-\lambda)^2h^{\top}\left\{\mathcal{I}(Q)+n\int \nabla\varphi(t)^{\top}\, \left(\mathcal{I}(\varphi(t))+ \frac{1-4\lambda}{4\lambda} \mathcal{I}^\dag(\varphi(t))\right)\,\nabla\varphi(t) \,dQ(t)\right\}h+ o(\|h\|^2).
    \end{align*}
\end{lemma}
where $\mathcal{I}^\dag(0) := 4\int \dot{\gamma}_{0}\,\dot{\gamma}_{0}^{\top} I(dP_0=0)\, d\nu$ is known as the Fisher defect.
\end{lemma} 

\begin{lemma}\label{lemma:local_behavior_hellinger}
Assume that the paths $h \mapsto dP_{\varphi(t+h)}$ satisfies \ref{as:hellinder-diff} and $\nabla \varphi$ is continuous at $t$,
\begin{align*}
    H^2\left(P_{\varphi(t+h)}, P_{\varphi(t)}\right)& = \frac{1}{4}h^{\top} \nabla\varphi(t)^{\top} \mathcal{I}(\varphi(t)) \nabla\varphi(t) h^{\top} + o(\|h\|^2).
\end{align*}
\end{lemma}

The main takeaway of the three lemmas is as follows: The local behavior of the Hellinger distance can be established under weaker conditions, while an analogous result under the chi-squared divergence requires more unpleasant regularity conditions. This is one of many reasons why the asymptotic theory according to H\`{a}jek and Le Cam promotes the square roots of density functions as the primary object to investigate.
\begin{remark}[On the Fisher defect $\mathcal{I}^\dag(\cdot)$]\label{remark:fisher-defect}
If the parameter $t=0$ is an interior point of the parameter space $\Theta$, then $\dot \gamma_0 = 0$ for $\nu$-almost all $x$ in $\{x : dP_0(x)=0\}$. Consequently, the Fisher defect must be zero. However, even for a Hellinger differentiable statistical model, it is still possible for the Fisher defect to be non-zero on the boundary. An illustrative example is the Bernoulli distribution with parameter $p^2$ at $p=0$. Example 7.2 of~\citet{polyanskiy2022information} provides a formal derivation, and Example 18 of~\citet{pollardAsymptotia} offers an additional example.
\end{remark}
\begin{proof}[\bfseries{Proof of Lemma~\ref{lemma:chi-convergence1}}]
We define $\mathrm{sign}(x)  := x/\|x\|$ and $u := \mathrm{sign}(\varphi(t+h)-\varphi(t))$. Since it is assumed that the path $h\mapsto P_{\varphi(t+h)}$ is absolutely continuous, there exists a gradient such that 
\begin{align*}
    dP_{\varphi(t+h)}-dP_{\varphi(t)} = \int_{\varphi(t)}^{\varphi(t+h)}\dot{\rho}_{s, u}\, ds
\end{align*}
for any $u\in\mathbb{S}^{d-1}$.
Using this result, the chi-squared divergence can be written out as follows:
\begin{align*}
    \chi^2(P_{\varphi(t+h)}\|P_{\varphi(t)}) &= \int\frac{\big(dP_{\varphi(t+h)}-dP_{\varphi(t)}\big)^2}{dP_{\varphi(t)}}\, d\nu\\
    &= \int\frac{\big(\int_{\varphi(t)}^{\varphi(t+h)}\dot{\rho}_{s, u}\, ds\big)^2}{dP_{\varphi(t)}}\, d\nu\\
    &= \int\frac{\left\{\int_{0}^{1}(\nabla \varphi(t+h\widetilde{s}) h)^{\top}\dot{\rho}_{\varphi(t+h\widetilde{s}), u}\, d\widetilde{s}\right\}^2}{dP_{\varphi(t)}(x)}\, d\nu\\
    &= h^{\top}\left(\int\frac{\int_{0}^{1}\int_{0}^{1}\nabla \varphi^{\top}(t+h\widetilde{s}_1)\dot{\rho}_{\varphi(t+h\widetilde{s}_1), u}\dot{\rho}^{\top}_{\varphi(t+h\widetilde{s}_2), u}\nabla \varphi(t+h\widetilde{s}_2) \, d\widetilde{s}_1\, d\widetilde{s}_2}{dP_{\varphi(t)}}\, d\nu\right)h.
\end{align*}
The last steps use the change of variables. By the uniform integrability assumption and the continuity of $\dot{\rho}_{t}$, both asserted by \ref{as:regular_density_iso}, it follows that 
    \begin{align*}
    \chi^2(P_{\varphi(t+h)}\|P_{\varphi(t)}) &=h^{\top}\nabla\varphi(t)^{\top}\left(\int\frac{ \dot{\rho}_{\varphi(t)} \dot{\rho}^{\top}_{\varphi(t)}}{dP_{\varphi(t)}}\, d\nu\right)\nabla \varphi(t) h+ o(\|h\|^2)
\end{align*}
by the dominated convergence. As the middle term in the parenthesis is finite under \ref{as:regular_density_iso}, Lemma 6 of \citet{gassiat2013revisiting} implies that the path $dP^{1/2}_{\varphi(t)}$ is Hellinger differentiable with the gradient 
\[\dot \xi_{\varphi(t)} := \frac{1}{2}\frac{ \dot{\rho}_{\varphi(t)}}{dP^{1/2}_{\varphi(t)}}I(dP_{\varphi(t)}>0).\]
We can use this to show that:
\begin{align*}
    \mathcal{I}(\varphi(t)) = 4\int \dot \xi_{\varphi(t)}\dot \xi_{\varphi(t)}^{\top}\, d\nu = \left(\int\frac{ \dot{\rho}_{\varphi(t)} \dot{\rho}^{\top}_{\varphi(t)}}{dP_{\varphi(t)}}\, d\nu\right)
\end{align*}
and this leads us to conclude     
\begin{align*}
    \chi^2(P_{\varphi(t+h)}\|P_{\varphi(t)}) &=h^{\top}\varphi(t)^{\top}\mathcal{I}(\varphi(t))\nabla \varphi(t) h+ o(\|h\|^2).
\end{align*} 
as desired.
\end{proof}
\begin{proof}[\bfseries{Proof of Lemma~\ref{lemma:chi-convergence2}}]
The analogous proof can be found in Section 7.14 of \cite{polyanskiy2022information}. To begin, we observe that the chi-squared divergence can be written out as
\begin{align*}
        \chi^2(\widetilde{\mathbb{P}}^n_h \| \lambda\widetilde{\mathbb{P}}^n_h + (1-\lambda)\widetilde{\mathbb{P}}^n_0)&=\iint_{\mathcal{X}\times\Theta}  \frac{\{d\widetilde{\mathbb{P}}^n_h-(\lambda\widetilde{\mathbb{P}}^n_h + (1-\lambda)\widetilde{\mathbb{P}}^n_0)\}^2}{\lambda d\widetilde{\mathbb{P}}^n_h + (1-\lambda)d\widetilde{\mathbb{P}}^n_0}\\
        &=(1-\lambda)^2\iint_{\mathcal{X}\times\Theta}  \frac{(d\widetilde{\mathbb{P}}^n_h-d\widetilde{\mathbb{P}}^n_0)^2}{\lambda d\widetilde{\mathbb{P}}^n_h + (1-\lambda)d\widetilde{\mathbb{P}}^n_0}.
    \end{align*}
We will consider the function 
\[\phi := f \mapsto \iint_{\mathcal{X}\times\Theta}  \frac{f^2-d\widetilde{\mathbb{P}}^n_0}{\{\lambda f^2 + (1-\lambda)d\widetilde{\mathbb{P}}^n_0\}^{1/2}}.\]
It can then be seen that $\chi^2(\widetilde{\mathbb{P}}^n_h \| \lambda\widetilde{\mathbb{P}}^n_h + (1-\lambda)\widetilde{\mathbb{P}}^n_0) = (1-\lambda)^2\iint_{\mathcal{X}\times\Theta} \phi(d\widetilde{\mathbb{P}}^{n/2}_h)^2$. When $\phi$ is Lipshitz continuous and $d\widetilde{\mathbb{P}}^{n/2}_h$ is absolutely continuous with the gradient $t \mapsto \dot \gamma_{t,u}$, then the composition function $\phi(d\widetilde{\mathbb{P}}^{n/2}_h)$ is also absolutely continuous with the gradient $t \mapsto \phi'(d\widetilde{\mathbb{P}}^{n/2}_t)\dot \gamma_{t,u}$. In Section 7.14 of \citet{polyanskiy2022information}, it is shown that $\phi$ is indeed Lipshitz continuous with the constant $(2-\lambda)/\{(1-\lambda)\lambda^{1/2}\}$. Additionally, $\phi$ has a continuous derivative and we have
    \begin{align*}
        \phi'(d\widetilde{\mathbb{P}}^n_0) := \begin{cases}
            2, & \mbox{if}\quad (x,t) : d\widetilde{\mathbb{P}}^n_0(x,t) > 0\\
            \lambda^{-1/2}& \mbox{if}\quad (x,t) : d\widetilde{\mathbb{P}}^n_0(x,t) = 0.
        \end{cases}
    \end{align*}
Putting together, we have 
\[\lim_{t \longrightarrow 0}\phi'(d\widetilde{\mathbb{P}}^{n/2}_t)\dot \gamma_{t,u} = \dot \gamma_{0,u}\left(2I(d\widetilde{\mathbb{P}}^n_0 > 0) + \frac{1}{\lambda^{1/2}}I(d\widetilde{\mathbb{P}}^n_0 = 0)\right).\]
Noting that $\phi(d\widetilde{\mathbb{P}}^{n/2}_0)=0$, it follows that
\begin{align*}
    \chi^2(\widetilde{\mathbb{P}}^n_h \| \lambda\widetilde{\mathbb{P}}^n_h + (1-\lambda)\widetilde{\mathbb{P}}^n_0) &= (1-\lambda)^2\iint_{\mathcal{X}\times\Theta}  \phi(d\widetilde{\mathbb{P}}^{n/2}_h)^2 \\
    &= (1-\lambda)^2\iint_{\mathcal{X}\times\Theta} \left(\phi(d\widetilde{\mathbb{P}}^{n/2}_h)-\phi(d\widetilde{\mathbb{P}}^{n/2}_0)\right)^2 \\
    &= (1-\lambda)^2\iint_{\mathcal{X}\times\Theta} \left(\int_0^h\phi'(d\widetilde{\mathbb{P}}^{n/2}_s)\dot \gamma_{s,u}\, ds\right)^2 \\
    &= (1-\lambda)^2h^{\top}\iint_{\mathcal{X}\times\Theta} \left(\int_0^1\int_0^1\phi'(d\widetilde{\mathbb{P}}^{n/2}_{hs_1})\phi'(d\widetilde{\mathbb{P}}^{n/2}_{hs_2})\dot \gamma_{hs_1,u}\dot \gamma^{\top}_{hs_2,u}\, ds_1\,ds_1\right)h. 
\end{align*}
Since $\phi'$ is continuous and $\dot \gamma$ is the uniform integrability, asserted by \ref{as:regular_root_density_iso}, it follows that 
\begin{align*}
   & \chi^2(\widetilde{\mathbb{P}}^n_h \| \lambda\widetilde{\mathbb{P}}^n_h + (1-\lambda)\widetilde{\mathbb{P}}^n_0) \\
   &\qquad = (1-\lambda)^2h^{\top}\left\{\iint_{\mathcal{X}\times\Theta} \dot \gamma_{0}\dot \gamma^{\top}_{0}\left(4I(d\widetilde{\mathbb{P}}^n_0 > 0) + \frac{1}{\lambda}I(d\widetilde{\mathbb{P}}^n_0 = 0)\right)\right\}h + o(\|h\|^2)\\
    &\qquad=(1-\lambda)^2h^{\top}4\left\{\iint_{\mathcal{X}\times\Theta} \dot \gamma_{0}\dot \gamma^{\top}_{0} + \iint_{\mathcal{X}\times\Theta} \dot \gamma_{0}\dot \gamma^{\top}_{0}\left(\frac{1-4\lambda}{4\lambda}I(d\widetilde{\mathbb{P}}^n_0 = 0)\right)\right\}h + o(\|h\|^2)
\end{align*}
as $\|h\| \longrightarrow 0$ by the dominated convergence theorem. Finally, it has been established that \ref{as:regular_root_density_iso} $\implies$ \ref{as:hellinder-diff}, and that the corresponding Fisher information can be computed using $4\iint_{\mathcal{X}\times\Theta} \dot \gamma_{0}\dot \gamma^{\top}_{0}$. By the definition of the Hellinger differentiability, it follows that  
\begin{align*}
    H^2(\widetilde{\mathbb{P}}^n_{h}, \widetilde{\mathbb{P}}^n_0) = \int\left(d\widetilde{\mathbb{P}}_{h}^{n/2}- d\widetilde{\mathbb{P}}_{0}^{n/2}\right)^2  =h^{\top} \left(\iint_{\mathcal{X}\times\Theta}\dot\gamma_{0}\dot\gamma_{0}^{\top}\right)h + o(\|h\|^2)
\end{align*}
as $\|h\| \longrightarrow 0$. Thus in view of Lemma~\ref{lemma:hellinger-decomposition}, we have 
\[\iint_{\mathcal{X}\times\Theta}\dot\gamma_{0}\dot\gamma_{0}^{\top} = \frac{1}{4}\left(\mathcal{I}(Q) +n\int_{\mathbb{R}^d}  \nabla\varphi(t)^{\top}\, \mathcal{I}(\varphi(t))\,\nabla\varphi(t) \, dQ(t)\right)\]
and similarly 
\[\iint_{\mathcal{X}\times\Theta}\dot\gamma_{0}\dot\gamma_{0}^{\top} I(d\widetilde{\mathbb{P}}^n_0 = 0)= \frac{1}{4}\left(\mathcal{I}^\dag(Q) +n\int_{\mathbb{R}^d}  \nabla\varphi(t)^{\top}\, \mathcal{I}^\dag(\varphi(t))\,\nabla\varphi(t) \, dQ(t)\right)\]
where $\mathcal{I}^\dag(0) := 4\int \dot{\gamma}_{0}\,\dot{\gamma}_{0}^{\top} I(d\widetilde{\mathbb{P}}^n_0=0)$ is the Fisher defect. When the prior $Q$ has a continuously differentiable density $q$, it must follow that $\nabla q(t)=0$ whenever $q(t)=0$. This implies that the Fisher defect $\mathcal{I}^\dag(Q)$ is zero. Thereby, we conclude that
\begin{align*}
    &\chi^2(\widetilde{\mathbb{P}}^n_h \| \lambda\widetilde{\mathbb{P}}^n_h + (1-\lambda)\widetilde{\mathbb{P}}^n_0) \\
    &\qquad = (1-\lambda)^2h^{\top}\left\{\mathcal{I}(Q)+n\int \nabla\varphi(t)^{\top}\, \left(\mathcal{I}(\varphi(t))+ \frac{1-4\lambda}{4\lambda} \mathcal{I}^\dag(\varphi(t))\right)\,\nabla\varphi(t) \,dQ(t)\right\}h+ o(\|h\|^2).
\end{align*}
as desired.
\end{proof}

\begin{proof}[\bfseries{Proof of Lemma~\ref{lemma:local_behavior_hellinger}}]

By the definition of the Hellinger differentiability, we have
\begin{align*}
     H^2\left(P_{\varphi(t+h)}, P_{\varphi(t)}\right)&=\int\left(dP^{1/2}_{\varphi(t+h)}-dP^{1/2}_{\varphi(t)}\right)^2 \\
     &=\int\left(dP^{1/2}_{\varphi(t)+\varphi(t+h)-\varphi(t)}-dP^{1/2}_{\varphi(t)}\right)^2 \\
     &= \int\left((\varphi(t+h)-\varphi(t))^{\top} \dot \xi_{\varphi(t)}\right)^2 + o(\|\varphi(t+h)-\varphi(t)\|^2)\\
     &= \int\left((\nabla\varphi(t+\lambda h)h)^{\top} \dot \xi_{\varphi(t)}\right)^2 + o(\|\nabla\varphi(t+\lambda h)h)\|^2) \\
     &= h^{\top} \nabla\varphi^{\top}(t+\lambda h) \left(\int\dot \xi_{\varphi(t)} \dot \xi_{\varphi(t)}^{\top}\, d\nu\right) \nabla\varphi(t+\lambda h) h^{\top} + o(\|\nabla\varphi(t+\lambda h)h)\|^2)
\end{align*}
for some constant $\lambda \in [0,1]$ possibley depending on $t$. Since the Fisher information is defined as $4\int\dot \xi_{\varphi(t)} \dot \xi_{\varphi(t)}^{\top}\, d\nu$ and $\nabla \varphi(t+\lambda h) \longrightarrow \nabla\varphi(t)$ as $\|h\| \longrightarrow 0$ by the continuity, we conclude that 
\begin{align*}
    H^2\left(P_{\varphi(t+h)}, P_{\varphi(t)}\right)& = \frac{1}{4}h^{\top} \nabla\varphi(t)^{\top} \mathcal{I}(\varphi(t)) \nabla\varphi(t) h^{\top} + o(\|h\|^2)
\end{align*}
as $\|h\| \longrightarrow 0$.
\end{proof}
\begin{lemma}\label{lemma:hellinger-decomposition}
Assuming $Q$ has a density function that is bounded and continuously differentiable, and $P_{t}$ is Hellinger differentiable at $t = \varphi(t)$ such that $\|\nabla\varphi\|_\infty < C$, 
    \begin{align*}
     H^2(  \widetilde{\mathbb{P}}^n_{h}, \widetilde{\mathbb{P}}^n_0) = \frac{1}{4}h^{\top}\left(\mathcal{I}(Q) +n\int_{\mathbb{R}^d}  \nabla\varphi(t)^{\top}\, \mathcal{I}(\varphi(t))\,\nabla\varphi(t) \, dQ(t) \right)h + o(\|h\|^2)
\end{align*}
\end{lemma}
\begin{proof}[\bfseries{Proof of Lemma~\ref{lemma:hellinger-decomposition}}]
First, the Hellinger distance between $\widetilde{\mathbb{P}}^n_{h}$ and $\widetilde{\mathbb{P}}^n_0$ can be written out as
\begin{align*}
    H^2(  \widetilde{\mathbb{P}}_{h}, \widetilde{\mathbb{P}}_0) &=2-2\iint dP^{n/2}_{\varphi(t+h)}\,dQ^{1/2}(t+h)\,dP^{n/2}_{\varphi(t)}\,dQ^{1/2}(t)\,  \\
    &=2-2\int_{\mathbb{R}^d} dQ^{1/2}(t+h)dQ^{1/2}(t)\int_{\mathcal{X}}  dP^{n/2}_{\varphi(t+h)}\, dP^{n/2}_{\varphi(t)}\\
    &=2-2\int_{\mathbb{R}^d} dQ^{1/2}(t+h)\,dQ^{1/2}(t)\left(1-\frac{1}{2}H^2(P^n_{\varphi(t+h)}, P^n_{\varphi(t)})\right)\\
    &=H^2(Q_{h}, Q) +\int_{\mathbb{R}^d} \,dQ^{1/2}(t+h)dQ^{1/2}(t)\,H^2(P^n_{\varphi(t+h)}, P^n_{\varphi(t)}).
\end{align*}
When the density function $Q$ is continuously differentiable, the induced location family is differentiable in quadratic mean (Example 7.8 of \citet{van2000asymptotic}), which implies the Hellinger differentiability. By the analogous result to Lemma~\ref{lemma:local_behavior_hellinger} for $\varphi(t) = t$, we have
\begin{align*}
    H^2(Q_{h}, Q) = \frac{1}{4}h^{\top}\mathcal{I}(Q) h + o(\|h\|^2).
\end{align*}
Next, using the tensorization property of the Hellinger distance as well as Lemma~\ref{lemma:local_behavior_hellinger}, it implies
\begin{align}
    H^2\left(P^n_{\varphi(t)}, P^n_{\varphi(t+h)}\right) = 2-2 \left(1- \frac{H^2(P_{\varphi(t)} , P_{\varphi(t+h)})}{2}\right)^n = 2-2 \left(1-h^{\top}Ch +o(\|h\|^2)\right)^n\nonumber
\end{align}
as $\|h\|\longrightarrow 0$ where $C := \frac{1}{8}u^{\top}  \nabla \varphi(t)^{\top}\mathcal{I}(\varphi(t))\nabla\varphi(t) u$. By treating $n \ge 1$ as a fixed constant as $\|h\|\longrightarrow 0$, we have
\begin{align}
    2-2\left(1- h^{\top} C h+ o(\|h\|^2)\right)^n = 2-2(1-nh^{\top} C h+o(n\|h\|^2)) = 2nh^{\top}Ch + o(n\|h\|^2), \nonumber
\end{align}
by a first-order Taylor expansion. Putting them together, we have
\begin{align*}
    H^2(  \widetilde{\mathbb{P}}^n_{h}, \widetilde{\mathbb{P}}^n_0) =\frac{1}{4}h^{\top}\left(\mathcal{I}(Q) +n\int_{\mathbb{R}^d}  \nabla\varphi(t)^{\top}\, \mathcal{I}(\varphi(t))\,\nabla\varphi(t) \, dQ(t) \right)h + o(\|h\|^2).
\end{align*}
where we invoke the dominated convergence theorem to interchange the limit and integration operations, which follows by the boundedness of $\nabla \varphi$ and $\mathcal{I}(\cdot)$.
\end{proof}


\clearpage
\section{Proofs of asymptotic properties}\label{supp:asympt}
This section provides the derivation of results from Section~\ref{section:asympt_const}. Throughout, we assume that $X_1, \dots, X_n$ is an IID observation from $P_{\theta_0} \in \{P_\theta : \theta \in \Theta\}$ and each distribution in this model is Hellinger differentiable at $\theta$. 
\subsection{Proofs of local asymptotic minimax theorem}
In this section, we use Theorems~\ref{cor:vt-approximation}--\ref{thm:crvtBound_hellinger} and Lemma~\ref{lemma:sharper_ihbound} to prove the local asymptotic minimax (LAM) theorem. Let $\psi: \Theta \mapsto \mathbb{R}$ be continuously differentiable at $\theta_0$ and $T : \mathcal{X}^n \mapsto \mathbb{R}$ be any sequence of measurable functions.
We can then state the following result:
\begin{proposition}\label{prop:LAM_paramtric1}
Assuming the setting of Theorem~\ref{theorem:LAM_parametric}, the following statements hold:
\begin{enumerate}
	\item[(i)] Theorem~\ref{cor:vt-approximation} implies the LAM theorem,
    \item[(ii)] if \ref{as:regular_root_density_iso} holds and the Fisher defect is zero at $\theta_0$, then Theorem~\ref{thm:crvtBound} implies the LAM theorem, 
    \item[(iii)] Theorem~\ref{thm:crvtBound_hellinger} implies the LAM theorem, and 
    \item[(iv)] Lemma~\ref{lemma:sharper_ihbound} implies the LAM theorem with its lower bound multiplied by the constant $C \approx 0.28953$. 
\end{enumerate}
\end{proposition}
\begin{proof}[\bfseries{Proof of Proposition~\ref{prop:LAM_paramtric1} (i)}]
Since $\Phi$ is arbitrary in Lemma~\ref{thm:L2-approximation}, we choose \[\Phi := \{\phi: \Theta \mapsto \mathbb{R} \, : \, \phi \textrm{ is continuously differentiable at $\theta_0$}\}.\] Since $\psi \in \Phi$  by the differentiability of $\psi$, it simply follows that 
\begin{align*}
\sup_{\theta\in\Theta_0}\,\mathbb{E}_{\theta}|T(X)-\psi(P_\theta)|^2  &\ge \sup_{\phi \in \Phi,\, Q\in \mathcal{Q}}\, \left[\Gamma_{Q, \phi}^{1/2} - \left(\int \|\psi(\theta)-\phi(\theta)\|^2dQ(\theta)\right)^{1/2}\right]_+^2\\
&\ge \left(\int_{\Theta_0} \nabla \psi(t)\,dQ\right)^\top \left(\mathcal{I}(Q)+\int_{\Theta_0}\mathcal{I}(t)\, dQ\right)^{-1} \left(\int_{\Theta_0} \nabla \psi(t)\,dQ\right).
\end{align*} 
Section 4 of \citet{gassiat2013revisiting} shows that the van Trees inequality indeed recovers the optimal asymptotic constant of the LAM theorem. This concludes the claim.
\end{proof}
We prove the statement (iii) first.
\begin{proof}[\bfseries{Proof of Proposition~\ref{prop:LAM_paramtric1} (iii)}]
In this application, the parameter space considered is given by $\Theta := B(\theta_0, cn^{-1/2})$, which is an open $\mathbb{R}^d$-ball centered at $\theta_0$ with radius $cn^{-1/2}$. A diffeomorphism $t \mapsto \varphi(t) = \theta_0 + cn^{-1/2}\varphi_0(t)$ is defined where $\varphi_0$ is itself a diffeomorphism from $\mathbb{R}^d$ to $B([0], 1)$. It follows that $\nabla \varphi(t) = cn^{-1/2}\nabla \varphi_0(t)$. We further assume that $\|\nabla \varphi_0\|_\infty < C$ by some universal constant $C$, which is defined later. Theorem~\ref{thm:crvtBound_hellinger} is then applied to the composite function $t \mapsto \widetilde{\psi}(t) := (\psi\circ \varphi)(t)$, and the following inequality is obtained:
\begin{equation}\label{eq:LAM_hellinger_step1}
    \begin{aligned}
    &\inf_{T} \, \sup_{\theta \in\mathbb{R}^d}\, \mathbb{E}_{P_\theta^n}|T(X)-\widetilde{\psi}(\theta)|^2\\
    &\qquad\ge\left[\frac{|\int_{\mathbb{R}^d} (\widetilde{\psi}(t)-\widetilde{\psi}(t-h))\,dQ(t)|}{2H(\widetilde{\mathbb{P}}_0, \widetilde{\mathbb{P}}_h)} -\left(\int_{\mathbb{R}^d} |\widetilde{\psi}(t)-\widetilde{\psi}(t-h)|^2 \, dQ(t)\right)^{1/2}\right]^2_+
\end{aligned}
\end{equation}
for any $h \in \mathbb{R}^d$. Using the mean value theorem and the chain rule, it follows that 
\begin{align}
    \widetilde{\psi}(t)-\widetilde{\psi}(t-h) 
    &=cn^{-1/2}\nabla \psi(\lambda_1 \varphi(t) + (1-\lambda_1)\varphi(t-h))^{\top} \nabla \varphi_0(t - \lambda_2 h)^{\top}h \nonumber
\end{align}
where $\lambda_1, \lambda_2 \in [0,1]$ are constants that can possibly depend on $t$. Next, under the Hellinger differentiability and $n$ IID observations from $P_{\theta_0}$, it is shown by Lemma~\ref{lemma:hellinger-decomposition} that as $\|h\|_2 \longrightarrow 0$, 
\begin{align*}
    H^2(\widetilde{\mathbb{P}}_0, \widetilde{\mathbb{P}}_h) &= \frac{1}{4}h^{\top}\left(\mathcal{I}(Q) +n\int_{\mathbb{R}^d}  \nabla\varphi(t)^{\top}\, \mathcal{I}(\varphi(t))\,\nabla\varphi(t) \, dQ(t) \right)h + o(\|h\|_2^2) \\
    &= \frac{1}{4}h^{\top}\left(\mathcal{I}(Q) +c^2\int_{\mathbb{R}^d}  \nabla\varphi_0(t)^{\top}\, \mathcal{I}(\varphi(t))\,\nabla\varphi_0(t) \, dQ(t) \right)h + o(\|h\|_2^2).
\end{align*}
The results obtained thus far are now used to evaluate the expression~\eqref{eq:LAM_hellinger_step1}. Since the inequality holds for any $h \in \mathbb{R}^d$, it also holds as $\|h\|_2\longrightarrow0$. We now denote $h=u \varepsilon$ where $u \in \mathbb{S}^{d-1}$ and $\varepsilon \longrightarrow 0$. This yields the following lower bound: 
\begin{align*}
    &\sup_{h \in \mathbb{R}^d}\,\left[\frac{|\int_{\mathbb{R}^d} (\widetilde{\psi}(t)-\widetilde{\psi}(t-h))\,dQ(t)|}{2H(\mathbb{P}_0, \mathbb{P}_h)} -\left(\int_{\mathbb{R}^d} |\widetilde{\psi}(t)-\widetilde{\psi}(t-h)|^2 \, dQ(t)\right)^{1/2}\right]^2_+\\
    &\qquad \ge \sup_{u \in \mathbb{S}^{d-1}}\,\limsup_{\varepsilon \longrightarrow 0} \left[\frac{|\int_{\mathbb{R}^d} (\widetilde{\psi}(t)-\widetilde{\psi}(t-u\varepsilon))\,dQ(t)|}{2H(\mathbb{P}_0, \mathbb{P}_{u\varepsilon})} -\left(\int_{\mathbb{R}^d} |\widetilde{\psi}(t)-\widetilde{\psi}(t-u\varepsilon)|^2 \, dQ(t)\right)^{1/2}\right]^2_+\\
    &\qquad = \sup_{u \in \mathbb{S}^{d-1}}\,\frac{c^2n^{-1}|\int_{\mathbb{R}^d} \nabla \psi(\varphi(t))^{\top} \nabla \varphi_0(t)^{\top}u\,dQ(t)|^2}{u^{\top}\left(\mathcal{I}(Q) +c^2\int_{\mathbb{R}^d}  \nabla\varphi_0(t)^{\top}\, \mathcal{I}(\varphi(t))\,\nabla\varphi_0(t) \, dQ(t) \right)u}.
\end{align*}
By multiplying both sides by $n$, we have for any $u \in \mathbb{S}^{d-1}$,
\begin{align*}
    &\liminf_{n\longrightarrow \infty}\, \inf_{T} \, \sup_{\theta \in\mathbb{R}^d}\, n\mathbb{E}_{P_\theta^n}|T(X)-\widetilde{\psi}(\theta)|^2\ge \frac{c^2|\int_{\mathbb{R}^d} \nabla \psi(\theta_0)^{\top} \nabla \varphi_0(t)^{\top}u\,dQ(t)|^2}{u^{\top}\left(\mathcal{I}(Q) +c^2\int_{\mathbb{R}^d}  \nabla\varphi_0(t)^{\top}\, \mathcal{I}(\theta_0)\,\nabla\varphi_0(t) \, dQ(t) \right)u}\\
    \implies & \liminf_{c\longrightarrow \infty}\,\liminf_{n\longrightarrow \infty}\, \inf_{T} \, \sup_{\theta \in\mathbb{R}^d}\, n\mathbb{E}_{P_\theta^n}|T(X)-\widetilde{\psi}(\theta)|^2  \ge \frac{|\int_{\mathbb{R}^d} \nabla \psi(\theta_0)^{\top} \nabla \varphi_0(t)^{\top}u\,dQ(t)|^2}{u^{\top}\left(\int_{\mathbb{R}^d}  \nabla\varphi_0(t)^{\top}\, \mathcal{I}(\theta_0)\,\nabla\varphi_0(t) \, dQ(t) \right)u}
\end{align*}
where the dominated convergence theorem is invoked during the last steps to exchange the limiting and the integration operations. This follows since $\varphi(t) \longrightarrow \theta_0$ as $n\longrightarrow \infty$ and both $\nabla \psi$ and $\mathcal{I}$ are continuous at $\theta_0$ by assumption. Using $u := \|M^{-1/2}\|^{-1}M^{-1/2}$ where $M := \int_{\mathbb{R}^d}  \nabla\varphi_0(t)^{\top}\, \mathcal{I}(\theta_0)\,\nabla\varphi_0(t) \, dQ(t)$, it implies that
\begin{align*}
    &\liminf_{c\longrightarrow \infty}\,\liminf_{n\longrightarrow \infty}\, \inf_{T} \, \sup_{\theta \in\mathbb{R}^d}\, n\mathbb{E}_{P_\theta^n}|T(X)-\widetilde{\psi}(\theta)|^2
    \\
    & \qquad \ge \nabla \psi(\theta_0)^{\top}\cdot\left\{ \int_{\mathbb{R}^d}  \nabla \varphi_0(t)^{\top}\left(\int_{\mathbb{R}^d}  \nabla\varphi_0(t)^{\top}\, \mathcal{I}(\theta_0)\,\nabla\varphi_0(t) \, dQ(t)\right)^{-1}\nabla \varphi_0(t)\,dQ(t)\right\} \cdot\nabla \psi(\theta_0).
\end{align*}
Since the choice of $\varphi_0$ was also arbitrary, we may consider 
\[\varphi_0(t) = \frac{2}{\pi} \arctan(\|t\|/\gamma) \quad \mbox{and}\quad\nabla \varphi_0 (t) =  \frac{2}{\pi\gamma}\frac{1}{1+(\|t\|/\gamma)^2} \frac{t}{\|t\|}.\]
This choice also satisfies that $\|\nabla \varphi_0\|_\infty < 2(\pi \gamma)^{-1}$. Plugging them into the expression, the leading constant $2(\pi \gamma)^{-1}$ will be canceled between the numerator and the denominator. Taking $\gamma \longrightarrow \infty$, the gradient of $\varphi_0$ converges to a constant. Since the gradient of $\varphi_0$ is uniformly bounded by construction, it is concluded that 
\begin{align*}
    &\liminf_{c\longrightarrow \infty}\,\liminf_{n\longrightarrow \infty}\, \inf_{T} \, \sup_{\theta \in\mathbb{R}^d}\, n\mathbb{E}_{P_\theta^n}|T(X)-\widetilde{\psi}(\theta)|^2\ge \nabla \psi(\theta_0)^{\top}\cdot \mathcal{I}(\theta_0)^{-1}\cdot \nabla \psi(\theta_0)
\end{align*}
in view of the dominated convergence theorem. 
\end{proof}
\begin{proof}[\bfseries{Proof of Proposition~\ref{prop:LAM_paramtric1} (ii)}]
Similar to the proof of Proposition~\ref{prop:LAM_paramtric1} (iii), the parameter space is given by $\Theta := B(\theta_0, cn^{-1/2})$, and we define a diffeomorphism $t \mapsto \varphi(t) = \theta_0 + cn^{-1/2}\varphi_0(t)$ where $\varphi_0$ is itself a diffeomorphism from $\mathbb{R}^d$ to $B([0], 1)$. We further assume that $\|\nabla \varphi_0\|_\infty < C$ by some universal constant $C$.  

Theorem~\ref{thm:crvtBound} is then applied to the composite function $t \mapsto \widetilde{\psi}(t) := (\psi\circ \varphi)(t)$, and for any $h \in \mathbb{R}^d$, the following inequality is obtained:
\begin{equation}
    \begin{aligned}
    &\inf_T\sup_{t \in \mathbb{R}^d}\, \mathbb{E}_{P_{\varphi(s)}}|T(X)-\widetilde{\psi}(t)|^2\\
    &\quad \ge \sup_{h, L\ge 0}\frac{(1-\lambda)^2}{1+L\lambda}\left[\frac{|\int_{\mathbb{R}^d} (\widetilde{\psi}(t)-\widetilde{\psi}(t-h))\, dQ(t)|^2}{\chi^2(\widetilde{\mathbb{P}}_h \| \lambda\widetilde{\mathbb{P}}_h + (1-\lambda)\widetilde{\mathbb{P}}_0)}- \frac{\lambda(1+1/L)}{(1-\lambda)^2}\int_{\mathbb{R}^d}|\widetilde{\psi}(t)-\widetilde{\psi}(t-h)|^2\, dQ(t) \right]_+.\label{eq:lam_chi2_step1}
    \end{aligned}
\end{equation}

We further assume that the prior distribution $Q$ has a continuously differentiable Lebesgue density. Then under the $n$ IID observations from the Hellinger differentiable distribution with the Fisher defect being zero, Lemma~\ref{lemma:chi-convergence2} implies as $\|h\|_2 \longrightarrow 0$,  
\begin{align*}
    &\chi^2(\widetilde{\mathbb{P}}_h \| \lambda\widetilde{\mathbb{P}}_h + (1-\lambda)\widetilde{\mathbb{P}}_0) \\
    &\qquad= \chi^2(Q_h\|Q)+\int_{\mathbb{R}^d} \chi^2(P^n_{\varphi(t+h)}\|\lambda P^n_{\varphi(t+h)}+(1-\lambda)P^n_{\varphi(t)})\,\frac{dQ_h^2}{dQ} \\
    &\qquad=h^{\top}\left(\mathcal{I}(Q) +n(1-\lambda)^2\int_{\mathbb{R}^d}  \nabla\varphi(t)^{\top}\, \mathcal{I}(\varphi(t))\,\nabla\varphi(t) \, dQ(t) \right)h + o(\|h\|_2^2) \\
    &\qquad= h^{\top}\left(\mathcal{I}(Q) +c^2(1-\lambda)^2\int_{\mathbb{R}^d}  \nabla\varphi_0(t)^{\top}\, \mathcal{I}(\varphi(t))\,\nabla\varphi_0(t) \, dQ(t) \right)h + o(\|h\|_2^2).
\end{align*}

The results obtained previously are now used to evaluate the expression~\eqref{eq:lam_chi2_step1}. Since the inequality holds for any $h \in \mathbb{R}^d$, it also holds as $\|h\|_2\longrightarrow0$. Here, $h$ is denoted by $u \varepsilon$ where $u \in \mathbb{S}^{d-1}$ and $\varepsilon \longrightarrow 0$. This yields the following inequality: 
\begin{align*}
    &\frac{(1-\lambda)^2}{1+L\lambda}\left[\frac{|\int_{\mathbb{R}^d} (\widetilde{\psi}(t)-\widetilde{\psi}(t-h))\, dQ(t)|^2}{\chi^2(\widetilde{\mathbb{P}}_h \| \lambda\widetilde{\mathbb{P}}_h + (1-\lambda)\widetilde{\mathbb{P}}_0)}- \frac{\lambda(1+1/L)}{(1-\lambda)^2}\int_{\mathbb{R}^d}|\widetilde{\psi}(t)-\widetilde{\psi}(t-h)|^2\, dQ(t) \right]_+\\
    &\qquad \ge \sup_{u \in \mathbb{S}^{d-1}}\,\limsup_{\varepsilon \longrightarrow 0} \,\frac{(1-\lambda)^2}{1+L\lambda}\bigg[\frac{|\int_{\mathbb{R}^d} (\widetilde{\psi}(t)-\widetilde{\psi}(t-u\varepsilon))\, dQ(t)|^2}{\chi^2(\widetilde{\mathbb{P}}_{u\varepsilon} \| \lambda\widetilde{\mathbb{P}}_{u\varepsilon} + (1-\lambda)\widetilde{\mathbb{P}}_0)} \\
    &\qquad\qquad- \frac{\lambda(1+1/L)}{(1-\lambda)^2}\int_{\mathbb{R}^d}|\widetilde{\psi}(t)-\widetilde{\psi}(t-u\varepsilon)|^2\, dQ(t) \bigg]_+\\
    &\qquad \ge \sup_{u \in \mathbb{S}^{d-1}}\,\frac{(1-\lambda)^2}{1+L\lambda}\left\{\frac{c^2n^{-1}|\int_{\mathbb{R}^d} \nabla \psi(\varphi(t))^{\top} \nabla \varphi_0(t)^{\top}u\,dQ(t)|^2}{u^{\top}\left(\mathcal{I}(Q) +c^2(1-\lambda)^2\int_{\mathbb{R}^d}  \nabla\varphi_0(t)^{\top}\, \mathcal{I}(\varphi(t))\,\nabla\varphi_0(t) \, dQ(t) \right)u}\right\}.
\end{align*}
Since the above display holds for any $\lambda \in [0,1]$, we let $\lambda \longrightarrow 0$, which results in 
\begin{align*}
    \inf_T\sup_{t \in \mathbb{R}^d}\, \mathbb{E}_{P_{\varphi(s)}}|T(X)-\widetilde{\psi}(t)|^2 & \ge \sup_{u \in \mathbb{S}^{d-1}}\,\left\{\frac{c^2n^{-1}|\int_{\mathbb{R}^d} \nabla \psi(\varphi(t))^{\top} \nabla \varphi_0(t)^{\top}u\,dQ(t)|^2}{u^{\top}\left(\mathcal{I}(Q) +c^2\int_{\mathbb{R}^d}  \nabla\varphi_0(t)^{\top}\, \mathcal{I}(\varphi(t))\,\nabla\varphi_0(t) \, dQ(t) \right)u}\right\}.
\end{align*}
The remaining proof is identical to the proof of the second statement of Proposition~\ref{prop:LAM_paramtric1}.
\end{proof}

\begin{proof}[\bfseries{Proof of Proposition~\ref{prop:LAM_paramtric1} (iv)}]


First, we observe that 
\begin{align*}
&\sup_{\|\theta_0 - \theta\| < cn^{-1/2}}\, n\mathbb{E}_{\theta}\big|T(X) - \psi(\theta)\big|^2
   \\
   & \qquad =n\sup_{\|h\| < c} \mathbb{E}_{\theta_0 + hn^{-1/2}}\big|T(X) - \psi(\theta_0 + hn^{-1/2})\big|^2 \\
   & \qquad\ge n\sup_{\|h\| < c}\frac{1}{2}\left(\mathbb{E}_{\theta_0}\big|T(X) - \psi(\theta_0)\big|^2+\mathbb{E}_{\theta_0 + hn^{-1/2}}\big|T(X) - \psi(\theta_0 + hn^{-1/2})\big|^2 \right)\nonumber.
\end{align*}
We now apply Lemma~\ref{lemma:sharper_ihbound} to two points in parameter space $\theta_0$ and $\theta_0 + hn^{-1/2}$, which implies 
\begin{align}
    &\sup_{\|h\| < c}\, n\mathbb{E}_{\theta_0 + hn^{-1/2}}\big|T(X) - \psi(\theta_0 + hn^{-1/2})\big|^2 \nonumber\\
    &\qquad\ge n\left[\frac{1 -H^2(P^n_{\theta_0 + hn^{-1/2}}, P^n_{\theta_0})}{4}\right]_+|\psi(\theta_0 + hn^{-1/2}) - \psi(\theta_0)|^2\quad \mbox{for all} \quad \|h\| < c.\nonumber
\end{align}
For the remaining of the proof, $h$ is denoted by $u \varepsilon$ where $u \in \mathbb{S}^{d-1}$ and $0 < \varepsilon < c$. We treat $\varepsilon$ and $u$ as fixed constants as $n \longrightarrow \infty$. First by the mean value theorem, we have
\begin{align}
    |\psi(\theta_0 + u \varepsilon n^{-1/2}) - \psi(\theta_0)| = \varepsilon n^{-1/2} \nabla \psi(\theta^*)^{\top}u \nonumber
\end{align}
where $\theta^* := \lambda \theta_0 + (1-\lambda)(\theta_0 + u \varepsilon n^{-1/2}) = \theta_0 - \lambda u \varepsilon n^{-1/2}$ for $\lambda \in [0,1]$. This implies that
\begin{align*}
    n |\psi(\theta_0 + u \varepsilon n^{-1/2}) - \psi(\theta_0)|^2 = \varepsilon^2 \left(\nabla \psi(\theta_0)^{\top}u\right)^2 + o(1)
\end{align*}
as $n \longrightarrow \infty$, which follows by the continuity of $\nabla \psi^{\top}$ at $\theta_0$. Next, by the Hellinger differentiability of $P_\theta$ at $\theta_0$, Lemma~\ref{lemma:local_behavior_hellinger} implies that the Hellinger distance between $P_{\theta_0 + u\varepsilon n^{-1/2}}$ and $P_{\theta_0}$ associated with \textit{one} observation converges to 
\begin{align*}
    H^2(P_{\theta_0 + u\varepsilon n^{-1/2}}, P_{\theta_0}) &= \frac{\varepsilon^2 u^{\top}\mathcal{I}(\theta_0)u}{4n} + o(n^{-1})
\end{align*}
as $n \longrightarrow \infty$. By the tensorization property of the Hellinger distance and the fact that $(1-Z_n/n)^n \longrightarrow \exp(-Z)$ as $Z_n \longrightarrow Z$ for $n \longrightarrow \infty$, we have
\begin{align*}
   H^2(P^n_{\theta_0 + u\varepsilon n^{-1/2}}, P^n_{\theta_0})= 2-2\left(1-\frac{H^2(P_{\theta_0 + u\varepsilon n^{-1/2}}, P_{\theta_0})}{2}\right)^n \longrightarrow 2-2\exp\left(-\frac{\varepsilon^2 u^{\top}\mathcal{I}(\theta_0)u}{8}\right)
\end{align*}
as $n \longrightarrow \infty$. Putting them together, we obtain 
\begin{align}
    &\liminf_{n \longrightarrow \infty}\,\sup_{\|\theta_0 - \theta\| < cn^{-1/2}}\, n\mathbb{E}_{\theta}\big|T(X) - \psi(\theta)\big|^2 \nonumber\\
    &\qquad\ge\left[-\frac{1}{4}-\frac{1}{2}\exp\left(-\frac{\varepsilon^2 u^{\top}\mathcal{I}(\theta_0)u}{8}\right)\right]_+\varepsilon^2 \left(\nabla \psi(\theta_0)^{\top}u\right)^2\quad \mbox{for all} \quad 0 \le \varepsilon < c \quad \mbox{and} \quad u \in \mathbb{S}^{d-1}\nonumber.
\end{align}
It now remains to optimize the above display for $0 \le \varepsilon < c$ and $u \in \mathbb{S}^{d-1}$ as $c \longrightarrow \infty$. Since $u^{\top}\mathcal{I}(\theta_0)u$ is a scalar, we parameterize $\varepsilon$ such that $\overline{\varepsilon} = \varepsilon \big(u^{\top}\mathcal{I}(\theta_0)u\big)^{1/2}$ and so we can optimize over  $\overline{\varepsilon}$ instead. This gives us that 
\begin{align}
    &\liminf_{c \longrightarrow \infty}\, \liminf_{n \longrightarrow \infty}\,\sup_{\|\theta_0 - \theta\| < cn^{-1/2}}\, n\mathbb{E}_{\theta}\big|T(X) - \psi(\theta)\big|^2 \nonumber\\
    &\qquad\ge\sup_{u \in \mathbb{S}^{d-1}}\, \sup_{0 \le \overline{\varepsilon} < \infty}\left(-\frac{1}{4}-\frac{1}{2}\exp\left(-\frac{\overline{\varepsilon}^2}{8}\right)\right)_+\overline{\varepsilon}^2 \left(\nabla \psi(\theta_0)^{\top}u\right)^2\big(u^{\top}\mathcal{I}(\theta_0)u\big)^{-1}\nonumber
\end{align}
For the leading constant, we obtain 
\begin{align*}
    \sup_{0 \le \overline{\varepsilon} < \infty}\left\{-\frac{1}{4}+\frac{1}{2}\exp\left(-\frac{\overline{\varepsilon}^2 }{8}\right)\right\}\overline{\varepsilon}^2 = C \approx 0.28953
\end{align*}
and the optimal $u$ is given by $u^* := \|\mathcal{I}(\theta_0)^{-1/2}\|^{-1}\mathcal{I}(\theta_0)^{-1/2}$. Therefore, we conclude that 
\begin{align}
    \liminf_{c \longrightarrow \infty}\, \liminf_{n\longrightarrow \infty} \, \inf_{T}\,\sup_{\|\theta_0 - \theta\| < cn^{-1/2}}\, n\mathbb{E}_{\theta}\big|T(X) - \psi(\theta)\big|^2 \ge C (\nabla \psi(\theta_0)^{\top} \, \mathcal{I}(\theta_0)^{-1}\nabla \psi(\theta_0))\nonumber
\end{align}
where $C \approx 0.28953$.
\end{proof}

\subsection{Proof of Proposition~\ref{prop:LAMsemiparametric}}
We first define the parametric path to be used for the proof of Proposition~\ref{prop:LAMsemiparametric}. Following Example 25.16 of \cite{van2000asymptotic}, we define bounded univariate parametric paths as follows:
\begin{equation}
    dP_t(x) = \frac{1}{C_t}\kappa(t g(x))\, dP_0(x)\nonumber
\end{equation}
where $C_t := \int \kappa(t g(x))\, dP_0$. We assume that $\kappa(0)=\kappa'(0)=1$ and $\|\kappa'\|_\infty \le K$ and $\|\kappa''\|_\infty \le K$ for some constant $K$. We then use the following result from \citet{duchi2021constrained}.
\begin{lemma}[Lemma 1 of \cite{duchi2021constrained}]\label{lemma:chi-div-on-path}

Assuming $g \in \mathcal{T}_{P_0}$ and $dP_{t}$ is the parametric path defined as above, then as $t \longrightarrow 0$, 
\begin{align*}
    \chi^2(P_{t,g} \|P_0) = t^2 \int g^2\, dP_0 + o(t^2)
\end{align*}
\end{lemma}
\begin{proof}[\bfseries{Proof of Proposition~\ref{prop:LAMsemiparametric} (i)}] 
By an analogous argument from Proposition~\ref{prop:LAM_paramtric1}, we can choose 
\[\Phi = \{\phi : \mathcal{P} \mapsto \mathbb{R}\, : \, \phi \textrm{ is pathwise differentiable relative to } \mathcal{T}_{P_0}\}\]
for the application of Theorem~\ref{cor:vt-approximation}. We then apply the van Trees inequality along each parametric path. Although \citet{gassiat2013revisiting} does not provide an explicit statement for nonparametric settings, the proof remains analogous. 
\end{proof}

\begin{proof}[\bfseries{Proof of Proposition~\ref{prop:LAMsemiparametric} (ii)}] 
For each fixed $g \in L_2^0(P_0)$, consider a parametric path defined by $P_{t,g}$. Without loss of generality, we assume that the unknown data-generating distribution corresponds to $P_{0, g}$. We then consider the diffeomorphism $t \mapsto \varphi(t) = cn^{-1/2} \varphi_0(t)$ where $\varphi_0 : \mathbb{R} \mapsto (-1,1)$. As the score function $g$ is fixed throughout the proof, we denote 
the parametric path by $P_t$ and omit the dependency on $g$.

We apply Corollary~\ref{cor:lambda-zero} to the univariate functional $t \mapsto \psi(P_{\varphi(t)})$ over the joint probability measures defined as
 \[d\widetilde{\mathbb{P}}_0(x, t) := dP^n_{\varphi(t)}(x)\, dQ(t)\quad\mbox{and}\quad d\widetilde{\mathbb{P}}_h(x,t) := dP^n_{\varphi(t+h)}(x)\, dQ(t+h).\]
we then obtain that 
 \begin{align*}
    &\inf_{T}\, \sup_{|\theta| < cn^{-1/2}}\, \mathbb{E}_{P_\theta^n}|T(X)-\psi(\theta)|^2~\ge~ \sup_{h\in\mathbb{R}}\,\frac{\left|\int_{\mathbb{R}} \left(\psi(P_{\varphi(t)})-\psi(P_{\varphi(t-h)})\right)\,dQ(t)\right|^2}{\chi^2(\widetilde{\mathbb{P}}_h \| \widetilde{\mathbb{P}}_0)}.
\end{align*}
By the pathwise differentiablity of the functional, it follows that 
\begin{align*}
    &\psi(P_{\varphi(t)})-\psi(P_{\varphi(t-h)}) \\
    &\qquad= \psi(P_{ cn^{-1/2}\varphi_0(t)})-\psi(P_{cn^{-1/2}\varphi_0(t-h)}) \\
    &\qquad= cn^{-1/2} (\varphi_0(t)-\varphi_0(t-h)) \int \dot \psi_{cn^{-1/2}\varphi_0(t)}\, g_{cn^{-1/2}\varphi_0(t)} \, dP_{cn^{-1/2}\varphi_0(t)} \\
    &\qquad\qquad + o(cn^{-1/2} (\varphi_0(t)-\varphi_0(t-h))) \\
    &\qquad= cn^{-1/2} \varphi'_0(t-\lambda h)h \int \dot \psi_{cn^{-1/2}\varphi_0(t)}\, g_{cn^{-1/2}\varphi_0(t)} \, dP_{cn^{-1/2}\varphi_0(t)} + o(cn^{-1/2}h)
\end{align*}
for some constant $\lambda \in [0,1]$ possibly depending on $t$. After multiplying by $n$ both sides and taking $n \longrightarrow \infty$, we obtain 
\begin{align*}
   &\liminf_{n\longrightarrow \infty}\, n\left|\int_{\mathbb{R}} \left(\psi(P_{\varphi(t)})-\psi(P_{\varphi(t-h)})\right)\,dQ(t)\right| \\
   &\qquad =  \liminf_{n\longrightarrow \infty}\,c^2h^2 \left|\int_{\mathbb{R}}\varphi'_0(t-\lambda h)\left(\int \dot \psi_{cn^{-1/2}\varphi_0(t-h)}\, g_{cn^{-1/2}\varphi_0(t)} \, dP_{cn^{-1/2}\varphi_0(t)} \right)\, dQ(t)\right|^2+ o(c^2h^2)\\
   &\qquad \ge c^2h^2 \left(\int \dot \psi_{0}\, g \, dP_{0} \right)^2 \left|\int_{\mathbb{R}}\varphi'_0(t-\lambda h)\, dQ(t)\right|^2+ o(c^2h^2)
\end{align*}
where we use Fatou's lemma in the last step. Specifically, the tangent space under consideration corresponds to the entire $L_2^0(P_0)$ and thus $g_{cn^{-1/2}\varphi_0(t-h)} \in L_2^0(P_0)$. Since $L_2^0(P_0)$ is also a complete space, it follows that
\[g_{cn^{-1/2}\varphi_0(t)} \longrightarrow g_{0} \in L_2^0(P_0).\]
Hence, we can deduce that
\[\int \dot \psi_{cn^{-1/2}\varphi_0(t-h)}\, g_{cn^{-1/2}\varphi_0(t)} \, dP_{cn^{-1/2}\varphi_0(t)}\longrightarrow \int \dot \psi_{0}\, g_0 \, dP_{0}\]
as $n \longrightarrow \infty.$

Next, we analyze the local behavior of the chi-squared divergence on the path. By Lemma~\ref{lemma:chi-div-on-path} provided above, we obtain 
\begin{align*}
    \chi^2(P_{cn^{-1/2}\varphi_0(t+h)}\|P_{cn^{-1/2}\varphi_0(t)}) &= c^2n^{-1}\{\varphi_0(t+h)-\varphi_0(t)\}^2\int g^2\, dP_{cn^{-1/2}\varphi_0(t)}\\
    &= c^2n^{-1}h^2\{\varphi'_0(t+\lambda h)\}^2\int g^2\, dP_{cn^{-1/2}\varphi_0(t)}
\end{align*}
for some constant $\lambda \in [0,1]$. By the tensorization property of the chi-squared divergence, we have
\begin{align*}
    \chi^2(P^n_{cn^{-1/2}\varphi_0(t+h)}\|P^n_{cn^{-1/2}\varphi_0(t)}) 
    & =\left\{1+ \chi^2(P_{cn^{-1/2}\varphi_0(t+h)}\|P_{cn^{-1/2}\varphi_0(t)})\right\}^n-1\\
    &=  \left[1+c^2n^{-1} \{\varphi'_0(t+\lambda h)\}^2h^2 \int  g^2 \, dP_{cn^{-1/2}\varphi_0(t)}\right]^n -1\\
    & \longrightarrow \exp\left(c^2 \{\varphi'_0(t+\lambda h)\}^2h^2 \int  g^2 \, dP_{0}\right) -1
\end{align*}
as $n \longrightarrow \infty$. Therefore, we conclude that 
\begin{align*}
      \chi^2(\mathbb{P}_h \| \mathbb{P}_0)&= \chi^2(Q_h\|Q)+\int_{\mathbb{R}} \chi^2(P^n_{\varphi(t+h)}\|P^n_{\varphi(t)})\,\frac{dQ_h^2}{dQ}\\&= 
      \chi^2(Q_h\|Q)+\int_{\mathbb{R}} \left\{\exp\left(c^2 \{\varphi'_0(t+\lambda 
 h)\}^2h^2 \int  g^2 \, dP_{0}\right) -1\right\}\,\frac{dQ_h^2}{dQ}
\end{align*}
as $ h \longrightarrow 0$. Thus we obtain 
 \begin{align*}
    &\liminf_{n\longrightarrow\infty}\,\inf_{T}\, \sup_{|\theta| < cn^{-1/2}}\, n\mathbb{E}_{P_\theta^n}|T(X)-\psi(\theta)|^2\\
    &\qquad \ge \liminf_{n\longrightarrow\infty}\,\frac{\left|\int_{\mathbb{R}} \left(\psi(P_{\varphi(t)})-\psi(P_{\varphi(t-h)})\right)\,dQ(t)\right|^2}{\chi^2(\mathbb{P}_h \| \mathbb{P}_0)} \quad\mbox{for all } h \in\mathbb{R}\\
    &\qquad=\frac{c^2h^2 \left(\int \dot \psi_{0}\, g \, dP_{0} \right)^2 \left|\int_{\mathbb{R}}\varphi'_0(t-\lambda h)\, dQ(t)\right|^2+ o(c^2h^2)}{\chi^2(Q_h\|Q)+\int_{\mathbb{R}} \left\{\exp\left(c^2 \{\varphi'_0(t+\lambda 
 h)\}^2h^2 \int  g^2 \, dP_{0}\right) -1\right\}\,\frac{dQ_h^2}{dQ}} \quad\mbox{for all } h \in\mathbb{R}.
\end{align*}
Since the above inequality holds for any $h\in\mathbb{R}$, we take $h \longrightarrow 0$ to obtain
 \begin{align*}
&\liminf_{n\longrightarrow\infty}\,\inf_{T}\, \sup_{|\theta| < cn^{-1/2}}\, n\mathbb{E}_{P_\theta^n}|T(X)-\psi(\theta)|^2~\ge~ \frac{c^2\left(\int  \dot \psi_{0}\, g \, dP_{0}\right)^2\left|\int_{\mathbb{R}} \varphi'_0(t)\,dQ\right|^2}{\mathcal{I}(Q) + c^2\left(\int  g^2 \, dP_{0}\right)\int_{\mathbb{R}} \{\varphi'_0(t)\}^2 \, dQ} \\
    \implies & 
    \liminf_{c\longrightarrow\infty}\,\liminf_{n\longrightarrow\infty}\,\inf_{T}\, \sup_{|\theta| < cn^{-1/2}}\, n\mathbb{E}_{P_\theta^n}|T(X)-\psi(\theta)|^2~\ge~ \frac{\left(\int  \dot \psi_{0}\, g \, dP_{0}\right)^2\left|\int_{\mathbb{R}} \varphi'_0(t)\,dQ\right|^2}{\left(\int  g^2 \, dP_{0}\right)\int_{\mathbb{R}} \{\varphi'_0(t)\}^2 \, dQ}.
\end{align*}
As shown in the proof of Proposition~\ref{prop:LAM_paramtric1}, it follows that $\sup_{\varphi_0}\, \{\mathbb{E}_Q\varphi_0'(t)\}^2/\mathbb{E}_Q \{\varphi_0'(t)\}^2=1$ with $\arctan$. Finally, taking the supremum of the score functions $g$ over $L_2^0(P_0)$, we obtain 
\begin{align*}
    &\sup_{g\in L_2^0(P_0)}\, \liminf_{c\longrightarrow\infty}\,\liminf_{n\longrightarrow\infty}\,\inf_{T}\, \sup_{|\theta| < cn^{-1/2}}\, n\mathbb{E}_{P_{\theta,g}^n}|T(X)-\psi(\theta)|^2~\ge~ \sup_{g\in L_2^0(P_0)}\,\frac{\left(\int  \dot \psi_{0}\, g \, dP_{0}\right)^2}{\int  g^2 \, dP_{0}} = \int  \dot \psi_{0}^2\, dP_{0}
\end{align*}
where the last equality follows by the fact that $L_2^0(P_0)$ is linear closure and by definition, the efficient influence function $\dot \psi_{0}$ is contained in $L_2^0(P_0)$ (See Lemma 2.2 of \citet{van2002semiparametric})

\end{proof}
\begin{proof}[\bfseries{Proof of Proposition~\ref{prop:LAMsemiparametric} (iii)}]
Following the notation and the setting from the previous proof, we apply Theorem~\ref{thm:crvtBound_hellinger} to the univariate functional $t \mapsto \psi(P_{\varphi(t)})$ over the joint probability measures $\widetilde{\mathbb{P}}_0(x, t)$ and $\widetilde{\mathbb{P}}_h(x,t)$, and we obtain 
\begin{equation*}
    \begin{aligned}
    &\inf_{T} \, \sup_{|\theta| < cn^{-1/2}}\, \mathbb{E}_{P_\theta^n}|T(X)-\psi(\theta)|^2\\
    &\qquad\ge\left[\frac{|\int_{\mathbb{R}} \left(\psi(P_{\varphi(t)})-\psi(P_{\varphi(t-h)})\right)\,dQ(t)|}{2H(\widetilde{\mathbb{P}}_0, \widetilde{\mathbb{P}}_h)} -\left(\int_{\mathbb{R}} \left(\psi(P_{\varphi(t)})-\psi(P_{\varphi(t-h)})\right)^2\,dQ(t)\right)^{1/2}\right]^2_+
\end{aligned}
\end{equation*}
for all $h \in\mathbb{R}$. As shown in the previous proof, it follows that 
\begin{align*}
   &\liminf_{n\longrightarrow \infty}\, n\left|\int_{\mathbb{R}} \left(\psi(P_{\varphi(t)})-\psi(P_{\varphi(t-h)})\right)\,dQ(t)\right|  \\
   &\qquad \ge c^2h^2 \left(\int \dot \psi_{0}\, g \, dP_{0} \right)^2 \left|\int_{\mathbb{R}}\varphi'_0(t-\lambda h)\, dQ(t)\right|^2+ o(c^2h^2)
\end{align*}
by Fatou's lemma. Similarly, it also follows that 
\begin{align*}
   &\liminf_{n\longrightarrow \infty}\, n\left(\int_{\mathbb{R}} \left(\psi(P_{\varphi(t)})-\psi(P_{\varphi(t-h)})\right)^2\,dQ(t)\right)^{1/2}  = O(c^2h^2).
\end{align*}
Next, we study the local behavior of the Hellinger distance. Since $\{P_\theta : \theta \in \Theta\}$ is a QMD family, the Hellinger distance associated with \textit{one} observation follows:
\begin{align}
    &H^2(P_{\varphi(t+h)}, P_{\varphi(t)}) \nonumber\\
    &\qquad = \int \left(dP_{cn^{-1/2}\varphi_0(t+h)}^{1/2}-dP_{cn^{-1/2}\varphi_0(t)}^{1/2}\right)^2\nonumber\\
    &\qquad = \frac{1}{4}c^2n^{-1}
    \int \{\varphi_0(t+h)-\varphi_0(t)\}^2\, g_{cn^{-1/2}\varphi_0(t)}^2 \, dP_{cn^{-1/2}\varphi_0(t)} + o(c^2n^{-1}\{\varphi_0(t+h)-\varphi_0(t)\}^2)\nonumber\\
    &\qquad = \frac{1}{4}c^2n^{-1}h^2
    \int \left\{\varphi'_0(t+\lambda h)\right\}^2\, g_{cn^{-1/2}\varphi_0(t)}^2 \, dP_{cn^{-1/2}\varphi_0(t)} + o(c^2n^{-1}h^2)\nonumber
\end{align}
and by the tensorization property,
\begin{align}
    &\liminf_{n\longrightarrow \infty}\,H^2\left(P^n_{\varphi(t+h)}, P^n_{\varphi(t)}\right) \nonumber\\
    &\qquad = \liminf_{n\longrightarrow \infty}\,\left\{2-2\left(1- \frac{H^2\left(P_{\varphi(t+h)}, P_{\varphi(t)}\right)}{2}\right)^n\right\}\nonumber\\
    &\qquad=\liminf_{n\longrightarrow \infty}\, \left\{2-2\left(1- \frac{c^2h^2
    \int \left\{\varphi'_0(t+\lambda h)\right\}^2\, g_{cn^{-1/2}\varphi_0(t)}^2 \, dP_{cn^{-1/2}\varphi_0(t)} + o(c^2n^{-1}h^2)}{8n}\right)^n\right\}\nonumber\\
    &\qquad= 2-2\exp\left(-\frac{c^2h^2
    \int \left\{\varphi'_0(t+\lambda h)\right\}^2\, g^2 \, dP_{0}}{8}\right).\nonumber
\end{align}
Thus we obtain 
\begin{align}
    &\liminf_{n \longrightarrow \infty}\, \inf_{T} \, \sup_{|\theta| < cn^{-1/2}}\, n\mathbb{E}_{P_\theta^n}|T(X)-\psi(\theta)|^2\nonumber\\
    &\qquad\ge\liminf_{n\longrightarrow \infty}\, \frac{n|\int_{\mathbb{R}} \left(\psi(P_{\varphi(t)})-\psi(P_{\varphi(t-h)})\right)\,dQ(t)|^2}{4H^2\left(\widetilde{\mathbb{P}}_0, \widetilde{\mathbb{P}}_h\right)} -O(c^2h^2)\quad\mbox{for all } h \in\mathbb{R}\nonumber\\
    &\qquad= \frac{c^2h^2 \left(\int \dot \psi_{0}\, g \, dP_{0} \right)^2 \left|\int_{\mathbb{R}}\varphi'_0(t-\lambda h)\, dQ(t)\right|^2+ o(c^2h^2)}{4\left(H^2(Q_h, Q_0) + \int_{\mathbb{R}} \left\{2-2\exp\left(-\frac{c^2h^2
    \int \left\{\varphi'_0(t+\lambda h)\right\}^2\, g^2 \, dP_{0}}{8}\right)\right\}\,dQ_h^{1/2} \, dQ^{1/2}\right)} -O(c^2h^2)\nonumber
\end{align}
for all $ h \in\mathbb{R}$. Since the above inequality holds for any $h\in\mathbb{R}$, we take $h \longrightarrow 0$ to obtain
\begin{align*}
&\liminf_{n\longrightarrow\infty}\,\inf_{T}\, \sup_{|\theta| < cn^{-1/2}}\, n\mathbb{E}_{P_\theta^n}|T(X)-\psi(\theta)|^2~\ge~ \frac{c^2\left(\int  \dot \psi_{0}\, g \, dP_{0}\right)^2\left|\int_{\mathbb{R}} \varphi'_0(t)\,dQ\right|^2}{\mathcal{I}(Q) + c^2\left(\int  g^2 \, dP_{0}\right)\int_{\mathbb{R}} \{\varphi'_0(t)\}^2 \, dQ}.
\end{align*}
This follows since 
\begin{align*}
    &\lim_{h\longrightarrow 0}\, H^2(Q_h, Q_0) + \int_{\mathbb{R}} \left\{2-2\exp\left(-\frac{c^2h^2
    \int \left\{\varphi'_0(t+\lambda h)\right\}^2\, g^2 \, dP_{0}}{8}\right)\right\}\,dQ_h^{1/2} \, dQ^{1/2} \\
    &\qquad = \frac{1}{4}h^2 \mathcal{I}(Q)+ \frac{1}{4}c^2h^2\left(\int g^2 \, dP_{0}\right)\int_{\mathbb{R}}\left\{\varphi'_0(t)\right\}^2\,dQ 
\end{align*} 
by Taylor expansion. The rest of the proof is identical to the statement (ii) of Proposition~\ref{prop:LAMsemiparametric}.
\end{proof}
\begin{proof}[\bfseries{Proof of Proposition~\ref{prop:LAMsemiparametric} (iv)}]
Similar to the proof of the statement (iv) of Proposition~\ref{prop:LAM_paramtric1}, we apply Lemma~\ref{lemma:sharper_ihbound} to the QMD parametric paths. Without loss of generality, we assume that $\theta_0 = 0$. Then for fixed $g \in \mathcal{T}_{P_0}$, we invoke Lemma~\ref{lemma:sharper_ihbound} as follows:
    \begin{align*}
&\sup_{|\theta| < cn^{-1/2}}\, n\mathbb{E}_{P^n_{\theta,g}}\big|T(X) - \psi(P_{\theta,g})\big|^2
 \\
 &\qquad \ge n\sup_{|h| < c}\frac{1}{2}\left(\mathbb{E}_{0}\big|T(X) - \psi(P_{0})\big|^2+\mathbb{E}_{hn^{-1/2}}\big|T(X) - \psi(P_{hn^{-1/2}})\big|^2 \right)\nonumber\\
    &\qquad\ge n\left[\frac{1 -H^2(P^n_{hn^{-1/2}}, P^n_{0})}{4}\right]_+|\psi(P_{hn^{-1/2}}) - \psi(P_{0})|^2\quad \mbox{for all} \quad |h| < c\nonumber.
\end{align*}
Since $\psi$ is pathwise differentiable, it follows that 
\begin{align*}
   &\liminf_{n\longrightarrow \infty}\, n\left|\psi(P_{0})-\psi(P_{hn^{-1/2}})\right| = h^2 \left(\int \dot \psi_{0}\, g \, dP_{0} \right)^2.
\end{align*}
Also by the QMD assumption of the parametric path, we have

\begin{align}
    H^2(P_{hn^{-1/2}}, P_{0}) &= \int \left(dP_{hn^{-1/2}}^{1/2}-dP_{0}^{1/2}\right)^2\nonumber= \frac{1}{4}h^2n^{-1}
    \int g^2 \, dP_{0} + o(h^2n^{-1})\nonumber
\end{align}
followed by the tensorization property of the Hellinger distance, 
\begin{align*}
    \liminf_{n\longrightarrow \infty}H^2(P^n_{hn^{-1/2}}, P^n_{0}) &= \liminf_{n\longrightarrow\infty}\, \left\{2-2\left(1- \frac{H^2\left(P_{hn^{-1/2}}, P_{0}\right)}{2}\right)^n\right\} \\
    &= \liminf_{n\longrightarrow\infty}\, \left\{2-2\left(1- \frac{h^2\int g^2 \, dP_{0} + o(h^2n^{-1})}{8n}\right)^n\right\} \\
    &= 2-2\exp\left(-\frac{h^2\int g^2 \, dP_{0}}{8}\right).
\end{align*}
Putting them together, we obtain 
\begin{align}
    &\liminf_{c \longrightarrow \infty}\,\liminf_{n \longrightarrow \infty}\,\sup_{|\theta| < cn^{-1/2}}\, n\mathbb{E}_{\theta}\big|T(X) - \psi(P_{\theta})\big|^2 \nonumber\\
    &\qquad \ge\sup_{0 \le h < \infty}\left[-\frac{1}{4}-\frac{1}{2}\exp\left(-\frac{h^2\int g^2 \, dP_{0}}{8}\right)\right]_+h^2 \left(\int \dot \psi_{0}\, g \, dP_{0} \right)^2. \nonumber
\end{align}
Similar to the proof of Proposition~\ref{prop:LAM_paramtric1} (iv), we let $\widetilde{h} = h  \left(\int g^2 \, dP_{0}\right)^{1/2}$ and optimize over $\widetilde{h}$ instead. This yields that 
\begin{align*}
    &\liminf_{c \longrightarrow \infty}\,\liminf_{n \longrightarrow \infty}\,\sup_{|\theta| < cn^{-1/2}}\, n\mathbb{E}_{\theta}\big|T(X) - \psi(P_{\theta})\big|^2 \ge C\left(\int \dot \psi_{0}\, g \, dP_{0} \right)^2 / \left(\int g^2 \, dP_{0} \right) \nonumber
\end{align*}
where $C \approx 0.28953$. Using Lemma 2.2 of \citet{van2002semiparametric} and taking the supremum over the linear closure of the tangent set, we conclude that 
\begin{align*}
    \sup_{g \in\mathcal{T}_{P_0}}\,\liminf_{c \longrightarrow \infty}\,\liminf_{n \longrightarrow \infty}\,\sup_{|\theta| < cn^{-1/2}}\, n\mathbb{E}_{\theta}\big|T(X) - \psi(P_{\theta})\big|^2 \ge C\int \dot \psi^2_{0}\, dP_0\nonumber
\end{align*}
where $C \approx 0.28953$.
\end{proof}
\subsection{Proof of Lemma~\ref{lemma:alpha-beta}}
To begin, we define a diffeomorphism $\varphi_0 : \mathbb{R}^d \mapsto B([0], 1)$ between $\mathbb{R}^d$ and an open unit ball in $\mathbb{R}^d$. We then construct the following mapping:
\begin{align}\label{eq:diffeomorphism_n_gamma}
    \varphi(s) := \theta_0 + cn^{-1/\alpha}\varphi_0(s)
\end{align}
for all $s \in \mathbb{R}^d$. The resulting mapping is a valid diffeomorphism between $\mathbb{R}^d$ and $B(\theta_0,cn^{-1/\alpha})$. Additionally by the differentiablity of $\varphi_0$, we have $\nabla \varphi(s) =c n^{-1/\alpha}\nabla \varphi_0(s)$. Similar to several preceding proofs such as the proof of Proposition~\ref{prop:LAM_paramtric1}, we apply Theorem~\ref{thm:crvtBound_hellinger} to the composition function $t \mapsto \widetilde{\psi}(t) := (\psi \circ \varphi)(t)$ and the statistical models 
 \begin{align*}
    d\widetilde{\mathbb{P}}^n_0(x,t) := dP^n_{\varphi(t)}(x)\,dQ(t) \quad \mbox{and} \quad  d\widetilde{\mathbb{P}}^n_h(x,t) := dP^n_{\varphi(t+h)}(x)\, dQ(t+h),
\end{align*}
which implies 
\begin{align*}
    &\inf_{T} \, \sup_{\theta \in\mathbb{R}^d}\, \mathbb{E}_{\theta}|T(X)-\widetilde \psi(\theta)|^2\\
    &\qquad\ge \bigg[\frac{\left|\int_{\mathbb{R}^d} \left(\widetilde \psi(t)-\widetilde \psi(t-h)\right)\,dQ(t)\right|}{2H(\widetilde{\mathbb{P}}_0, \widetilde{\mathbb{P}}_h)} -\left(\int_{\mathbb{R}^d} \left|\widetilde \psi(t)-\widetilde \psi(t-h)\right|^2 \, dQ(t)\right)^{1/2}\bigg]^2_+.
\end{align*}
for any $h \in \mathbb{R}^d$. Let $t_*$ be a point in $\mathbb{R}^d$ such that $t_* :=  t - \lambda h$ for some $\lambda \in [0,1]$ and $\vartheta_t := \theta_0 + cn^{-1/\alpha}\varphi_0(t)$. Assuming that $n$ is large enough such that $\|\theta_0 - \vartheta_t\| < \delta/2$, then by the H\"{o}lder-smoothness assumption of $\psi$, we have 
\begin{align*}
    &\widetilde{\psi}(t)-\widetilde{\psi}(t-h)\\
    &\qquad = \psi(\theta_0 + cn^{-1/\alpha}\varphi_0(t)) - \psi(\theta_0 + cn^{-1/\alpha}\varphi_0(t-h))\\
    &\qquad= \psi(\vartheta_t) - \psi(\vartheta_t - cn^{-1/\alpha}(\varphi_0(t)-\varphi_0(t-h))) \\
    &\qquad= C_{2}\big(\vartheta_t, \mathrm{sign}(\varphi_0(t)-\varphi_0(t-h))\big)c^\beta n^{-\beta/\alpha}\|\varphi_0(t)-\varphi_0(t-h)\|^\beta + o(c^\beta n^{-\beta/\alpha}\|\varphi_0(t)-\varphi_0(t-h)\|^\beta)\\
    &\qquad= C_{2}\big(\vartheta_t, \mathrm{sign}(\varphi_0(t)-\varphi_0(t-h))\big)c^\beta n^{-\beta/\alpha}\|\nabla \varphi_0(t_*)h\|^\beta + o(c^\beta n^{-\beta/\alpha}\|\nabla \varphi_0(t_*)h\|^\beta).
\end{align*}
Multiplying by $n^{\beta/\alpha}$ and taking the limit as $n \longrightarrow \infty$, we have 
\begin{align*}
    &\liminf_{n\longrightarrow \infty }\, n^{\beta/\alpha}\left|\int_{\mathbb{R}^d} \left(\widetilde \psi(t)-\widetilde \psi(t-h)\right)\,dQ(t)\right|  \\
    &\qquad ~\ge~ c^\beta \int_{\mathbb{R}^d} C_{2}\big(\vartheta_t, \mathrm{sign}(\varphi_0(t)-\varphi_0(t-h))\big)\|\nabla \varphi_0(t-\lambda h)h\|^\beta \, dQ(t)
\end{align*}
by Fatou's lemma. We also have that 
\begin{align*}
    &\liminf_{n\longrightarrow \infty }\, \left(\int_{\mathbb{R}^d} \left|\widetilde \psi(t)-\widetilde \psi(t-h)\right|^2\,dQ(t)\right)^{1/2} \\
    &\qquad = c^\beta \left(\int_{\mathbb{R}^d} C_{2}^2\big(\vartheta_t, \mathrm{sign}(\varphi_0(t)-\varphi_0(t-h))\big)\|\nabla \varphi_0(t-\lambda h)h\|^{2\beta} \, dQ(t)\right)^{1/2}.
\end{align*}
by the dominated convergence theorem, which follows since $\nabla \varphi_0$ is continuous and thus bounded as $h \longrightarrow 0$. Moving onto the Hellinger distance, by the H\"{older}-smoothness assumption, we have 
\begin{align*}
&H^2\left(P_{\varphi(t+h)}, P_{\varphi(t)}\right)\\
&\qquad =H^2\left(P_{\theta_0 + cn^{-1/\alpha}\varphi_0(t+h)}, P_{\theta_0 + cn^{-1/\alpha}\varphi_0(t)}\right)\\
&\qquad=H^2\left(P_{\theta_t+cn^{-1/\alpha}(\varphi_0(t+h)-\varphi_0(t))}, P_{\theta_t}\right)\\
    &\qquad= C_{1}\big(\vartheta_t, \mathrm{sign}(\varphi_0(t+h)-\varphi_0(t))\big)c^\alpha n^{-1}\|\varphi_0(t+h)-\varphi_0(t)\|^\alpha + o(c^\alpha n^{-1}\|\varphi_0(t+h)-\varphi_0(t)\|^\alpha)\\
    &\qquad= C_{1}\big(\vartheta_t, \mathrm{sign}(\varphi_0(t+h)-\varphi_0(t))\big)c^\alpha n^{-1}\|\nabla \varphi_0(t+\lambda h)h\|^\alpha + o(c^\alpha n^{-1}\|\nabla \varphi_0(t+\lambda h)h\|^\alpha),
\end{align*}
which then followed by the tensorization property of the Hellinger distance to yield
\begin{align*}
    \liminf_{n\longrightarrow \infty}\, H^2\left(P^n_{\varphi(t+h)}, P^n_{\varphi(t)}\right) &=\liminf_{n\longrightarrow \infty}\, \left\{ 2-2 \left(1- \frac{H^2\left(P_{\theta_0 + cn^{-1/\alpha}\varphi_0(t)}, P_{\theta_0 + cn^{-1/\alpha}\varphi_0(t-h)}\right)}{2}\right)^n\right\}\\
    &= 2-2 \exp \left(- \frac{C_{1}\big(\vartheta_t, \mathrm{sign}(\varphi_0(t+h)-\varphi_0(t))\big)c^\alpha \|\nabla \varphi_0(t +\lambda h)h\|^\alpha}{2}\right).
\end{align*}
We then conclude that 
\begin{align*}
    &H^2(\widetilde{\mathbb{P}}_0, \widetilde{\mathbb{P}}_h) \\
    &\qquad = H^2(Q_h, Q_0) + \int_{\mathbb{R}} H^2(P^n_{\varphi(t + h)}, P^n_{\varphi(t)})\,dQ_h^{1/2} \, dQ^{1/2}\\
    &\qquad= H^2(Q_h, Q_0) + \int_{\mathbb{R}} \left\{2-2\exp\left(-\frac{C_{1}\big(\vartheta_t, \mathrm{sign}(\varphi_0(t+h)-\varphi_0(t))\big)c^\alpha \|\nabla \varphi_0(t +\lambda h)h\|^\alpha}{2}\right)\right\}\,dQ_h^{1/2} \, dQ^{1/2}\\ 
    &\qquad= 2 - 2\int_{\mathbb{R}} \exp\left(-\frac{C_{1}\big(\vartheta_t, \mathrm{sign}(\varphi_0(t+h)-\varphi_0(t))\big)c^\alpha \|\nabla \varphi_0(t +\lambda h)h\|^\alpha}{2}\right)\,dQ_h^{1/2} \, dQ^{1/2}.
\end{align*} 

For the ease of notation, we denote the constants as follows: 
\[\widetilde{C}_1(t) := C_{1}\big( \vartheta_t, \mathrm{sign}(\varphi_0(t+h)-\varphi_0(t))\big)\quad \mbox{and} \quad \widetilde{C}_2(t) := C_{2}\big(\vartheta_t, \mathrm{sign}(\varphi_0(t)-\varphi_0(t-h))\big).\]
Then, putting these intermediate results together, we have 
\begin{align*}
    &\liminf_{n\longrightarrow\infty}\,\inf_{T} \, \sup_{\theta \in\mathbb{R}^d}\, n^{2\beta/\alpha}\mathbb{E}_{\theta}|T(X)-\widetilde \psi(\theta)|^2\\
    &\qquad\ge \left[\frac{c^\beta \int_{\mathbb{R}^d} \widetilde{C}_{2}(t)\|\nabla \varphi_0(t-\lambda h)h\|^{\beta} \, dQ(t)}{2\left\{2 - 2\int_{\mathbb{R}} \exp\left(-\widetilde{C}_{1}(t)c^\alpha \|\nabla \varphi_0(t +\lambda h)h\|^\alpha/2\right)\,dQ_h^{1/2} \, dQ^{1/2}\right\}^{1/2}} \right.\\
    &\qquad\qquad-c^\beta \left(\int_{\mathbb{R}^d} \widetilde{C}^2_{2}(t)\|\nabla \varphi_0(t-\lambda h)h\|^{2\beta} \, dQ(t)
    \right)^{1/2}\bigg]^2_+ 
\end{align*}
where the above display holds for any $h \in \mathbb{R}^d$. We denote $\bar{h} = hc$ and let $c \longrightarrow \infty$. This further simplifies the expression to 

\begin{align*}
    &\liminf_{c\longrightarrow\infty}\liminf_{n\longrightarrow\infty}\,\inf_{T} \, \sup_{\theta \in\mathbb{R}^d}\, n^{2\beta/\alpha}\mathbb{E}_{\theta}|T(X)-\widetilde \psi(\theta)|^2\\
    &\qquad\ge \liminf_{\bar{h}\longrightarrow\infty}\left[\frac{\int_{\mathbb{R}^d} \widetilde{C}_{2}(t)\|\nabla \varphi_0(t)\bar{h}\|^{\beta} \, dQ(t)}{2\left\{2 - 2\int_{\mathbb{R}} \exp\left(-\widetilde{C}_{1}(t)\|\nabla \varphi_0(t )\bar{h}\|^\alpha/2\right)\,dQ\right\}^{1/2}}\right. \\
    &\qquad\qquad-\left(\int_{\mathbb{R}^d} \widetilde{C}^2_{2}(t)\|\nabla \varphi_0(t)\bar{h}\|^{2\beta} \, dQ(t)
    \right)^{1/2}\bigg]^2_+.
\end{align*}

We claim that there exists the choice of $\varphi_0$ and $Q$ to make the lower bound strictly greater than zero. Since the above inequality holds for any $\varphi_0$ and $\bar{h} \in \mathbb{R}^d$, we consider the class of diffeomorphism such that 
\[\|\nabla\varphi_0\bar{h}\| \longrightarrow \gamma \quad \text{as} \quad \|\bar{h}\| \longrightarrow \infty. \]
Denoting $\bar{h} = u \varepsilon$ for $u \in \mathbb{S}^{d-1}$, this can be obtained, for instance, by 
\[\varphi_0(t) = \frac{2}{\pi} \arctan\left(\frac{\pi \gamma \|t\|}{2\varepsilon}\right) \quad \mbox{and}\quad\nabla \varphi_0 (t) =  \frac{\gamma}{\varepsilon}\left\{1+\left(\frac{\pi \gamma \|t\|}{2\varepsilon}\right)^2\right\}^{-1} \frac{t}{\|t\|}.\]
For any choice of $\gamma > 0$, the image of $\varphi_0$ is $B([0], 1)$ and $\|\nabla\varphi_0\bar{h}\| \longrightarrow \gamma$ as $\|\bar h\|\longrightarrow \infty$. Since $C_{2}(t,\mathrm{sign}(h)) \le \overline{C}_1$ and $C_{2}(t,\mathrm{sign}(h)) \le \overline{C}_2$ uniformly, we obtain
\begin{align*}
    &\liminf_{c\longrightarrow\infty}\liminf_{n\longrightarrow\infty}\,\inf_{T} \, \sup_{\theta \in\mathbb{R}^d}\, n^{2\beta/\alpha}\mathbb{E}_{\theta}|T(X)-\widetilde \psi(\theta)|^2\\
    &\qquad \ge  \left[\frac{\gamma^\beta \int \widetilde{C}_2(t)\, dQ}{2\left\{2 - 2 \exp\left(-\overline{C}_{1}\gamma^\alpha/2\right)\right\}^{1/2}} -\gamma^\beta\left(\int \widetilde{C}^2_{2}(t) \, dQ
    \right)^{1/2}\right]^2_+.
\end{align*}

Choosing $\gamma = (2\Delta /\overline{C}_1)^{1/\alpha}$ for some $\Delta > 0$, the lower bound becomes
\begin{align*}
  &\left[\frac{(2\Delta /\overline{C}_1)^{\beta/\alpha} \int \widetilde{C}_2(t)\, dQ}{2\left\{2 - 2 \exp\left(-\Delta\right)\right\}^{1/2}} -(2\Delta /\overline{C}_1)^{\beta/\alpha}\left(\int \widetilde{C}^2_{2}(t) \, dQ
    \right)^{1/2}\right]^2_+ \\
    &\qquad = (2\Delta /\overline{C}_1)^{2\beta/\alpha}\left[\frac{\int \widetilde{C}_2(t)\, dQ}{2\left\{2 - 2 \exp\left(-\Delta\right)\right\}^{1/2}}-\left(\int \widetilde{C}^2_{2}(t) \, dQ
    \right)^{1/2}\right]_+^2.
\end{align*}
The final display is strictly greater than zero when 
\begin{align*}
    &\frac{\int \widetilde{C}_2(t)\, dQ}{2\left\{2 - 2 \exp\left(-\delta\right)\right\}^{1/2}}-\left(\int \widetilde{C}^2_{2}(t) \, dQ
    \right)^{1/2} > 0 \Longrightarrow \frac{\int \widetilde{C}_2(t)\, dQ}{\left(\int \widetilde{C}^2_{2}(t) \, dQ
    \right)^{1/2}} > 2\left\{2 - 2 \exp\left(-\Delta\right)\right\}^{1/2}.
\end{align*}
Hence, it suffices to show that there exists $Q$ such that the 
\[\int \widetilde{C}_2(t)\, dQ/\left(\int \widetilde{C}^2_{2}(t) \, dQ
    \right)^{1/2} > \varepsilon\]
satisfies for any $\varepsilon > 0$. By the assumption that $\overline{C}_2$ is strictly positive and it is continuous in $t$ around $\overline{C}_2$, we have $\int \widetilde{C}_2(t)\, dQ > \underbar{c}_2$ and $\int \widetilde{C}^2_{2}(t) \, dQ \le \|C^2_2\|_\infty< \infty$ by the H\"{o}lder inequality. Finally, selecting $\Delta$ small enough such that 
\begin{align*}
\underbar{c}_2/\overline{C}_2^2 > 2\left\{2 - 2 \exp\left(-\Delta\right)\right\}^{1/2},
\end{align*}
the lower bound becomes strictly positive. 
\newpage
\section{Proof of Lemma~\ref{lemma:density-lam}}\label{supp:density-estimation}
Let $X_1, \ldots, X_n$ be drawn IID from the unknown density $f_0$. In this section, we develop a local asymptotic minimax lower bound for the density at $X=x_0$, which is a non-differentiable functional. We consider the approximation functional via convolution such that 
\begin{align}
	\psi(f) := f(x_0)\quad \textrm{and}\quad \phi(f) := \int h^{-1}K\left(\frac{x-x_0}{h}\right) \, f(x)\, dx\nonumber
\end{align}
where $K$ is a function, satisfying \ref{as:kernel}, and $h>0$ is a bandwidth parameter. We assume that $f_0$ is $s$-times continuously differentiable at $x_0$. We consider the following set of density functions:
\begin{align}
	\mathcal{F}(s, x_0, M) := \left\{f \textrm{ is $s$-times differentiable at $x_0$, satisfying } 
	|f^{(s)}(x_0)| \le (1+M)|f^{(s)}_0(x_0)|\nonumber
\right\}
\end{align}
Furthermore, we define the following localized set of density functions 
\begin{align}
	U(\delta; \varepsilon) := \left\{f\in\mathcal{F}(s, x_0, M)\,:\, \int_{|x-x_0|\le \varepsilon}\, |f^{(k)}(x)-f^{(k)}_0(x)|\, dx\le \delta\quad \textrm{for all } k \in \{0, 1, \ldots, s\}\right\}\nonumber.
\end{align}

\begin{proof}[\bfseries{Proof of Lemma~\ref{lemma:density-lam}}]Throughout the proof, $\varepsilon$ and $M$ are considered fixed. With the choice of a kernel function, satisfying~\ref{as:kernel}, the approximation functional $\phi$ is absolutely continuous with respect to $f$ and thus Theorem~\ref{cor:vt-approximation} applies. It then implies for any approximating functional $\phi$,
\begin{align}\label{eq:lowerbound-density}
&\inf_T\, \sup_{f \in U(\delta; \varepsilon)}\,\mathbb{E}_{f}|T(X)-\psi(f)|^2 \ge \sup_{Q}\, \left[\sqrt{\frac{|\int_{\Theta} \phi'(f)\,dQ|^2}{\mathcal{I}(Q)+\int_{\Theta}\mathcal{I}(f)\, dQ}} - \|\psi-\phi\|_{L_2(Q)}\right]_+^2.
\end{align}
where the probability measure $Q$ is placed over the collection of $f$, which we formalize by constructing a differentiable parametric submodel. Since the current result holds for any choice of $\phi$, and thus for any $h > 0$, the lower bound still holds by restricting ourselves to the choice of $h \in (0, \varepsilon \wedge 1)$. Similarly, the lower bound is agnostic to the choice of $K$. Thus we focus on the $K$ whose support is contained in $[-1,1]$. Throughout this section, we denote by $\mathfrak{C}$ a fixed constant that may vary line by line.

\paragraph{Local parameter space}
For any function $g \in L_0^2(f_0)$, that is, a mean-zero and finite variance under the density $f_0$, we define a differentiable parametric path of densities as 
\begin{align}
	f_t(x) = (1+tg(x))f_0(x)\nonumber
\end{align}
for all $x \in \mathcal{X}$ and $t \geq 0$. As Theorem~\ref{cor:vt-approximation} holds for any choice of differentiable paths, we consider the choice of $g=\Phi_0$ such that 
\begin{align}
	\Phi_0 := x\mapsto h^{-1}K\left(\frac{x-x_0}{h}\right) - \int h^{-1}K\left(\frac{x-x_0}{h}\right) f_0(x) \, dx.\nonumber
\end{align}
It is straightforward to verify $\Phi_0$ is a valid score function. In fact, $\Phi_0$ is an efficient influence function of the functional $\mathbb{E}_f[\phi]$ relative to the maximal tangent space. We then consider the particular choice of paths given by $\widetilde f_t(x) := (1+t\Phi_0(x))f_0(x)$. 

First, we ensure the choice of parameter permits $\widetilde f_t$ to belong to the class of interest for all $t$. The score function $\Phi_0$ is bounded uniformly as
\begin{align}
	\|\Phi_0\|_\infty  \le \left\|h^{-1}K\left(\frac{x-x_0}{h}\right) - \int h^{-1}K\left(\frac{x-x_0}{h}\right) f_0(x) \, dx\right\|_\infty \le 2h^{-1}\|K\|_\infty \nonumber.
\end{align}
Since $K$ is uniformly bounded, it follows that  
\begin{align}
	1+t\Phi_0(x) \ge 1-t\|\Phi_0\|_\infty \ge 1-2th^{-1}\|K\|_\infty .\nonumber
\end{align}
Therefore the induced path generates a nonnegative function (that integrates to one) so long as $t\le h/(2\|K\|_\infty)$. Next, we examine the smoothness of the induced path. Since $\Phi_0^{(k)}(x) = h^{-(k+1)}K^{(k)}\left((x-x_0)/h\right)$, it follows 
\begin{align}
	|\widetilde f_t^{(s)}(x_0)|&\le |f_0^{(s)}(x_0)|+\left|\sum_{k=0}^s {s \choose k} t\left(\frac{d^{s-k}}{dx^{s-k}}\Phi_0(x_0)\right)f_0^{(k)}(x_0)\right|\nonumber\\
&\le | f_0^{(s)}(x_0)| + \sum_{k=0}^s {s \choose k} \left(\frac{t}{h^{s-k+1}} |K^{(s-k)}(0)|\right)| f_0^{(k)}(x_0)|\nonumber\\
	&\le | f_0^{(s)}(x_0)| + t\left\{\frac{|K^{(s)}(0)|}{h^{s+1}}f_0(x_0) + \frac{\mathfrak{C}}{h^s}\right\}\nonumber
\end{align}
for some constant $\mathfrak{C} > 0$. We then define 
\begin{align}
	C_M := \left(\frac{1-h\mathfrak{C}}{|K^{(s)}(0)|f_0(x_0)}\right)M | f_0^{(s)}(x_0)|,\nonumber
\end{align} for $t \in [0, C_M h^{s+1}]$. With this choice of $t$, the density on the path satisfies $|\widetilde f_t^{(s)}(x_0)| \le (1+M)| f_0^{(s)}(x_0)|$. Next, we observe that for any $\widetilde f_t$
\begin{align}
	\int_{|x-x_0|\le \varepsilon} |\widetilde f^{(\ell)}_t(x)-f^{(\ell)}_0(x)|\, dx\nonumber
	&=  t\sum_{k=0}^\ell {\ell \choose k} \frac{1}{h^{\ell-k+1}} \int_{|x-x_0|\le \varepsilon}  \left| K^{(\ell-k)}\left(\frac{x-x_0}{h}\right)f_0^{(k)}(x)\right|\, dx\nonumber\\
	&\le   th^{-\ell} \sup_{|x-x_0|\le h}\,|f_0(x)|\int_{-1}^1 |K^{(\ell)}(u)|\,du + \mathfrak{C}h^{-(\ell-1)} \nonumber
\end{align}
where the last step follows since $h < (1\wedge \varepsilon)$. All densities on the parametric path hence belong to $U(\delta; \varepsilon)$ so long as 
\begin{align}
	&t\le \frac{\delta h^{s}}{ \sup_{|x-x_0|\le h}\,|f_0(x)|\int_{-1}^1 |K^{(s)}(u)|\,du + \mathfrak{C}h} = \delta h^{s}\Delta\quad\textrm{where}\nonumber\\&\qquad\Delta:=\left(\sup_{|x-x_0|\le h}\,|f_0(x)|\int_{-1}^1 |K^{(s)}(u)|\,du+ \mathfrak{C}h\right)^{-1}\nonumber.
\end{align} 

It remains to analyze the lower bound given by \eqref{eq:lowerbound-density} over the parametric submodel $\{\widetilde f_t : t \in \Theta_0\}$ where 
\begin{align}
	\Theta_0 := \left\{t: 0 \le t \le  (\delta h^s \Delta\wedge C_Mh^{s+1})\right\}\nonumber.
\end{align}
This suggests that we need to analyze two regimes, $\delta \lesssim h$ and $h \lesssim \delta$. For instance, the usual global minimax lower bound over $\mathcal{F}$ corresponds to taking $\delta$ as a fixed constant. 

\paragraph{Approximation bias}
Let $\mathfrak{B}:=(\delta h^s \Delta\wedge C_Mh^{s+1})$. Based on the earlier derivation, we consider a prior Q supported on $[0, \mathfrak{B}]$. Such a prior can be constructed from the following dilation:
\begin{align}
	Q(t) := \frac{2}{\mathfrak{B}}Q_0\left(\frac{t-\mathfrak{B}/2}{\mathfrak{B}/2}\right) \quad \textrm{where} \quad Q_0(t) := \cos^2(\pi t/2)I(|t|\le 1).\nonumber
\end{align}
 The choice of the cosine density is motivated by the fact that this density minimizes Fisher information among density with support over $[-1,1]$ \citep{uhrmann1995minimal}. The corresponding Fisher information is $\mathcal{I}(Q_0)=\pi^2$. First, for any $\widetilde f_t$ over $t \in \Theta_0$, we have
\begin{align}
	\|\psi-\phi\|_{L_2(Q)} &= \left(\int |\psi(\widetilde f_t)-\phi(\widetilde f_t)|^2\, dQ(t)\right)^{1/2}\nonumber\\
	&= \left(\int \left|\widetilde f_t(x_0)-\int h^{-1}K\left(\frac{x-x_0}{h}\right) \, \widetilde f_t(x)\, dx\right|^2\, dQ(t)\right)^{1/2}\nonumber\\
	&= \left(\int \left|\int_{-1}^1 K(u)\int_{x_0}^{x_0 + uh} \frac{\widetilde f_t^{(s)}(\xi)(x_0+uh-\xi)^{s-1}}{(s-1)!}\, d\xi\, du\right|^2\, dQ(t)\right)^{1/2}\nonumber\\
	&\le \left(\int \left|\int_{-1}^1 K(u)\frac{|uh|^{s-1}}{(s-1)!}\int_{x_0}^{x_0 + uh} |\widetilde f_t^{(s)}(\xi)|\, d\xi\, du\right|^2\, dQ(t)\right)^{1/2}\nonumber
\end{align}

Furthermore, it follows that 
\begin{align}
	\int_{x_0}^{x_0 + uh} |\widetilde f_t^{(s)}(\xi)|\, d\xi &= \int_{x_0}^{x_0 + uh} \left|f_0^{(s)}(\xi) +  \sum_{k=0}^s {s \choose k} \left(\frac{t}{h^{s-k+1}} K^{(s-k)}((\xi-x_0)/h)\right)\,f_0^{(k)}(\xi)\right|\, d\xi\nonumber\\
	&\le \int_{x_0}^{x_0 + uh} |f_0^{(s)}(\xi)| \, d\xi+ \int_{x_0}^{x_0 + uh} \left| \frac{t}{h^{s+1}} K^{(s)}((\xi-x_0)/h)\,f_0(\xi) + \mathfrak{C}h^{-s}\right|\, d\xi\nonumber\\
	&\le \int_{x_0}^{x_0 + uh} |f_0^{(s)}(\xi)| \, d\xi+ \frac{t}{h^{s}} \int_{0}^{u} \left| K^{(s)}(\zeta)\,f_0(x_0+\zeta h)+ \mathfrak{C}h\right|\, d\zeta\nonumber\\
	&\le |uh|\sup_{|x-x_0|\le h} |f_0^{(s)}(x)| + \frac{\mathfrak{B}}{h^{s}} \sup_{|x-x_0|\le h} |f_0(x)|\int_{0}^{u} | K^{(s)}(\zeta) + \mathfrak{C}h|\, d\zeta.\nonumber
\end{align}
 The lower order terms cancel by the fact that $\int u^k K(u)\, du=0$ for all $k \in \{1,2,\ldots,s-1\}$. Also by the definition of $\mathfrak{B}$, it follows $\mathfrak{B}/h^s = (\delta  \Delta\wedge C_Mh) \le C_M h$. This leads to  
 \begin{align}
 	&\left(\int \left|\int_{-1}^1 K(u)\frac{|uh|^{s-1}}{(s-1)!}\int_{x_0}^{x_0 + uh} |\widetilde f_t^{(s)}(\xi)|\, d\xi\, du\right|^2\, dQ(t)\right)^{1/2}\nonumber \\
 	&\qquad \le\left|\int_{-1}^1 K(u)\frac{|uh|^{s-1}}{(s-1)!}\left(|uh|\sup_{|x-x_0|\le h} |f_0^{(s)}(x)| + C_Mh \sup_{|x-x_0|\le h} |f_0(x)|\int_{0}^{u} | K^{(s)}(\zeta)+ \mathfrak{C}h|\, d\zeta\right)\, du\right|\nonumber\\
 	&\qquad \le h^s\left|\int_{-1}^1 \frac{K(u)|u|^{s}}{(s-1)!}\, du\left(\sup_{|x-x_0|\le h} |f_0^{(s)}(x)| +  C_M \int_{-1}^{1} | K^{(s)}(\zeta) |\, d\zeta\sup_{|x-x_0|\le h} |f_0(x)|+ \mathfrak{C}h\right)\right|\nonumber
 \end{align}

Thus, we conclude that 
\begin{align}\label{eq:def-B}
	&\|\psi-\phi\|_{L_2(Q)} \le Bh^s \nonumber \\
	&\qquad\textrm{where}\quad B:=\left|\int_{-1}^1 \frac{K(u)|u|^{s}}{(s-1)!}\, du\left(\sup_{|x-x_0|\le h} |f_0^{(s)}(x)| +  \frac{C_M}{\Delta}+ \mathfrak{C}h \right)\right|
\end{align}
and this term no longer depends on $Q$. 

\paragraph{Surrogate efficiency}
We now turn to the numerator and the denominator from the first term of \eqref{eq:lowerbound-density}. First, by the choice of differentiable paths, the derivative with respect to $t$ is given by $\Phi_0(x)f_0(x)$. This does not depend on $t$ as we constructed a linear path. Hence the numerator of \eqref{eq:lowerbound-density} is given by 
\begin{align}
	\sqrt{\left|\int_{ 0 \le t <  \mathfrak{B}} \phi'(\widetilde f_t)\, dQ\right|^2} = \sqrt{\left|\int \frac{1}{h}K\left(\frac{x-x_0}{h}\right)\Phi_0(x)f_0(x)\, dx\right|^2}.\nonumber
\end{align}
This result also comes directly from the fact that the approximation functional $\phi$ is pathwise differentiable with its efficient influence function given by $\Phi_0$. We can further write out this term as 
\begin{align}
	\int \frac{1}{h}K\left(\frac{x-x_0}{h}\right)\Phi_0(x)f_0(x)\, dx&=\int \frac{1}{h^2}K^2\left(\frac{x-x_0}{h}\right)f_0(x)\, dx - \left(\int h^{-1}K\left(\frac{x-x_0}{h}\right) f_0(x) \, dx\right)^2 \nonumber\\
	&= h^{-1}\int_{-1}^1 K^{2}(u)f_0(x_0+uh)\, du - \left(\int_{-1}^1 K(u) f_0(x_0+uh) \, du\right)^2\nonumber.
\end{align} 
We introduce the notation 
\begin{align}\label{eq:def-V1}
	V_1 := \left|\int_{-1}^1 K^{2}(u)f_0(x_0+uh)\, du - h\left(\int_{-1}^1 K(u) f_0(x_0+uh) \, du\right)^2\right|,
\end{align}
and the numerator can be denoted by $h^{-1}V_1$. We now move onto the denominator. First, the Fisher information of the dilation $Q$ is simply given by 
\begin{align}\label{eq:def-gamma}
	\mathcal{I}(Q) = \frac{4\pi^2}{\mathfrak{B}^2}\quad\textrm{where}\quad\mathfrak{B}:=(\delta h^s \Delta\wedge C_Mh^{s+1}).
\end{align}
The Fisher information associated with $\widetilde f_t$ under the $n$ IID observations is given by 
\begin{align}
	&\int \left(\frac{\frac{d}{dt} \widetilde f_t^n(x_1, \ldots, x_n) }{\widetilde f_t^n(x_1, \ldots, x_n)}\right)^2\widetilde f_t^n(x_1, \ldots, x_n) \,d(x_1,\ldots, x_n)\nonumber \\
	&\qquad =\int \left(\frac{\widetilde f_t'(x_1)\widetilde f_t(x_2, \ldots, x_n) + \ldots + \widetilde f_t'(x_n)\widetilde f_t(x_1, \ldots, x_{n-1})}{\widetilde f_t^n(x_1, \ldots, x_n)}\right)^2\widetilde f_t^n(x_1, \ldots, x_n) \,d(x_1,\ldots, x_n) \nonumber\\
	&\qquad =\int \left(\sum_{i=1}^n\frac{\widetilde f_t'(x_i)}{\widetilde f_t(x_i)}\right)^2\widetilde f_t^n(x_1, \ldots, x_n) \,d(x_1,\ldots, x_n) \nonumber \\
	&\qquad = n\int  \left(\frac{\widetilde f_t'(x)}{\widetilde f_t(x)}\right)^2\widetilde f_t(x)\, dx + \sum_{i\neq j} \int \widetilde f_t'(x_i)\, dx_i\int \widetilde f_t'(x_j)\, dx_j\nonumber
\end{align}
Since the choice of the path implies $\widetilde f_t' = \Phi_0(x)f_0(x)$, which integrates to zero, the Fisher information for the path for each $t \in \Theta_0$ is given by 
\begin{align}
	\mathcal{I}(t) = n \int\frac{\Phi_0^2(x)f^2_0(x)}{(1-t\Phi_0(x))f_0(x)}\, dx = n \int\frac{\Phi_0^2(x)f_0(x)}{(1-t\Phi_0(x))}\, dx.\nonumber
\end{align}
The denominator of \eqref{eq:lowerbound-density} is thus given by 
\begin{align}
	&\mathcal{I}(Q) + n \int_{\Theta_0} \mathcal{I}(\widetilde f_t)\, dQ(t) \nonumber\\
	&\qquad = \frac{4\pi^2}{\mathfrak{B}^2} + n  \int_{0\le t\le \mathfrak{B}}\int \frac{\Phi^2_0(x)f_0(x) }{(1-t\Phi(x))}\, dx\, dQ(t)\nonumber\nonumber\\
	&\qquad=\frac{4\pi^2}{\mathfrak{B}^2}+ n  \int_{-1}^1\int \frac{\Phi^2_0(x)f_0(x) }{(1-(\mathfrak{B}/2+t\mathfrak{B}/2)\Phi_0(x))}\, dx\, dQ_0(t)\nonumber.
\end{align}
As it follows that 
\begin{align}
	&\int \frac{\Phi^2_0(x)f_0(x) }{(1-(\mathfrak{B}/2+t\mathfrak{B}/2)\Phi_0(x))}\, dx \nonumber\\
	&\qquad \le \int_{-1}^1 \frac{h^{-1} K^2(u) f_0(x_0 + uh)\, }{(1-(\mathfrak{B}/2+t\mathfrak{B}/2)\{K(u)f_0(x_0+uh)-\int K(u')f_0(x_0+u'h)\, du'\})}\, du\nonumber\\
	&\qquad \le \int_{-1}^1 \frac{h^{-1} K^2(u) f_0(x_0 + uh)\, }{(1-(\mathfrak{B}/2+t\mathfrak{B}/2)K(u)f_0(x_0 + uh))}\, du,\nonumber
\end{align}
we conclude
\begin{align}\nonumber
	\mathcal{I}(Q) + n \int_{\Theta_0} \mathcal{I}(\widetilde f_t)\, dQ(t) \le \frac{4\pi^2}{\mathfrak{B}^2}+ nh^{-1} V_2\nonumber
\end{align}
where 
\begin{align}\label{eq:def-V2}
	V_2 := \int_{-1}^1\int_{-1}^1 \frac{K^2(u) f_0(x_0 + uh)\, }{(1-(\mathfrak{B}/2+t\mathfrak{B}/2)K(u)f_0(x_0 + uh))}\, du\, dQ_0(t).
\end{align}

\paragraph{Optimization}
Putting together all intermediate results \eqref{eq:def-B}, \eqref{eq:def-V1}, \eqref{eq:def-gamma}, and \eqref{eq:def-V2} together and plugging them into  \eqref{eq:lowerbound-density}, we obtain  
\begin{align}
    \inf_{T} \, \sup_{f \in U(\delta)}\, \mathbb{E}_{f}|T(X)-\psi(f)|^2&\ge \left[\sqrt{\frac{h^{-2}V_1^2}{ \mathfrak{B}^{-2}4\pi^2+nh^{-1} V_2}} - \sqrt{h^{2s}B^2}\right]_+^2\nonumber.
\end{align}
Now we consider two cases. 
\paragraph{When $\mathfrak{B} = C_Mh^{s+1}$.} In this case, the lower bound becomes 
\begin{align}
	&\left[\sqrt{\frac{V_1^2}{C_M^{-2}h^{-2s}4\pi^2+nh V_2}} - \sqrt{h^{2s}B^2}\right]_+^2 \nonumber\\
 &\qquad = \left[\sqrt{\frac{V_1^2 h^{2s}}{C_M^{-2}4\pi^2+nh^{2s+1} V_2}} - \sqrt{h^{2s}B^2}\right]_+^2\nonumber\\
	&\qquad = \frac{h^{2s}}{C_M^{-2}4\pi^2+nh^{2s+1} V_2}\left[\sqrt{V_1^2} - \sqrt{B^2(C_M^{-2}4\pi^2+nh^{2s+1} V_2)}\right]_+^2\nonumber.
\end{align}
Recall that the present lower bound holds for any $h$, we choose an optimal choice of $h$ for the term outside of the bracket. The optimal choice here is given by 
\begin{align}
	h = \left(\frac{8sC_M^{-2}\pi^2}{nV_2}\right)^{\frac{1}{2s+1}}\nonumber,
\end{align}
and the corresponding lower bound is 
\begin{align}
	\inf_T \sup_{f\in U(\delta)} \mathbb{E}_{f}|T(X)-f(x_0)|^2
	& \ge \frac{8s^{2s/(2s+1)}}{(4+8s)}\left(C_M^{-2}\pi^2\right)^{\frac{-1}{2s+1}}\left(\frac{1}{nV_2}\right)^{\frac{2s}{2s+1}}\left[V_1 - BC_M^{-1}\pi\sqrt{(4+8s)
	}\right]_+^2 \nonumber\\
	& =\frac{8s^{2s/(2s+1)}}{(4+8s)}\pi^{-\frac{2}{2s+1}}C_M^{-\frac{4s}{2s+1}}\left(\frac{1}{nV_2}\right)^{\frac{2s}{2s+1}}\left[V_1C_M - B\pi\sqrt{(4+8s)
	}\right]_+^2. \nonumber
\end{align}
The above inequality holds for any $n \ge 1$. To derive a simpler asymptotic constant, we multiply both sides on the inequality by $n^{2s/(2s+1)}$ and consider the limit as $n\longrightarrow \infty$. First, we recall from the definition $\mathfrak{B}\longrightarrow 0$ as $n\longrightarrow \infty$ and thus 
\begin{align}
	V_1, V_2  &\longrightarrow\int K^2(u) \, du f_0(x_0) \quad \textrm{,} \quad C_M \longrightarrow \left(\frac{M |f_0^{(s)}(x_0)|}{|K^{(s)}(0)| f_0(x_0)}\right) \quad\textrm{and}\nonumber\\
	B &\longrightarrow |f_0^{(s)}(x_0)| \left|\int_{-1}^1 \frac{K(u)|u|^{s}}{(s-1)!}\, du\left(1 +  \frac{M\int_{-1}^1 |K^{(s)}(u)|\,du}{|K^{(s)}(0)|} \right)\right|\nonumber.
\end{align}
Therefore, we have
\begin{align}
	&\liminf_{n\longrightarrow \infty}\, \inf_T \sup_{f\in U(\delta)} n^{2s/(2s+1)}\mathbb{E}_{f}|T(X)-f(x_0)|^2
	\nonumber\\
	&  \qquad \ge C(s)f_0(x_0)^{2s/(2s+1)}|f_0^{(s)}(x_0)|^{2/(2s+1)}\left(\frac{|K^{(s)}(0)|}{M}\right)^{\frac{4s}{2s+1}}\left(\int K^2(u)\, du\right)^{2s/(2s+1)}\nonumber\\
	&\qquad\qquad \left[\frac{M\int K^2(u)\, du}{|K^{(s)}(0)|}-\sqrt{\pi^2(4+8s)}\left|\int_{-1}^1 \frac{K(u)|u|^{s}}{(s-1)!}\, du\left(1 +  \frac{M\int_{-1}^1 |K^{(s)}(u)|\,du}{|K^{(s)}(0)|} \right)\right|\right]_+^2\nonumber\\
	&  \qquad = C(s, M, K)f_0(x_0)^{2s/(2s+1)}|f_0^{(s)}(x_0)|^{2/(2s+1)}\nonumber
	\end{align}
where $C(s)$ is a constant only depending on $s$ and 
\begin{align}
	C(s, M, K) &:= \frac{8s^{2s/(2s+1)}}{(4+8s)}\pi^{-\frac{2}{2s+1}}\left(\frac{|K^{(s)}(0)|}{M}\right)^{\frac{4s}{2s+1}}\left(\int K^2(u)\, du\right)^{2s/(2s+1)}\nonumber\\
	&\qquad \left[\frac{M\int K^2(u)\, du}{|K^{(s)}(0)|}-\sqrt{\pi^2(4+8s)}\left|\int_{-1}^1 \frac{K(u)|u|^{s}}{(s-1)!}\, du\left(1 +  \frac{M\int_{-1}^1 |K^{(s)}(u)|\,du}{|K^{(s)}(0)|} \right)\right|\right]_+^2.\nonumber
\end{align} 
Crucially, the expression above does not depend on $f_0$. Finally, we can take $\mathfrak{B} = C_Mh^{s+1}$ when $\Delta \delta/C_M > h$, meaning that 
\begin{align}
	\delta > \frac{C_M}{\Delta}\left(\frac{8s\pi^2}{C_M^2 V_2}\right)^{\frac{1}{2s+1}} n^{-\frac{1}{2s+1}}\nonumber
\end{align}
This is certainly the case when $\delta = c_0 n^{-r}$ for any $r < \frac{1}{2s+1}$. This proves the result. 
\end{proof}

We provide the corresponding results outside of the setting of Lemma~\ref{lemma:density-lam}. First, when $r=\frac{1}{2s+1}$ and $c_0 > \frac{C_M}{\Delta}\left(\frac{8s\pi^2}{C_M^2 V_2}\right)^{\frac{1}{2s+1}}$, the above bound remains valid. Next, we consider when $\mathfrak{B} = \Delta h^s \delta$. This is the case when $\Delta \delta/C_M < h$. In this case, we consider the choice $h \downarrow \Delta\delta/C_M$. The corresponding lower bound is 

\begin{align}
	&\inf_{T} \, \sup_{f \in U(\delta)}\, \mathbb{E}_{f}|T(X)-\psi(f)|^2 \nonumber\\
	&\qquad \ge \left[\sqrt{\frac{V_1^2}{\Delta^{-2}h^{-2s}(h/\delta)^24\pi^2+nh V_2}} - \sqrt{h^{2s}B^2}\right]_+^2 \nonumber\\
	&\qquad \ge  \left[\sqrt{\frac{V_1^2(\Delta c_0n^{-r}/C_M)^{2s}}{C_M^{-2}4\pi^2+n(\Delta c_0n^{-r}/C_M)^{2s+1} V_2}} - \sqrt{(\Delta c_0n^{-r}/C_M)^{2s}B^2}\right]_+^2\nonumber\\
	&\qquad = \frac{1}{n^{1-r}}\left[\sqrt{\frac{V_1^2(\Delta c_0n^{-r}/C_M)^{2s}}{C_M^{-2}4\pi^2/(n^{1-r})+((\Delta c_0/C_M)(\Delta c_0n^{-r}/C_M)^{2s}) V_2}} - \sqrt{n^{1-r}(\Delta c_0n^{-r}/C_M)^{2s}B^2}\right]_+^2\nonumber.
\end{align}
Multiplying both sides of the expression by $n^{1-r}$ and taking the limit as $n\longrightarrow \infty$, we conclude that 
\begin{align}
	\liminf_{n\longrightarrow \infty}\, \inf_{T} \, \sup_{f \in U(\delta)}\, n^{1-r}\mathbb{E}_{f}|T(X)-\psi(f)|^2 \ge 0\nonumber
\end{align}
since the first term in the square bracket converges to a constant while the second term in the square bracket is divergent as $n^{1-r} \longrightarrow \infty$.
We can take $\mathfrak{B} = \Delta h^s \delta$ when $\Delta \delta/C_M < h$, which is the case $\delta = c_0 n^{-r}$ for any $r > \frac{1}{2s+1}$. When $r=\frac{1}{2s+1}$ and $c_0 < \frac{C_M}{\Delta}\left(\frac{8s\pi^2}{C_M^2 V_2}\right)^{\frac{1}{2s+1}}$,  we have 
\begin{align}
	\liminf_{n\longrightarrow \infty}\, \inf_{T} \, \sup_{f \in U(\delta)}\, n^{2s/(2s+1)}\mathbb{E}_{f}|T(X)-\psi(f)|^2 \ge 0\nonumber,
\end{align}
Hence this derivation leads to a trivial lower bound when $r\ge 1/(2s + 1)$.
\clearpage
\section{Proof of Lemma~\ref{lemma:max-vt}}\label{supp:ode}
In this section, we aim to derive a concrete constant for the lower bound given by Lemma~\ref{lemma:max-vt} in the main text. Suppose $n$ IID observation is drawn from $P_0$, which belongs to the local model $\{P_\theta : |\theta|<\delta\}$ for any $\delta > 0$. The functional of interest is $\psi(\theta) = \max(0, \theta)$, which is non-smooth at $\theta=0$. We apply Theorem~\ref{cor:vt-approximation} and obtain the lower bound 
\begin{align}
	\inf_T\, \sup_{|\theta| < \delta} \, \mathbb{E}_\theta |T(X)-\psi(\theta)|^2 \ge \inf_{Q\in \mathcal{Q}^\dag}\, \frac{\left(\int_{-\delta}^\delta I(t > 0) q(t)\, dt\right)^2}{\mathcal{I}(Q)+n\int_{-\delta}^\delta \mathcal{I}(t)q(t)\, dt}\nonumber
\end{align}
where $\mathcal{Q}^\dag\equiv \mathcal{Q}^\dag(-\delta,\delta)$ is the set of ``nice" probability measures on $(-\delta, \delta)$ that satisfies Definition~\ref{as:regularprior-vt}. We further define the density function $q$ as the dilation of the probability density $\nu$ such that $q(t) := \delta^{-1}\nu(t/\delta)$ for $\delta > 0$. The density $\nu$ is defined on $[-1,1]$. Then by simple change of variables with $u = t/\delta$, we obtain 
\begin{align}
	\int_{-\delta}^\delta I(t > 0) q(t)\, dt = \int_{-\delta}^\delta I(t > 0) \delta^{-1}\nu(t/\delta)\, dt =  \int_{-1}^{1} I(u > 0) \nu(u)\, du,\nonumber 
\end{align}
\begin{align}
		\int_{-\delta}^\delta n\mathcal{I}(t)q(t)\, dt = \int_{-\delta}^{\delta} n\mathcal{I}(t)\delta^{-1}\nu(t/\delta)\, dt = \int_{-1}^{1} n\mathcal{I}(\delta u)\nu(u)\, du\nonumber 
\end{align}
and $\delta^2\mathcal{I}(Q) = \mathcal{I}(\nu)$. Plugging them in, we obtain 
\begin{align}
	\inf_T\, \sup_{|\theta| < \delta} \, \mathbb{E}_\theta |T(X)-\psi(\theta)|^2 \ge \inf_{\nu \in \mathcal{V}(-1,1)}\frac{\left(\int_{-1}^{1} I(u > 0) \nu(u)\, du\right)^2}{\delta^{-2}\mathcal{I}(\nu)+n\int_{-1}^{1} \mathcal{I}(\delta u)\nu(u)\, du}\nonumber.
\end{align}

We now provide a concrete lower bound given by \eqref{eq:kepler-main}. To make progress, we bound the denominator with the largest Fisher information $\mathcal{I}(t)$ over the local model $|t|<\delta$. When $\{P_t : t \in \Theta\}$ is a location model, the Fisher information is a constant for all $t \in \Theta$. We then have that 
\begin{align}
\inf_{\nu \in \mathcal{V}(-1,1)}\, \frac{\left(\int_{-1}^{1} I(u > 0) \nu(u)\, du\right)^2}{\delta^{-2}\mathcal{I}(\nu)+n\int_{-1}^{1} \mathcal{I}(\delta u)\nu(u)\, du} &\ge \inf_{\nu \in \mathcal{V}(-1,1)}\,\frac{\left(\int_{-1}^{1} I(u > 0) \nu(u)\, du\right)^2}{\delta^{-2}\mathcal{I}(\nu)+n\int_{-1}^{1} \{\sup_{t\in[-\delta,\delta]}\mathcal{I}(t)\}\nu(u)\, du}\nonumber \\
&=\sup_{a\in[0,1]}\, \inf_{\mathcal{I}(\nu)}\,\frac{a^2}{\delta^{-2}\mathcal{I}(\nu)+\sup_{t\in[-\delta,\delta]}\, n\mathcal{I}(t)}\label{eq:van-trees-max}
\end{align}
where the infimum minimizes the Fisher information of the density $\nu$ under the constraint that $\int_{0}^1 \nu(t)\, dt=a$. In other words, we need to solve the following problem:
\begin{align}
	\inf_\nu \int_{-1}^1 \frac{(\nu'(t))^2}{\nu(t)}dt \textrm{ i.e., } \nu \textrm{ is absolutely continuous, }\int_0^1 \nu(t)\,dt=a \textrm{, and } \lim_{t\longrightarrow\pm\delta}\nu(t) = 0.\nonumber
\end{align}
The last constraint is introduced so $\nu$ satisfies Definition \ref{as:regularprior-vt} required for the van Trees inequality. Theorem 2.1 of~\cite{ernst2017minimizing} implies that the minimum Fisher information under the constraint satisfies
\[
-2\eta'(t) - \eta^2(t) = B_1I\{-1 \le t \le 0\} + B_2I\{0 < t\le 1\}.
\]
where $\eta(t) = \nu'(t)/\nu(t)$ and $B_1$, $B_2$ are constants. The general solution of this first-order non-homogeneous ordinary differential equation is given by 
\begin{align*}
    \eta(t) =\nu'(t)/\nu(t)= \begin{cases}
        \sqrt{B_1}\tan\left(\frac{\sqrt{B_1}}{2}(c_1-t)\right) & \text{where } -1 \le x \le 0\\
        \sqrt{B_2}\tan\left(\frac{\sqrt{B_2}}{2}(c_2-t)\right) & \text{where } 0 < x \le 1
    \end{cases}
\end{align*}
for constants $c_1, c_2$. The general solution $\nu$ is then given by 
\begin{align*}
    \nu(t) = \begin{cases}
        {C_1}\cos^2\left(\frac{\sqrt{B_1}}{2}(t-c_1)\right) & \text{where } -1 \le x \le 0\\
        {C_2}\cos^2\left(\frac{\sqrt{B_2}}{2}(t-c_2)\right) & \text{where } 0 < x \le 1,
    \end{cases}
\end{align*}
for constants $C_1, C_2$. In other words, combining two squared cosine functions gives the minimum Fisher information prior $\nu$. A squared cosine function is known to minimize the Fisher information over distributions supported on $[-1,1]$ \citep{uhrmann1995minimal}. In the current setting, the density $\nu$ can have support over any subset of $[-1,1]$. 

To simplify the derivation, we focus on a case when $C\equiv{C_1}={C_2}$, $B\equiv B_1=B_2$, and $c\equiv c_1=c_2$. In other words, we consider a location-scale family of squared cosine densities whose support is contained in $[-1,1]$ such that $\supp(\nu) \subseteq [-1,1]$. Although we investigated other cases, this simplification still resulted in the best constant. We denote the left end of the support as $s_{-}\le 0$ and the right end as $s_{+} \ge 0$. The width of the support is defined as $(s_{+}-s_{-})$. First, by the assumption that the density vanishes towards the boundary, we have
\begin{align}
	C\cos^2\left(\frac{\sqrt{B}}{2}(s_+-c)\right) = 0 \Longrightarrow \sqrt{B}(s_+-c) = \pi \textrm{ and}\nonumber\\
	C\cos^2\left(\frac{\sqrt{B}}{2}(s_--c)\right) = 0 \Longrightarrow \sqrt{B}(s_--c) = -\pi.\nonumber
\end{align}
Putting together, we have $\sqrt{B}(s_+ - s_-) = 2\pi$ and $c=(s_+-s_-)/2$.
Next, by the constraint of $\int_{0}^1 \nu(t)\,dt = a$, we have
\begin{align}
	C\int_{0}^{s_+}\cos^2\left(\frac{\sqrt{B}}{2}(x-c)\right) \, dt = \frac{Cs_{+}}{2} - \frac{C}{2\sqrt{B}}\sin(-c\sqrt{B})=a, \textrm{ and}
     \label{constraint:integral_right}\\
    C\int_{s_{-}}^0\cos^2\left(\frac{\sqrt{B}}{2}(x-c)\right) \, dt = -\frac{Cs_{-}}{2} + \frac{C}{2\sqrt{B}}\sin(-c\sqrt{B})=1-a.\label{constraint:integral_left}
\end{align}
Putting together, we have $C = 2/(s_+-s_-)$. 
Finally, we derive the Fisher information for this parametric family. First, for each $t \in [s_-, s_+]$, 
\begin{align}
    \frac{\{\nu'(t)\}^2}{\nu(t)} &= \frac{BC^2\cos^2\left(\frac{\sqrt{B}}{2}(t-c)\right)\sin^2\left(\frac{\sqrt{B}}{2}(t-c)\right)}{C\cos^2\left(\frac{\sqrt{B}}{2}(t-c)\right)}\nonumber= BC\sin^2\left(\frac{\sqrt{B}}{2}(t-c)\right).\nonumber
\end{align}
This implies that 
\begin{align}
	\mathcal{I}(\nu) = \int_{s_-}^{s^+}BC\sin^2\left(\frac{\sqrt{B}}{2}(t-c)\right)\, dt = \frac{BC}{2}(s_+-s_-) = \frac{4\pi^2}{(s_+-s_-)^2}.\label{eq:fisher-information-cosine}
\end{align}
Hence, the Fisher information only depends on the width of the support. It thus remains to characterize the width of a squared cosine prior under the constraint $\int_0^1\nu(t)\, dt=a$. Plugging in the reduced expressions into \eqref{constraint:integral_right} and \eqref{constraint:integral_left}, we obtain 
\begin{align}
 \frac{Cs_{+}}{2} - \frac{C}{2\sqrt{B}}\sin(-c\sqrt{B})=a &\implies \frac{s_{+}}{s_{+}-s_{-}} - \frac{\sin\left(-\pi\frac{s_{+}+s_{-}}{s_{+}-s_{-}}\right)}{2\pi}=a\textrm{, and}\nonumber\\
 -\frac{Cs_{-}}{2} + \frac{C}{2\sqrt{B}}\sin(-c\sqrt{B})=1-a &\implies -\frac{s_{-}}{s_{+}-s_{-}} + \frac{\sin\left(-\pi\frac{s_{+}+s_{-}}{s_{+}-s_{-}}\right)}{2\pi}=1-a.\nonumber
\end{align}
Putting together, we obtain 
\begin{align}
 Y_a - \frac{\sin\left(-\pi Y_a\right)}{\pi}=2a-1\, \textrm{ where }\, Y_a=\frac{s_{+}+s_{-}}{s_{+}-s_{-}}.\label{eq:kepler}
\end{align}
This expression is known as the Kepler equation and there is no closed-form solution for the inverse problem \citep{kepler}. We thus provide a computational method to estimate the smallest Fisher information given $Y_a$ for each $a\in [0,1]$. Once $Y_a$ is estimated from $a$, it follows that 
\begin{align*}
    &\frac{s_{+}+s_{-}}{s_{+}-s_{-}} = Y_a \Longleftrightarrow (1-Y_a)s_{+} = -(1+Y_a)s_{-}.
\end{align*}
This concludes that 
\begin{align}
	s_+ - s_- = \frac{2}{Y_a+1}s_+ = -\frac{2}{1-Y_a}s_-. \nonumber
\end{align}
As the minimum Fisher information is achieved by maximizing the width (see Equation~\ref{eq:fisher-information-cosine}), we take $s_+=1$ when $a > 0.5$ or $s_-=-1$ when $a \le 0.5$. Denoting by $W_a$ the width of the support of $\nu$ that minimizes the Fisher information for each $a \in [0,1]$, we conclude that 
\begin{align*}
	\inf\left\{\mathcal{I}(\nu) :\int_0^1 \nu(t)\, dt = a\right\}= 4\pi^2/W_a^2 \quad \textrm{where} \quad W_a =\begin{cases}
		2/(Y_a+1) & \textrm{ when } a > 0.5\\
		2/(1-Y_a)& \textrm{ when } a \le 0.5
	\end{cases}
\end{align*}
We also note that \eqref{eq:kepler} is an odd function and the expression for $W_a$ is an even function. Hence, we can modify them to
\begin{align}
 Y_a - \frac{\sin\left(-\pi Y_a\right)}{\pi}=|2a-1|\, \textrm{ and }\, W_a=2/(Y_a+1).\nonumber
\end{align}
\begin{figure}[h]
         \centering
       \includegraphics[width=5in]{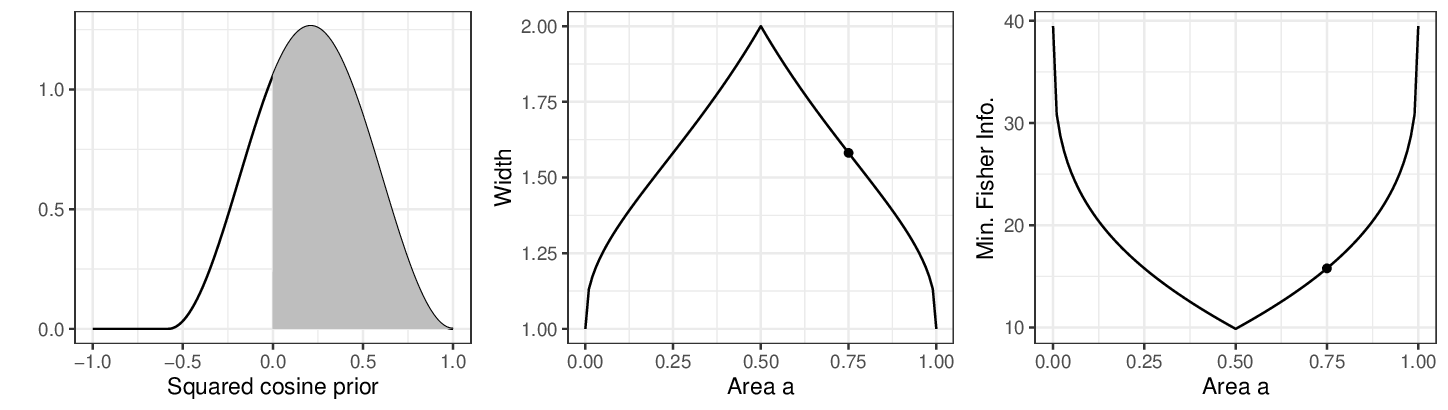}
     \caption{Visual representation of the minimum Fisher information priors for the location-scale family of squared cosine densities. The left panel displays an example of the minimum Fisher information prior when $a=0.75$. The shaded area corresponds to the constraint $\int_0^1 \nu(t)\,dt$, which is $0.75$ in this case. The center panel displays the relationship between the constraint $a$ (or the shaded area in the left panel) and the width of the support of the minimum Fisher information priors. The dot in the plot corresponds to the density with $a=0.75$ displayed in the left panel. The right panel displays the relationship between the moment constraint $a$ and the minimum Fisher information. The dot on the plot again corresponds to the density with $a=0.75$ displayed in the left panel. The minimum Fisher information is achieved by setting $a=0.5$; however, this is not necessarily the least-favorable prior for minimax lower bounds.}\label{fig:kepler}
\end{figure}

Finally, plugging this expression back into the van Trees inequality given by \eqref{eq:van-trees-max}, we conclude 
\begin{align}
	\inf_T\, \sup_{|\theta| < \delta} \, \mathbb{E}_\theta |T(X)-\psi(\theta)|^2 \ge \sup_{a\in[0,1]}\, \frac{a^2}{4\pi^2\delta^{-2}W_a^2+n \sup_{t\in[-\delta,\delta]}\mathcal{I}(t)}.\nonumber
\end{align}

We note that $a=0.5$ corresponds to the minimum Fisher information prior while this may not necessarily maximize the minimax lower bound due to the second term in the denominator. Figure~\ref{fig:kepler} provides a visual representation of the squared cosine priors and the minimum Fisher information.

\clearpage
\section{Additional derivations related to examples}
In this section, we provide the remaining derivation for Examples~\ref{example2} and \ref{example3} from the main text.
\subsection{Proof of Lemma~\ref{lemma:max-power}}
Suppose $X_1, \ldots, X_n$ is drawn from $P_{\theta_0}$ that belongs to a local model $\{P_\theta : |\theta| < \delta, \theta\in\mathbb{R}\}$ for fixed $\delta > 0$. We assume that this model is Hellinger differentiable with the Fisher information $\mathcal{I}(t)$ for $t \in (-\delta, \delta)$. Let $\psi(P_\theta) = \max(0, \theta^\alpha)$ for $0 < \alpha \le 1$.

For fixed $\delta > 0$, we consider the diffeomorphism for the local univariate parameter set $\{\theta : |\theta-\theta_0| < \delta\}$ defined as $\varphi(t) := \theta_0 + \delta \varphi_0(t)$ where $\varphi_0 : \mathbb{R} \mapsto (-1,1)$. As Theorem~\ref{thm:crvtBound_hellinger} holds for any choice of diffeomorphism, we further assume that $\varphi_0(0)=0$, $\|\varphi'_0\|_\infty < C$ for some constant, and $\varphi_0$ is an increasing function. The derivative of this mapping is $\varphi'(t) = \delta\varphi_0'(t)$. We now apply Theorem~\ref{thm:crvtBound_hellinger} to the composite function $(\psi \circ \varphi) := t\mapsto \max(0, \varphi^\alpha(t))$, 
	\begin{equation}
    \begin{aligned}
    &\inf_{T}\, \sup_{|\theta-\theta_0| < \delta} \, \mathbb{E}_{\theta}|T(X)-\psi(\varphi(\theta)) |^2\\
    &\qquad\ge \limsup_{h \longrightarrow 0}\, \left[\frac{\left|\int_{\mathbb{R}} \left(\psi(\varphi(t))-\psi(\varphi(t-h))\right)\,dQ(t)\right|}{2H\left(\widetilde{\mathbb{P}}_0, \widetilde{\mathbb{P}}_h\right)} -\left(\int_{\mathbb{R}} \left|\psi(\varphi(t))-\psi(\varphi(t-h))\right|^2 \, dQ(t)\right)^{1/2}\right]^2_+.\nonumber
    \end{aligned}
\end{equation}
By Lemma~\ref{lemma:hellinger-decomposition} in Supplementary Material, we have
\begin{align*}
    H^2(\widetilde{\mathbb{P}}_0, \widetilde{\mathbb{P}}_h) 
    &= \frac{h^2}{4}\left(\mathcal{I}(Q) +\delta^2\int_{\mathbb{R}}  \{\varphi_0'(t)\}^2\, n\mathcal{I}(\theta_0 + \delta\varphi_0(t))\, dQ(t) \right) + o(h^2).
\end{align*}
We now analyze the behavior of $\psi(\varphi(t))-\psi(\varphi(t-h))$ as $|h|\longrightarrow 0$ for different values of $\theta_0$. As the functional $\psi$ is directionally differentiable, we need to consider $h \longrightarrow 0$ from left and right.

\paragraph{Case 1 $\theta_0 < 0$:} When $h \longrightarrow 0$ from right, or $h \downarrow 0$, we have $I\{t:\theta_0 + \delta\varphi_0(t-h)> 0\} \Longrightarrow I\{t:\theta_0 + \delta\varphi_0(t)> 0\}$ by the monotonicity of $\varphi_0$. We also observe that 
\begin{align}
	\theta_0 + \delta\varphi_0(t-h)> 0 \Longleftrightarrow \theta_0 >-\delta\varphi_0(t-h) \Longleftrightarrow |\theta_0|/  \delta <\varphi_0(t-h). \nonumber
\end{align}
We then have 
\begin{align*}
    \psi(\varphi(t)) - \psi(\varphi(t-h)) &= \psi(\varphi(t)) I\{t:|\theta_0|/  \delta <\varphi_0(t)\} - \psi(\varphi(t-h)) I\{t:|\theta_0|/  \delta <\varphi_0(t-h)\} \\
    &= \left\{\varphi(t)^\alpha-\varphi(t-h)^\alpha\right\} I\{t:|\theta_0|/  \delta <\varphi_0(t-h)\} \\
    &\qquad + \varphi(t)^\alpha I\{t:\varphi_0(t-h)\le |\theta_0|/  \delta < \varphi_0(t)\}.
    \end{align*}
We denote by $\lambda_1, \lambda_2$ constants $0 \le \lambda_1, \lambda_2 \le 1$, possibly depening on $t$. By the mean value theorem,
\begin{align*}
    \varphi(t)^\alpha-\varphi(t-h)^\alpha
    &= \alpha \{(1-\lambda_1)\varphi(t)+\lambda_1\varphi(t-h)\}^{\alpha-1}\varphi'(t -\lambda_2 h)h\\
    &= \delta \alpha [\theta_0 +\delta\{(1-\lambda_1)\varphi_0(t)+\lambda_1\varphi_0(t-h)\}]^{\alpha-1}\varphi'_0(t-\lambda_2 h)h.
\end{align*}
Therefore, as $h \downarrow 0$, we have 
\begin{align*}
    \varphi(t)^\alpha-\varphi(t-h)^\alpha = \delta\alpha \{\theta_0 +\delta\varphi_0(t)\}^{\alpha-1}\varphi'_0(t)h + o(h),
\end{align*}
which follows by the continuity of $\varphi_0$ and $\varphi'_0$. Similarly as $h \downarrow 0$, we have
\begin{align}
	\varphi(t)^\alpha I\{t:\varphi_0(t-h)\le |\theta_0|/  \delta < \varphi_0(t)\} &= \varphi(t)^\alpha I\{t:\varphi_0(t)-h\varphi'_0(t_*)\le |\theta_0|/  \delta < \varphi_0(t)\}\nonumber\\
 &= \varphi(t)^\alpha I\{t:-h\varphi'_0(t_*)\le |\theta_0|/  \delta < 0\}\nonumber
\end{align}
for some $t_* \in [t-h, t]$. For any fixed $\theta_0, \varphi_0$ and $\delta$, this term is zero for $h$ small enough. Putting together, we conclude that, as $h \downarrow 0$,
\begin{align}
	\psi(\varphi(t)) - \psi(\varphi(t-h)) = \delta\alpha \{\theta_0 +\delta\varphi_0(t)\}^{\alpha-1}\varphi'_0(t)h\, I\{t:|\theta_0|/  \delta <\varphi_0(t)\} + o(h).\nonumber
\end{align}
We repeat the analysis for $h \longrightarrow 0$ from left, or $h \uparrow 0$. The analysis is identical except now we have $I\{t:\theta_0 + \delta\varphi_0(t)> 0\} \Longrightarrow I\{t:\theta_0 + \delta\varphi_0(t-h)> 0\}$. All results remain the same for $h \uparrow 0$ as well. 

Finally, by plugging each term into the lower bound given by Theorem~\ref{thm:crvtBound_hellinger} and taking the limit in view of the dominated convergence theorem, we conclude that 
\begin{align}
	&\inf_{T}\, \sup_{|\theta-\theta_0| < \delta} \, \mathbb{E}_{\theta}|T(X)-\psi(\varphi(\theta)) |^2\nonumber\\
    &\qquad\ge \limsup_{h \longrightarrow 0}\, \left[\frac{\left|\int_{\mathbb{R}} \left(\psi(\varphi(t))-\psi(\varphi(t-h))\right)\,dQ(t)\right|}{2H\left(\widetilde{\mathbb{P}}_0, \widetilde{\mathbb{P}}_h\right)} -\left(\int_{\mathbb{R}} \left|\psi(\varphi(t))-\psi(\varphi(t-h))\right|^2 \, dQ(t)\right)^{1/2}\right]^2_+.\nonumber\\
    &\qquad = \frac{\delta^2\alpha^2 |\int_{\mathbb{R}} \{\theta_0 +\delta\varphi_0(t)\}^{\alpha-1}\varphi'_0(t)\, I\{t: |\theta_0|/  \delta <\varphi_0(t)\}\, dQ(t)|^2}{\mathcal{I}(Q) +\delta^2\int_{\mathbb{R}}  \{\varphi_0'(t)\}^2\, n\mathcal{I}(\theta_0 + \delta\varphi_0(t))\, dQ(t) }.\nonumber
    \end{align}
As the upper bound does not involve $Q, \varphi_0$, we take the supremum over them. When $|\theta_0| > \delta$, the indicator in the numerator evaluates zero for all values of $t \in \mathbb{R}$. The local minimax lower bound becomes zero in such cases.

\paragraph{Case 2 $\theta_0 > 0$:}
Similar to the case with $\theta_0 < 0$, we consider $h\longrightarrow 0$ from two directions. The results are analogous thus we only provide the argument for $h \downarrow 0$. When $h \downarrow 0$, we have $I\{t:\theta_0 - \delta\varphi_0(t) > 0\} \Longrightarrow I\{t:\theta_0 - \delta\varphi_0(t-h) > 0\}$. We then have 
\begin{align*}
    \psi(\varphi(t)) - \psi(\varphi(t-h)) &= \left\{\varphi(t)^\alpha-\varphi(t-h)^\alpha\right\} I\{t:\varphi_0(t) < \theta_0/  \delta \} \\
    &\qquad - \varphi(t-h)^\alpha I\{t:\varphi_0(t-h)< \theta_0/  \delta \le \varphi_0(t)\}.
    \end{align*}
By the analogous argument from $\theta_0 < 0$, we can show that as $h \downarrow 0$,
\begin{align}
	\psi(\varphi(t)) - \psi(\varphi(t-h)) = \delta\alpha \{\theta_0 +\delta\varphi_0(t)\}^{\alpha-1}\varphi'_0(t)h\, I\{t:\varphi_0(t) < \theta_0/  \delta \} + o(h).\nonumber
\end{align}
In view of the dominated convergence theorem, we conclude that 
\begin{align}
	&\inf_{T}\, \sup_{|\theta-\theta_0| < \delta} \, \mathbb{E}_{\theta}|T(X)-\psi(\varphi(\theta)) |^2\ge \frac{\delta^2\alpha^2 |\int_{\mathbb{R}} \{\theta_0 +\delta\varphi_0(t)\}^{\alpha-1}\varphi'_0(t)\, I\{t: \varphi_0(t) < \theta_0/  \delta\}\, dQ(t)|^2}{\mathcal{I}(Q) +\delta^2\int_{\mathbb{R}}  \{\varphi_0'(t)\}^2\, n\mathcal{I}(\theta_0 + \delta\varphi_0(t))\, dQ(t) }.\nonumber
 \end{align}
\paragraph{Case 3 $\theta_0 = 0$:} The analysis is also similar to the previous two cases. After the analogous derivation, we arrive at
\begin{align}
	&\inf_{T}\, \sup_{|\theta-\theta_0| < \delta} \, \mathbb{E}_{\theta}|T(X)-\psi(\varphi(\theta)) |^2\ge \frac{\delta^2\alpha^2 |\int_{\mathbb{R}} \{\theta_0 +\delta\varphi_0(t)\}^{\alpha-1}\varphi'_0(t)\, I\{t: \varphi_0(t) < \theta_0/  \delta\}\, dQ(t)|^2}{\mathcal{I}(Q) +\delta^2\int_{\mathbb{R}}  \{\varphi_0'(t)\}^2\, n\mathcal{I}(\theta_0 + \delta\varphi_0(t))\, dQ(t) }.\nonumber
 \end{align}
 By plugging in $\theta_0=0$, we obtain that 
\begin{align}
	&\inf_{T}\, \sup_{|\theta-\theta_0| < \delta} \, \mathbb{E}_{\theta}|T(X)-\psi(\varphi(\theta)) |^2\ge \frac{\delta^{2\alpha}\alpha^2 |\int_{\mathbb{R}} \{\varphi_0(t)\}^{\alpha-1}\varphi'_0(t)\, I\{t: \varphi_0(t) > 0\}\, dQ(t)|^2}{\mathcal{I}(Q) +\delta^2\int_{\mathbb{R}}  \{\varphi_0'(t)\}^2\, n\mathcal{I}(\delta \varphi_0(t))\, dQ(t) }.\nonumber
 \end{align}

\subsection{Proof of Lemma~\ref{lemma:uniform}}
Suppose $X_1, \dots, X_n$ is an IID observation from $P_{\theta_0}$, which belongs to a model $\{\text{Unif}(0, \theta) \,: \,  0 < \theta\}$.
The local parameter set we consider is given by $\Theta = (\theta_0 - cn^{-1}, \theta_0+cn^{-1})$, we define the diffeomorephism in the form of $\varphi(t) := \theta_0 + cn^{-1} \varphi_0(t)$ where $\varphi_0 : \mathbb{R} \mapsto (-1,1)$ is invertible and differentiable increasing function. Based on the analogous application of Theorem~\ref{thm:crvtBound_hellinger} as in Example~\ref{example2}, we obtain, for any $h \in \mathbb{R}$, 
\begin{equation}
    \begin{aligned}
    \inf_{T}\, \sup_{|\theta-\theta_0| < cn^{-1}} \, \mathbb{E}_{P^n_{\theta}}|T(X)-\theta |^2
    & =\inf_{T}\, \sup_{t \in \mathbb{R}}\,  \mathbb{E}_{P^n_{\varphi(t)}}|T(X)-\varphi(t) |^2\\
    &\ge  \left[\frac{\left|\int_{\mathbb{R}} \left(\varphi(t)-\varphi(t-h)\right)\, dQ(t)\right|}{2H\left(\widetilde{\mathbb{P}}^n_0, \widetilde{\mathbb{P}}^n_h\right)} -\left(\int_{\mathbb{R}} \left|\varphi(t)-\varphi(t-h)\right|^2\, dQ(t)\right)^{1/2}\right]^2_+ 
    \end{aligned}
\end{equation}
For each $c$, we assume that $n$ is a large enough constant such that $\theta_0 - cn^{-1} > 0$. By the mean value theorem, we have 
\[
\varphi(t) - \varphi(t-h) = cn^{-1}(\varphi_0(t) - \varphi_0(t - h)) = cn^{-1}h\varphi_0'(t+\lambda h)
\]
for some $\lambda \in [0,1]$. The derivation thus focuses primarily on the Hellinger distance between $\mathrm{Unif}(0, \varphi(t))$ and $\mathrm{Unif}(0, \varphi(t+h))$. First, the Hellinger distance associated with one observation is given by
\begin{align*}
    2-2\int dP^{1/2}_{\varphi(t)}\, dP^{1/2}_{\varphi(t+h)} &= 2-2\frac{\varphi(t)}{\varphi^{1/2}(t)\varphi^{1/2}(t+h)}\\
    &= 2-2\frac{\varphi^{1/2}(t)+\varphi^{1/2}(t+h)-\varphi^{1/2}(t+h)}{\varphi^{1/2}(t+h)}\\
    &=2\frac{\varphi^{1/2}(t+h)-\varphi^{1/2}(t)}{\varphi^{1/2}(t+h)} \\
    &= \frac{\{\lambda_1\varphi(t)+(1-\lambda_1)\varphi(t+h)\}^{-1/2}\varphi'(t+\lambda_2 h)h}{\varphi^{1/2}(t+h)} \\
    &= \frac{c\varphi'_0(t+\lambda_2 h)h}{n\theta_0} + o(1/n).
\end{align*}
Furthermore, by the tensorization property of the Hellinger distance, we have
\begin{align*}
    \liminf_{n\longrightarrow \infty}\,H^2\left(P^n_{\varphi(t+h)}, P^n_{\varphi(t)}\right) &= \liminf_{n\longrightarrow \infty}\,\left\{2-2 \left(1- \frac{H^2\left(P_{\varphi(t)}, P_{\varphi(t+h)}\right)}{2}\right)^n\right\}\\
    &= 2-2 \exp \left(- \frac{c\varphi'_0( t + \lambda_2 h)h}{2\theta_0}\right).
\end{align*}
The last step follows by $(1-x/n)^n \longrightarrow \ \exp(-x)$ as $n \longrightarrow \infty$. Therefore we have 
\begin{align*}
    \liminf_{n\longrightarrow \infty}\, H^2\left(\widetilde{\mathbb{P}}^n_0,  \widetilde{\mathbb{P}}^n_{h}\right) 
    & =\liminf_{n\longrightarrow \infty}\left\{\,H^2(Q_{h}, Q) +\int H^2\left(P^n_{\varphi(t+h)}, P^n_{\varphi(t)}\right)\,dQ^{1/2}(t+h)\,dQ^{1/2}(t)\right\} \\
    & =2 -2\int\exp \left(- \frac{c\varphi'_0( t + \lambda h)h}{2\theta_0}\right)\,dQ^{1/2}(t+h)\,dQ^{1/2}(t). 
\end{align*}

Putting together, we have 
\begin{align*}
    &\liminf_{n \longrightarrow \infty}\,\inf_{T}\, \sup_{|\theta-\theta_0| < cn^{-1}} \, n^2\mathbb{E}_{P^n_{\theta}}|T(X)-\theta |^2\\
    &\qquad \ge  \bigg[\frac{ch\left|\int \left(\varphi_0'(t+\lambda h)\right)\, dQ(t)\right|}{2\left(2 -2\int  \exp \left(- \frac{c\varphi'_0( t + \lambda h)h}{2\theta_0}\right)\,dQ^{1/2}(t+h)\,dQ^{1/2}(t)\right)^{1/2}} \\
    &\qquad\qquad -ch\left(\int \{\varphi_0'(t+\lambda h)\}^2\, dQ(t)\right)^{1/2}\bigg]^2_+.
\end{align*}
Since the above inequality holds for any $h$, we denote by $h = (\theta_0/c)\eta $ and it still holds for any $\eta$. Under this parameterization, we have
\begin{align*}
    &\liminf_{n \longrightarrow \infty}\,\inf_{T}\, \sup_{|\theta-\theta_0| < cn^{-1}} \, n^2\mathbb{E}_{P^n_{\theta}}|T(X)-\theta |^2\\
    &\qquad \ge  \bigg[\frac{\theta_0\eta \left|\int \left(\varphi_0'(t+\lambda (\theta_0/c)\eta )\right)\, dQ(t)\right|}{2\left(2 -2\int  \exp \left(- \frac{\varphi'_0( t + \lambda (\theta_0/c)\eta )\eta }{2}\right)\,dQ^{1/2}(t+(\theta_0/c)\eta )\,dQ^{1/2}(t)\right)^{1/2}} \\
    &\qquad\qquad -\theta_0\eta \left(\int \{\varphi_0'(t+\lambda (\theta_0/c)\eta )\}^2\, dQ(t)\right)^{1/2}\bigg]^2_+
\end{align*}
for any $\eta$. Now, taking $c \longrightarrow \infty$, we have 
\begin{align*}
        &\liminf_{c \longrightarrow \infty}\,\liminf_{n \longrightarrow \infty}\,\inf_{T}\, \sup_{|\theta-\theta_0| < cn^{-1}} \, n^2\mathbb{E}_{P^n_{\theta}}|T(X)-\theta |^2\\
    &\qquad \ge  \sup_{\eta \in\mathbb{R},\, Q,\, \varphi_0}\, \theta_0^2\left[\frac{\eta \left|\int \varphi_0'(t)\, dQ(t)\right|}{2\left(2 -2\int  \exp \left(- \varphi'_0( t)\eta /2\right)\, dQ(t)\right)^{1/2}} -\eta \left(\int \{\varphi_0'(t)\}^2\, dQ(t)\right)^{1/2}\right]^2_+.
\end{align*}
The constant can be made strictly positive for certain choices of $Q$ by the analogous argument as the proof of Lemma~\ref{lemma:alpha-beta}. This concludes the claim.
\subsection{On the constant for Example 4}
As discussed in the main text, the optimal constant for the lower bound is known and is 1. However, the current expression involves the choice of $Q$ and $\varphi_0$, and it is unclear whether the optimal constant can be obtained by certain choices of $Q$ and $\varphi_0$. While we do not pursue this to the full extent, we provide some initial attempts in the following propositions. 
\begin{proposition}\label{prop:unif-constant1}
Consider a sequence of diffeomorphism $\varphi_0(t;\eta)$, which is indexed by $\eta > 0$. Assuming that $\eta\varphi_0'(t;\eta) \longrightarrow C$ as $\eta \longrightarrow \infty$ for some finite constant, we obtain
\begin{align*}
    &\liminf_{c\longrightarrow \infty}\, \liminf_{n\longrightarrow \infty}\, \inf_{T}\,\sup_{|\theta -\theta_0|< cn^{-1}}\mathbb{E}_{P_\theta^n}n^2|T(X)-\theta|^2 \ge C^* \theta_0^2 
\end{align*}
where $C^* \approx 0.0635^2$.
\end{proposition}
\begin{proof}[\bfseries{Proof of Proposition~\ref{prop:unif-constant1}}]
Consider the following sequence of diffeomorphisms indexed by $\eta$ such that
\[\eta \varphi'_0(t; \eta) \longrightarrow C\]
as $\eta \longrightarrow \infty$. This is satisfied, for instance, by $2/\pi \arctan(t/\eta)$. Then the lower bound can be simplified as 
    \begin{align*}
        &\sup_{\eta \in\mathbb{R},\, Q,\, \varphi_0}\, \left[\frac{\eta \left|\int \varphi_0'(t)\, dQ(t)\right|}{2\left(2 -2\int  \exp \left(- \varphi'_0( t)\eta /2\right)\, dQ(t)\right)^{1/2}} -\eta \left(\int \{\varphi_0'(t)\}^2\, dQ(t)\right)^{1/2}\right]^2_+\\
        &\qquad \ge \sup_{C}\, \left[\frac{C}{2\left(2 -2  \exp \left(- C/2\right)\right)^{1/2}} -C\right]^2_+ \approx 0.0635^2
\end{align*}
by optimizing for $C$.
\end{proof}
\begin{proposition}\label{prop:unif-constant2}
Under the settings as Example~\ref{example3}, Lemma~\ref{lemma:sharper_ihbound} provides the following local asymptotic minimax lower bound:
\begin{align*}
    &\liminf_{c\longrightarrow \infty}\, \liminf_{n\longrightarrow \infty}\, \inf_{T}\,\sup_{|\theta -\theta_0|< cn^{-1}}\mathbb{E}_{P_\theta^n}n^2|T(X)-\theta|^2 \ge C^* \theta_0^2 
\end{align*}
where $C^* \approx 0.0558$.
\end{proposition}
\begin{proof}[\bfseries{Proof of Proposition~\ref{prop:unif-constant2}}]
We now apply Lemma~\ref{lemma:sharper_ihbound} to derive the optimal constant for the lower bound. Consider two points in parameter space $\theta_0$ and $\theta_0 + hn^{-1}$, and Lemma~\ref{lemma:sharper_ihbound} implies 
\begin{align}
    &\liminf_{n\longrightarrow\infty}\,\sup_{|h| < c}\, n^2\mathbb{E}_{\theta_0 + hn^{-1/2}}\big|T(X) - \psi(\theta_0 + hn^{-1/2})\big|^2 \ge \liminf_{n\longrightarrow\infty}\,\left[\frac{1 -H^2(P^n_{\theta_0 + hn^{-1/2}}, P^n_{\theta_0})}{4}\right]_+h^2
\end{align}
for all $|h| < c$. An analogous derivation to Example 3 will yield 
\begin{align*}
    \liminf_{n\longrightarrow \infty}\,H^2\left(P^n_{\theta_0+hn^{-1}}, P^n_{\theta_0}\right) 
    &= 2-2 \exp \left(- \frac{h}{2\theta_0}\right)
\end{align*}
for $n$ IID observations from two uniform distributions. Putting together, it follows that
\begin{align}
    \liminf_{c\longrightarrow \infty}\,\liminf_{n\longrightarrow\infty}\,\sup_{|h| < \infty}\, n^2\mathbb{E}_{\theta_0 + hn^{-1/2}}\big|T(X) - \psi(\theta_0 + hn^{-1/2})\big|^2 &\ge \sup_{|h|<c}\,\left[-\frac{1}{4}+\frac{1}{2}\exp \left(- \frac{h}{2\theta_0}\right)\right]_+h^2 \nonumber\\
    &\ge \sup_{|\eta |<\infty}\,4\theta_0^2\left[-\frac{1}{4}+\frac{1}{2}\exp (- \eta )\right]_+\eta ^2\nonumber \\
    &\approx0.0558\,\theta_0^2\nonumber
\end{align}
where we parameterize $h = 2\eta \theta_0$ and optimize for $\eta$ to obtain the result.
\end{proof}

\clearpage
\section{Derivation of the risk of estimators}\label{supp:upper-bound}
In this section, we provide exact upper bounds and lower bounds for the truncated Gaussian mean estimation. To recall, we consider the following estimation problem:
\begin{align}
	\sup_{|\theta| < \delta}\, n\mathbb{E}_\theta |T(X) - \max(\theta, 0)|^2 \nonumber
\end{align}
where we observe $n \ge 1$ IID observations $X:=X_1, \ldots, X_n$ from $N(0, 1)$. We define a local model as $\{N(\theta,1) : |\theta| < \delta\}$ for $\delta > 0$.

\subsection{Lower bound under Gaussian prior and arctan diffeomorphism}
In Example~\ref{example2} (iii) with $\alpha=1$, we have derived that 
\begin{align}
	\inf_{T}\, \sup_{|\theta| < \delta} \, n\mathbb{E}_{\theta}|T(X)-\psi(P_\theta)  |^2 \ge\sup_{Q, \varphi_0}\, \frac{n\delta^{2} |\mathbb{E}_Q \varphi'_0(t)\, I\{t:\varphi_0(t) >0 \}|^2}{\mathcal{I}(Q) +\delta^2\mathbb{E}_Q  \{\varphi_0'(t)\}^2\, \mathcal{I}(\delta\varphi_0(t))}.\nonumber
\end{align}
The Fisher information of distributions in our local model is given by $\mathcal{I}(\theta) = n$ for all $|\theta| < \delta$ under $n$ IID observations. To simplify the derivation, we further restrict the choice of priors and diffeomorphism to be 
\begin{align}
	Q \in \{N(\mu, \sigma^2) : (\mu, \sigma)\in \mathbb{R}\times \mathbb{R}_+\}\textrm{ and } \varphi_0 \in \{t \mapsto \pi/2 \arctan(t/\eta) : \eta > 0\}.\nonumber
\end{align}
We define $\xi_1 := \mu/\eta$ and $\xi_2:=\sigma/\eta$. Since $\varphi'_0(t) = 2(\pi\eta)^{-1}(1+(t/\eta)^2)^{-1}$, we have
 \begin{align}
	\mathbb{E}_Q \{\varphi_0'(t)\}^2 &=\int \frac{4}{\pi^2\eta^2}\left(\frac{1}{1+(t/\eta)^2}\right)^2\frac{1}{\sigma\sqrt{2\pi}}\exp(-(t-\mu)^2/(2\sigma^2))\, dt\nonumber \\
	&= \frac{4}{\pi^2\eta^2}\int\left(\frac{1}{1+(\xi_1+u\xi_2)^2}\right)^2\phi(u)\, du \nonumber
\end{align}
where $\phi$ is a density function of a standard normal distribution. Similarly, 
\begin{align}
	|\mathbb{E}_Q \varphi_0'(t)I(\varphi_0(t)>0)|^2 &=\left|\int \frac{2}{\pi\eta}\left(\frac{1}{1+(t/\eta)^2}\right)\frac{1}{\sigma\sqrt{2\pi}}\exp(-(t-\mu)^2/(2\sigma^2))I(t > 0)\, dt \right|^2\nonumber\\
	&=\frac{4}{\pi^2\eta^2}\left|\int \left(\frac{1}{1+(\mu/\eta+u\sigma/\eta)^2}\right)\frac{1}{\sqrt{2\pi}}\exp(-u^2/2)I(\mu + u\sigma > 0)\, du\right|^2 \nonumber\\
	&= \frac{4}{\pi^2\eta^2}\left|\int \left(\frac{1}{1+(\xi_1+u\xi_2)^2}\right)I(\xi_1+u\xi_2 > 0)\phi(u)\, du\right|^2\nonumber.
\end{align}
Using the fact that $\mathcal{I}(Q)=\sigma^{-2}$, the lower bound becomes
\begin{align}
	\frac{n\delta^2|\mathbb{E}_Q \varphi_0'(t)I(t>0)|^2}{(1/\sigma^{2}) + \delta^2 n\mathbb{E}_Q \{\varphi_0'(t)\}^2} &= \frac{4 n\xi_2^2\left|\int \left(1+(\xi_1+u\xi_2)^2\right)^{-1}I(u >-\xi_1/\xi_2)\phi(u)\, du\right|^2}{\pi^2\delta^{-2} + 4n\int\left(1+(\xi_1+u\xi_2)^2\right)^{-2}\phi(u)\, du }\nonumber\\
	&=\frac{4 n\xi_2^2\left|\E \left[\left(1+(\xi_1+Z\xi_2)^2\right)^{-1}I(Z >-\xi_1/\xi_2)\right]\right|^2}{\pi^2\delta^{-2} + 4n\E\left[\left(1+(\xi_1+Z\xi_2)^2\right)^{-2} \right]}\nonumber
\end{align}
where $Z \overset{d}{=} N(0,1)$. We conclude the claim by optimizing over $(\xi_1, \xi_2) \in\mathbb{R}\times \mathbb{R}_+$ for given $\delta > 0$ and $n \ge 1$.

\subsection{Proof of Proposition~\ref{prop:risk-mle}}
In this section, we provide the exact local risks of the following plug-in estimators, defined as
\begin{align}
    S_n^{\text{plug-in}} &:=  \psi(\widehat \theta_{\text{MLE}}) \quad \text{where} \quad 
    \widehat \theta_{\text{MLE}} := \overline{X}_n = \frac{1}{n}\sum_{i=1}^n X_i, \quad \text{and}\nonumber\\
    S_n^{\text{pre-test}} &:=  \psi(\widehat \theta_{\text{pre-test}})\quad \text{where} \quad 
    \widehat \theta_{\text{pre-test}} := \begin{cases}
    \overline{X}_n & \text{If}\quad  |\overline{X}_n| \ge n^{-1/4}\\
    0 &\text{otherwise}.
    \end{cases}\nonumber
\end{align}
\begin{proof}[\bfseries{Proof of Proposition~\ref{prop:risk-mle}}]
	First, we consider the plug-in MLE. By splitting the risk into two cases, we have
	\begin{align}
		&\sup_{|\theta|<\delta}\, n\mathbb{E}_\theta|S_n^{\text{plug-in}}-\max(\theta,0)|^2  = \max\left\{\sup_{0 \le \theta <\delta}\, n\mathbb{E}_\theta|S_n^{\text{plug-in}}-\theta|^2, \sup_{-\delta < \theta <0}\, n\mathbb{E}_\theta|S_n^{\text{plug-in}}|^2\right\}\nonumber.
	\end{align}
	For the first case, we have
	\begin{align}
		&\sup_{0 \le \theta <\delta}\, n\mathbb{E}_\theta|S_n^{\text{plug-in}}-\theta|^2 \nonumber\\
		&\qquad = \sup_{0 \le \theta <\delta}\, n\left\{\mathbb{E}_\theta[|\overline{X}_n-\theta|^2I(\overline{X}_n \ge 0)] + \mathbb{E}_\theta[\theta^2I(\overline{X}_n < 0)]\right\} \nonumber\\
		&\qquad = \sup_{0 \le \theta <\delta}\, \left\{\mathbb{E}_\theta[|n^{1/2}(\overline{X}_n-\theta)|^2I(n^{1/2}(\overline{X}_n-\theta) \ge -n^{1/2}\theta)]+ n\theta^2\mathbb{E}_\theta[I(n^{1/2}(\overline{X}_n-\theta) < -n^{1/2}\theta)]\right\}\nonumber.
	\end{align}
	Since $Z:=n^{1/2}(\overline{X}_n-\theta) \overset{d}{=} N(0,1)$, we conclude 
	\begin{align}
		&\sup_{0 \le \theta <\delta}\, n\mathbb{E}_\theta|S_n^{\text{plug-in}}-\theta|^2= \sup_{0 \le \theta <\delta}\, \left\{\mathbb{E}_\theta[Z^2I(Z \ge -n^{1/2}\theta)]+ n\theta^2\mathbb{E}_\theta[I(Z < -n^{1/2}\theta)]\right\}\nonumber.
	\end{align}
	
Next, we claim that 
\begin{align}\label{eq:risk_inequality}
   \sup_{- \delta < \theta \le 0}\,n\mathbb{E}_\theta |S_n^{\text{plug-in}}|^2 \le  \sup_{0 \le \theta < \delta}\, \mathbb{E}_\theta |S_n^{\text{plug-in}} - \theta|^2,
\end{align}
that is, the risk is always greater when $\theta \ge 0$ than $\theta < 0$. Following an analogous argument from $\theta \ge 0$, we obtain
\begin{align*}
    \sup_{ - \delta < \theta \le 0}\, n\mathbb{E}_\theta |S_n^{\text{plug-in}}|^2 &= \sup_{ - \delta < \theta \le 0}\,n\mathbb{E}_\theta [\overline{X}_n^{2} I(\overline{X}_n \ge 0)] \\
    &= \sup_{-\delta \le \theta < 0}\,\E [(Z+n^{1/2}\theta)^2 I(Z \ge -n^{1/2}\theta)] 
\end{align*}
where $Z \overset{d}{=}N(0,1)$. Denoting by $\eta = n^{1/2}\theta$ and by $\phi$ the density function of a standard Gaussian distribution, it follows, for any $\eta \le 0$, 
\begin{align*}
    \E (Z+\eta)^{2} I(Z\ge -\eta) &= \int_{-\eta}^\infty (z+\eta)^{2}\phi(z)\, dz\le  \int_{-\eta}^\infty (z+\eta)^{2}\phi(z+\eta)\, dz=  \int_{0}^\infty \tilde{z}^{2}\phi(\tilde{z})\, d\tilde{z}.
\end{align*}
The middle inequality follows since the density $\phi(z)$ is non-increasing on $0 \le z$ and thus $\phi(z) \le \phi(z+\eta)$ for any $\eta \le 0$. The last quantity is equivalent to $\E [Z^{2} I(Z\ge 0)]$ hence the risk of the estimator for any  $\eta \le 0$ is upper bounded by the case when $\eta = 0$, or equivalently when $\theta = 0$. Therefore, equation \eqref{eq:risk_inequality} is implied as desired. This concludes the claim.

For the pre-test estimator, the derivation is almost analogous except that we have 
	\begin{align}
		\sup_{0 \le \theta <\delta}\, n\mathbb{E}_\theta|S_n^{\text{pre-test}}-\theta|^2 &= \sup_{0 \le \theta <\delta}\, n\left\{\mathbb{E}_\theta[|\overline{X}_n-\theta|^2I(\overline{X}_n \ge n^{-1/4})] + \mathbb{E}_\theta[\theta^2I(\overline{X}_n < n^{-1/4})]\right\} \nonumber\\
		& = \sup_{0 \le \theta <\delta}\, \left\{\mathbb{E}_\theta[Z^2I(Z \ge n^{1/4}-n^{1/2}\theta)]+ n\theta^2\mathbb{E}_\theta[I(Z < n^{1/4}-n^{1/2}\theta)]\right\}\nonumber.
	\end{align}
The remaining derivation is omitted.
\end{proof}

\clearpage
\section{Connection to the modulus of continuity}\label{supp:modulus}
We discuss the connections between Theorem~\ref{thm:crvtBound_hellinger} and existing minimax lower bounds in terms of the Hellinger distance. Previous works by \citet{donoho1987geometrizing, donoho1991geometrizing2} provide a geometric interpretation of the minimax rate through \textit{the modulus of continuity}. We define the modulus of continuity of a real-valued functional $\psi$ with respect to the Helligner distance as
\begin{equation*}
    \omega(\varepsilon) := \left\{\sup_{\theta_1, \theta_2\in\Theta_0}\,|\psi(P_{\theta_1})-\psi(P_{\theta_2})| \, : \, H^2(P_{\theta_1}, P_{\theta_2}) \le \varepsilon^2 \right\}.
\end{equation*}
This quantity captures the maximum fluctuation of functionals evaluated at sufficiently ``similar" distributions. We momentarily focus on estimating $\psi(P_0)$ based on $n$ IID observations $X_1, \ldots, X_n$ drawn from $P_0$. The minimax risk is considered for the supremum over a local model $\mathcal{P}_n := \{P : H^2(P_0, P) \le n^{-1}\}$, which is a set of distributions concentrated at $P_0$, analogous to the setting of the LAM theorem. Section 9.4 of \cite{donoho1987geometrizing} proves that, under additional assumptions, for any $n$ sufficiently large, \begin{align}\label{eq:hellinger_modulus}
    &\inf_{T}\, \sup_{P \in\mathcal{P}_n}\, n\mathbb{E}_{P}|T(X)-\psi(P)|^2 > \frac{1}{16}\left(\frac{\omega(n^{-1/2})}{n^{-1/2}}\right)^2.
   	\end{align}
\cite{donoho1987geometrizing} also shows that, under the setting of the LAM theorem (Theorem~\ref{theorem:LAM_parametric}) and $\psi(P_\theta) = \theta$, it holds
\begin{align}\label{eq:lam-donoho}
	 \liminf_{n\longrightarrow\infty}\,\frac{1}{16}\left(\frac{\omega(n^{-1/2})}{n^{-1/2}}\right)^2= \frac{1}{4\mathcal{I}(\theta_0)}.
\end{align}
Hence, the non-asymptotic minimax lower bound given by \eqref{eq:hellinger_modulus} converges to $1/4$ the optimal constant according to the LAM theorem. It may be tempting to multiply the lower bound of \eqref{eq:hellinger_modulus} by $4$, that is, to consider the sequence $4^{-1}\{n^{1/2}\omega(n^{-1/2})\}^2$. Although this sequence converges to the correct constant as $n \longrightarrow \infty$, it is an invalid lower bound as there exists an estimator that violates this inequality, provided in Section 8.4 of \cite{donoho1987geometrizing}. \cite{donoho1987geometrizing} conjectured that a different approach is necessary to develop a sequence of non-asymptotic minimax lower bounds that converge to the correct limit. This result is provided by Theorem~\ref{thm:crvtBound_hellinger}. Our result is similar to \cite{donoho1987geometrizing} in the sense that both define local models relative to the Hellinger distance; however, they differ as we consider the mixture of distribution over the Hellinger ball instead of two distributions. It is also worth noting that an analogous non-asymptotic result in terms of the two-point modulus of continuity is studied by \citet{chen1997general} without particular focus on a sharp constant.

Finally, similar objects as the Hellinger modulus of continuity have been considered in the econometrics literature. When the local asymptotic minimax lower bound, such as the right-hand term of Theorem~\ref{thm:crvtBound_hellinger}, is bounded away from zero, the functional estimation is called \textit{ill-posed} \citep{potscher2002lower, forchini2005ill}. This precludes the existence of a (locally) uniformly consistent estimator of $\psi(P_0)$ in the nonparametric model. \citet{potscher2002lower} and \citet{forchini2005ill} define the modulus of continuity in the total variation distance as
\begin{equation*}
    \omega(\varepsilon) := \left\{\sup_{\theta_1, \theta_2\in\Theta}|\psi(P_{\theta_1})-\psi(P_{\theta_2})| \, : \, \textrm{TV}(P_{\theta_1}, P_{\theta_2}) \le \varepsilon \right\}
\end{equation*}
where the total variation distance is defined as 
\begin{align*}
     \textrm{TV}(P_0, P_1) := \sup_{A \in \mathcal{A}}|P_0(A) - P_1(A)|,
\end{align*}
with the supremum over all Borel measurable sets $A$.
Theorem 2.1 of \citet{potscher2002lower} provides the sufficient condition under which the following holds:
\begin{align}
	\lim_{\varepsilon\longrightarrow 0}\, \inf_{T}\, \sup_{P \in\mathcal{P}(\varepsilon)}\, \mathbb{E}_{P}|T(X)-\psi(P)|^2&\ge \lim_{\varepsilon\longrightarrow 0}\, \frac{1}{4}\omega(\varepsilon)^2.\nonumber
\end{align}
where $\mathcal{P}(\varepsilon) := \{P  : \mathrm{TV}(P_0, P)\le \varepsilon\}$. \citet{potscher2002lower} and \citet{forchini2005ill} analyze concrete problems where the lower bound is bounded away from zero. We note that there is no significant loss in considering the TV- or Hellinger-moduli as both distances define the same topology of the space of probability measures. However, our result also differs from \citet{potscher2002lower} and \citet{forchini2005ill} as we focus on the non-asymptotic results as well as the optimal constant. When the asymptotic lower bound of Theorem~\ref{thm:crvtBound_hellinger} is bounded away from zero, it can be considered as the measure of \textit{ill-posedness}. We defer the corresponding analysis to future works.

\section{Additional illustration}
Section~\ref{section:estimators} provides the detail of the simulation study. Figure~\ref{fig:estimators1} below shows the same phenomena as Figure~\ref{fig:estimators} by fixing $\delta$ and varying the sample size $n$.
\setcounter{figure}{1}
\begin{figure*}[h]
    \centering
        \centering
        \includegraphics[width=5in]{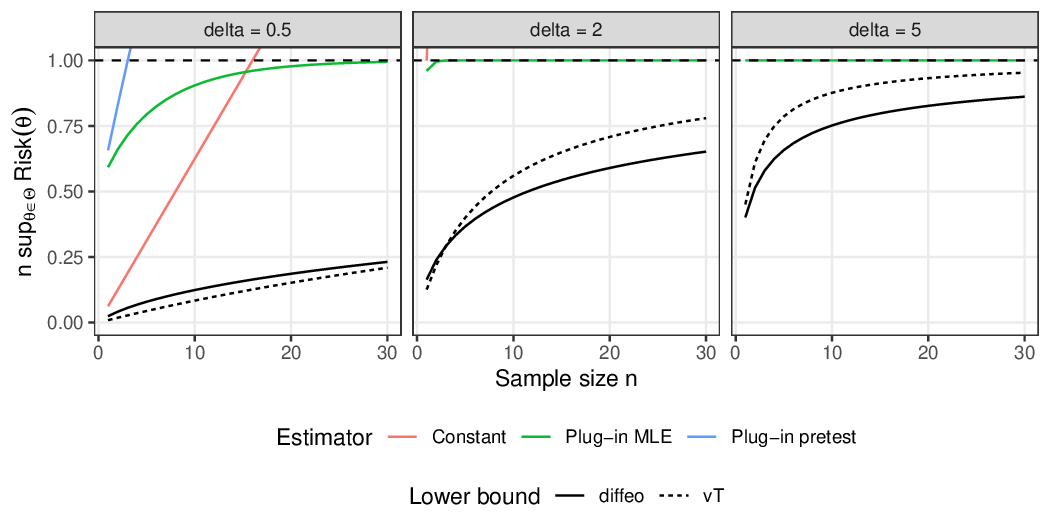}
    \caption{The non-asymptotic local minimax lower bounds and the risk given by different estimators: For fixed local neighborhoods with varying sample sizes.}
\label{fig:estimators1}
\end{figure*}

\newpage
\section{Further remarks on the present results}\label{supp:extensions}
\subsection{Non-pathwise differentiable functionals}

While we demonstrate the application of the proposed lower bounds to derive the semiparametric efficiency bound for pathwise differentiable functionals, the main results from Section~\ref{section:van_trees} still hold for general non-pathwise differentiable functionals. First, by taking $\Phi$ as a collection of pathwise differentiable functionals, Lemma~\ref{thm:L2-approximation} implies
\begin{align*}
    &\sup_{g \in \mathcal{T}_{P_0}}\, \inf_{T} \, \sup_{\theta \in \Theta}\, \mathbb{E}_{\theta}|T(X)-\psi(P_{\theta,g})|^2\\
    &\,\,\ge \sup_{g \in \mathcal{T}_{P_0}}\,\sup_{\phi \in \Phi,\, Q\in \mathcal{Q}(\Theta)}\, \left[\left(\int \,\mathbb{E}_{\theta}|T(X)-\phi(P_{\theta,g})|^2\, dQ\right)^{1/2} - \left(\int \|\psi(P_{\theta,g})-\phi(P_{\theta,g})\|^2dQ(\theta)\right)^{1/2}\right]_+^2
\end{align*}
This suggests that the efficiency bound for non-pathwise differentiable functionals can be analyzed by first placing a prior over the local parametric paths and then deriving the Bayes risk for estimating $\phi$ at the expense of the approximation error.

Alternatively, we can apply Theorem~\ref{thm:crvtBound_hellinger} to each parametric path and take the supremum over any generic tangent set $\mathcal{T}_{P_0} \subseteq L_2^0(P_0)$. This implies
\begin{align*}
    &\sup_{g \in \mathcal{T}_{P_0}}\, \inf_{T} \, \sup_{\theta \in\mathbb{R}^d}\, \mathbb{E}_{\theta}|T(X)-\psi(P_{\theta,g})|^2\\
    &\,\, \ge \sup_{g \in \mathcal{T}_{P_0}}\, \sup_{h\in\mathbb{R}^d}\left[\left(\frac{|\int_{\mathbb{R}^d} \left(\psi(P_{t,g})-\psi(P_{t-h,g}\right)\,dQ(t)|^2}{4H^2(\mathbb{P}_{0,g}, \mathbb{P}_{h,g})} \right)^{1/2}-\left(\int_{\mathbb{R}^d} |\psi(P_{t,g})-\psi(P_{t-h,g})|^2 \, dQ(t)\right)^{1/2}\right]^2_+
\end{align*}
where
 \[d\mathbb{P}_{0,g}(x, \theta) := dP_{\theta, g}(x)\, dQ(\theta)\quad\mbox{and}\quad d\mathbb{P}_{h,g}(x,\theta) := dP_{\theta+h, g}(x)\, dQ(\theta+h).\]
In the above non-asymptotic lower bound, the limiting behavior of $\psi(P_{t,g})-\psi(P_{t-h,g})$ as $h \longrightarrow 0$ is unspecified, allowing $\psi$ to be non-smooth at $t$.

\subsection{Non-convex sets}
This manuscript extends the result to a general parameter set $\Theta_0 \subseteq \Theta$ via diffeomorphism. 
Alternatively, we can use the fact that the supremum over $\theta\in\Theta_0$ can be obtained by considering any two points $\theta_0, \theta_1\in\Theta_0$ and any smooth path $\gamma  :=  \gamma_{\theta_0, \theta_1}:(-\infty, \infty)\mapsto\Theta_0$ such that $\lim_{u\to-\infty}\gamma(u) = \theta_0, \lim_{u\to\infty}\gamma(u) = \theta_1$. Let $\widetilde{\psi}(u) := \psi(\gamma(u))$ for $u\in\mathbb{R}$,
\[
d\mathbb{P}_0(x,u) := dP_{\gamma(u)}(x)dQ(u)\quad\mbox{and}\quad d\mathbb{P}_h(x,u) := dP_{\gamma(u + h)}(x)dQ(u+h),
\]
where $Q$ is a probability measure on $\mathbb{R}$.
Using this notation, we obtain
\begin{align*}
&\inf_{T}\,\sup_{\theta\in\Theta_0}\,\mathbb{E}_{\theta}\|T(X) - \psi(\theta)\|^2\\ &\quad= \inf_{T}\,\sup_{\theta_0, \theta_1}\,\sup_{\gamma}\,\sup_{u\in\mathbb{R}}\,\mathbb{E}_{\gamma(u)}\|T(X) - \widetilde{\psi}(u)\|^2\\
&\quad\ge \sup_{\theta_0, \theta_1}\,\sup_{\gamma}\,\sup_{u\in\mathbb{R}}\,\left[\left(\frac{\|\int_{\mathbb{R}} (\widetilde{\psi}(u) - \widetilde{\psi}(u - h))\, dQ(u)\|^2}{4H^2(\mathbb{P}_0, \mathbb{P}_h)}\right)^{1/2} - \left(\int_{\mathbb{R}} \|\widetilde{\psi}(u) - \widetilde{\psi}(u - h)\|^2\, dQ(u)\right)^{1/2}\right]_+^2.
\end{align*}
As a concrete example, suppose we are interested in finding a lower bound when $\Theta_0 := \{\theta:\,\rho(\theta - \theta_0) \le \delta\}$ for some positive homogeneous norm function $\rho(\cdot)$, then one can take $\theta_1 = \theta_0 + \delta h$ such that $\rho(h) \le 1$ and consider $\gamma(u) = \theta_0 + \Phi(u)\delta h$ where $\Phi(\cdot)$ is the CDF of the standard normal distribution. This approach of taking a smooth path between two arbitrary points allows for the application of the results above to non-convex, but continuous, parameter spaces $\Theta_0$.



\end{appendices}
\end{document}